\documentclass{amsart}[11pt]
\usepackage[margin=1.0in]{geometry}

\usepackage{inputenc}
\usepackage{amsmath,mathtools}
\usepackage{amssymb, amsfonts}
\usepackage[mathscr]{eucal}
\usepackage{graphicx,epstopdf,bm}
\usepackage{amsaddr}
\usepackage{wrapfig}
\usepackage{subfig}
\usepackage{placeins}

\usepackage{dcolumn}%
\usepackage[title]{appendix}
\usepackage[colorlinks,bookmarks,urlcolor=blue,citecolor=blue,linkcolor=blue]{hyperref}
\usepackage[capitalize]{cleveref}
\usepackage{algorithm}
\usepackage{algpseudocode}

\usepackage{bbm}
\usepackage{mathtools}
\usepackage{multirow}
\usepackage{xcolor}

\usepackage{algorithm}
\usepackage{algpseudocode}

\usepackage{amsthm}
\numberwithin{equation}{section}
\theoremstyle{plain}
\newtheorem{theorem}{Theorem}[section]
\newtheorem{definition}{Definition}[section]

\newtheorem{lemma}[theorem]{Lemma}
\newtheorem{assumption}[theorem]{Assumption}

\allowdisplaybreaks

\usepackage{bbm}

\renewcommand{\vec}[1]{\boldsymbol{#1}}

\newcommand{\paren}[1]{\left(#1\right)}
\newcommand{\brac}[1]{\left[#1\right]}
\newcommand{\E}{\mathbb{E}}
\newcommand{\avg}[1]{\E\brac{#1}}

\newcommand{\lap}[1]{\Delta#1}

\newcommand{\abs}[1]{\left|#1\right|}
\newcommand{\norm}[1]{\Vert#1\Vert}

\newcommand{\DA}{D^{\textrm{A}}}
\newcommand{\DB}{D^{\textrm{B}}}
\newcommand{\DC}{D^{\textrm{C}}}

\newcommand{\bi}{\vec{i}}
\newcommand{\bj}{\vec{j}}

\newcommand{\vn}{\vec{n}}
\newcommand{\vN}{\vec{N}}

\newcommand{\vP}{\vec{P}}

\newcommand{\Pn}{P^{\vec{n}}}
\newcommand{\vG}{\vec{G}}

\newcommand{\vx}{\vec{x}}
\newcommand{\vy}{\vec{y}}

\newcommand{\vq}{\vec{q}}

\newcommand{\vQ}{\vec{Q}}

\newcommand{\vqn}{\vec{q}^{\vec{n}}}
\newcommand{\vqnj}{\vec{q}^{n_j}}

\newcommand{\ind}{\mathbbm{1}}

\def\N{\mathbb{N}}
\def\R{\mathbb{R}}

\renewcommand{\epsilon}{\varepsilon}

\newcommand{\Def}{\overset{\text{def}}{=}}

\newcommand{\BP}{\mathbb{P}}

\newcommand{\filt}{\mathscr{F}}

\newcommand{\la}{\left \langle}
\newcommand{\ra}{\right\rangle}

\newcommand{\RN}[1]{
  \textup{\uppercase\expandafter{\romannumeral#1}}%
}
\DeclarePairedDelimiter\ceil{\lceil}{\rceil}

\newcommand{\cC}{\mathcal{C}}
\newcommand{\cD}{\mathcal{D}}

\newcommand{\Ol}{\Omega^{(\ell)}}
\newcommand{\Otl}{\tilde{\Omega}^{(\ell)}}
\newcommand{\given}{\,\middle|\,}
\newcommand{\vupl}{\vec{\upsilon}^{(\ell)}}
\newcommand{\Klg}{K_{\ell}^{\gamma}}
\DeclareMathOperator{\Bernoulli}{Bernoulli}
\DeclareMathOperator{\Cov}{Cov}

\makeatletter
\def\mathcolor#1#{\@mathcolor{#1}}
\def\@mathcolor#1#2#3{%
  \protect\leavevmode
  \begingroup
    \color#1{#2}#3%
  \endgroup
}
\makeatother

\title{Fluctuation analysis for particle-based stochastic reaction-diffusion models}
\author{M. Heldman, S. Isaacson, J. Ma and K. Spiliopoulos}
\thanks{Department of Mathematics and Statistics, Boston University, 111 Cummington Mall, Boston, MA 02215. Emails: heldmanm@bu.edu, isaacson@math.bu.edu, majw@bu.edu, kspiliop@math.bu.edu}
\thanks{All authors were supported by ARO W911NF-20-1-0244. MH was also supported by NSF-DMS 1902854. KS was also supported by NSD-DMS 2107856. We would like to thank the two anonymous reviewers of this article for their critical and essential review that greatly improved this work. }
\date{\today}

\begin{document}

\begin{abstract}
  Recent works have derived and proven the large-population mean-field limit for
  several classes of particle-based stochastic reaction-diffusion (PBSRD)
  models. These limits correspond to systems of partial integral-differential
  equations (PIDEs) that generalize standard mass-action reaction-diffusion PDE
  models. In this work we derive and prove the next order fluctuation
  corrections to such limits, which we show satisfy systems of stochastic PIDEs
  with Gaussian noise. Numerical examples are presented to illustrate how
  including the fluctuation corrections can enable the accurate estimation of
  higher order statistics of the underlying PBSRD model.
\end{abstract}

\maketitle

\section{Introduction}
In this work we consider the large-population limit of particle-based stochastic reaction-diffusion (PBSRD) models. Such models are used to study the interplay between spatial transport via diffusion and chemical reactions for systems of interacting particles/molecules, and have been used in a variety of biological and population process studies. Several recent works~\cite{Nolen2019,IMS:2022,PopovicVeber22} have investigated the large-population (thermodynamic) limit of such models, first studied in~\cite{Oelschlager89a}, where the number of particles and a system size parameter (typically domain volume and/or Avogadro's number) are allowed to become infinite such that the initial concentration of different chemical species (i.e. different types of particles) is held constant. In such limits one obtains systems of partial-integral differential equations (PIDEs) for the associated macroscopic concentration field of each species.

Here we build on our previous studies in~\cite{IMS:2022} and~\cite{IMS:2021}, in which we investigated the large-population mean-field limit for the volume reactivity (VR) PBSRD model, which is a generalization of the classical Doi model~\cite{TeramotoDoiModel1967,DoiSecondQuantA,DoiSecondQuantB}. In the VR model, molecules are approximated as point particles that move by Brownian Motion, unimolecular reactions such as $A \to B$ occur as an internal Poisson process with a fixed probability per time for each substrate molecule (i.e. each $A$ molecule), and bimolecular reactions between two individual reactant particles, like $A + B \to \dots$, occur with a probability per time based on the positions of the two substrate molecules. Products of a reaction are placed randomly based on specified particle placement densities. In~\cite{IMS:2022} we derived a set of reaction-diffusion PIDEs that we then proved corresponded to the rigorous large-population limit of the VR PBSRD model, with our analysis being valid for general reaction networks involving both first and second order reactions in which particle numbers can change. Our method of proof was based on formulating a weak representation for the dynamics of measure-valued stochastic processes (MVSPs), which correspond to concentration fields of different chemical species in the VR PBSRD model. We then adapted the classical Stroock-Varadhan Martingale approach to calculating mean-field limits~\cite{MR88a:60130,StroockVaradhan} to prove the large-population limit for the MVSPs. A similar approach was used in~\cite{Oelschlager89a} to derive the analogous mean-field PIDEs for systems that involved only linear zero and first-order reactions (and therefore no interactions between particles). In contrast, a different approach was considered in~\cite{Nolen2019}, which derived the analogous mean-field limit using relative entropy methods, but for systems involving only reactions that conserve particle numbers. Using entropy methods allowed for more quantitative estimates, but as pointed out by the authors, appears restricted to systems in which particle numbers are unchanged by each reaction. Finally, we note that~\cite{PopovicVeber22} considers a novel underlying PBSRD model that is distinct from the VR model. This enables the consideration of higher than second order reactions, and is used in studying a large-population limit in which some species remain discrete with a stochastic evolution while others converge to continuous concentration fields.

We next investigated in~\cite{IMS:2021} how the derived mean-field PIDEs relate to commonly used, but phenomenological, mass action reaction-diffusion PDE models. We demonstrated that the latter could be interpreted as an approximation to the rigorous mean-field PIDEs in the limit that bimolecular interactions become short-ranged when bimolecular reaction kernels have an averaging form. Collectively, the works~\cite{Oelschlager89a,Nolen2019,IMS:2022,IMS:2021} establish the rigorous large population limit for VR PBSRD models, and then rigorously demonstrate how they relate to commonly used macroscopic reaction-diffusion PDE models.

The derived mean-field model has several clear limitations. In particular, as the resulting PIDEs are deterministic it is unable to give any approximation to higher order statistics of the underlying VR PBSRD models. To address this issue, in this work we derive the next order fluctuation corrections to the mean-field PIDEs we derived in~\cite{IMS:2022}, giving a proof that the scaled difference between the PBSRD MVSPs and the mean-field PIDE solutions, see (\ref{Eq:FluctProcess}), converges to the derived fluctuation corrections in the large-population limit.

Our analysis allows us to obtain an approximation to the distribution of the empirical measure of the system to leading order in the system size parameter. Denote by $\mu^{\gamma,j}_{t}$ the marginal distribution (molar concentration) of species $j=1\cdots, J$, where $\gamma$ represents the system size parameter (e.g. Avogadro's number in $\R^d$). Letting $\bar{\mu}^{j}_{t}$ be the limit of $\mu^{\gamma,j}_{t}$ as $\gamma\rightarrow\infty$, we study the behavior of the fluctuation corrections
\begin{align}
\Xi^{\gamma,j}_{t}=\sqrt{\gamma}\left(\mu^{\gamma,j}_{t}-\bar{\mu}^{j}_{t}\right),\quad j=1,\cdots J,\label{Eq:FluctProcess}
\end{align}
as $\gamma\rightarrow\infty$. We prove that, in the appropriate sense, $\Xi^{\gamma,j}_{t}\rightarrow \bar{\Xi}^{j}_{t}$ as $\gamma\rightarrow\infty$ with $\bar{\Xi}^{j}_{\cdot}$ being Gaussian. We therefore obtain the approximation
\[
\mu^{\gamma,j}_{t}\approx \bar{\mu}^{j}_{t}+\frac{1}{\sqrt{\gamma}}\bar{\Xi}^{j}_{t},
\]
which expands on earlier mean-field limit studies by allowing the approximation of the distribution of $\mu^{\gamma,j}_{t}$, as well as associated statistics.

The proof of convergence is based on weak convergence analysis of interacting particle systems. One of the main challenges is that the fluctuations of the empirical distribution correspond to a signed measure-valued process. As discussed in \cite{BDG:2007,FluctuationSpiliopoulosSirignanoGiesecke}, the space of signed measures endowed with the weak topology is in general not metrizable. For this reason, we study convergence of the distribution-valued fluctuation process in the dual space of an appropriately weighted Sobolev space (weights are introduced to control the behavior at infinity).  We use the weights introduced in \cite{Meleard} in studying the McKean-Vlasov equation, and show fluctuations arising from our VR PBSRD model are elements of these spaces, see Section~\ref{S:SobolevSpace} for a short exposition. The limit, $\bar{\Xi}^j_t$, is then obtained as a distribution-valued process in an appropriate weighted negative Sobolev space. In the process of studying the limit of $\Xi^{\gamma,j}$ as $\gamma\rightarrow\infty$, we need to linearize the dynamics of $\mu^{\gamma,j}$ around its limiting behavior $\bar{\mu}^{j}$. Doing so produces a number of remainder terms that need to be appropriately controlled; details of these bounds are in Section~\ref{S:ProofMainTheorem}.


We note that our driving goal in this work is to rigorously derive the main correction terms to the mean-field large-population limit of PBSRD models.  The derived fluctuations limit identifies and enables the approximation of the resulting error terms for the mean-field limit. Our results can be used to both assess errors in the mean-field approximation, but also, as our numerical studies of Section~\ref{S:numerics} demonstrate, enable the approximation of a broader range of statistics from the underlying PBSRD model, reducing the need for expensive Monte-Carlo simulations of the underlying PBSRD model. A number open questions around the large population limit and properties of the PBSRD model still remain. For example, in our work we assume the well-posedness of the process associated with the PBSRD model, and finiteness of certain of its spatial moments. It is of great mathematical and practical interest to identify general reaction networks and placement mechanisms for which one can rigorously prove the validity of these assumption. Initial steps towards this direction have been recently taken in  \cite{Nolen2019,C:2023,IMS:2021}. In addition, it would also be of great interest to apply the theoretical results of this paper to study more complex reaction networks as arise in many biological applications. Our numerical results in Section~\ref{S:numerics} demonstrate the theory, but at the same time they also demonstrate that the fluctuation limit is computationally tractable. It can thus be used to aid computational modeling studies, for example being used in uncertainty quantification and to enable more efficient approximation of higher order statistics than direct simulation of the underlying particle model. We plan to explore these directions in future work.

As stating our main result, Theorem~\ref{T:FluctuationsThm}, requires significant notation and a number of assumptions, we begin in Section~\ref{S:summaryexample} by summarizing our results for the special case of the reversible $A + B \leftrightarrows C$ reaction. In Section~\ref{S:GeneralSetup} we then summarize the main notation we will use in describing the weak MVSP representation for the VR PBSRD model. Section~\ref{S:PreliminaryResults} summarizes our earlier work on the large-population mean-field limit from~\cite{IMS:2022}, along with recalling the main assumptions under which we proved the large-population mean-field limit. In Section~\ref{S:mainresult} we present the main result of this work, Theorem~\ref{T:FluctuationsThm}, along with stating the additional assumptions we introduce to prove this theorem. We summarize in Section~\ref{S:SobolevSpace} the weighted Sobolev spaces used to control behavior at infinity, and in which we prove convergence. In Section~\ref{S:numerics} we illustrate how including fluctuation corrections can enable the accurate approximation of higher-order (i.e. non-mean) statistics of an underlying VR PBSRD model. This is demonstrated via comparison of numerical solutions to the SPIDEs of the fluctuation process model and a jump process approximation to the VR PBSRD model~\cite{I:2013,IZ:2018}. Finally, in Section~\ref{S:ProofMainTheorem} we provide the proof of Theorem~\ref{T:FluctuationsThm} via proving uniform spatial moment bounds of the PBSRD model in Section~\ref{S:MomentBounds}, proving tightness in Section~\ref{S:Tightness}, identifying the limiting SPIDEs satisfied by the fluctuation corrections in Section~\ref{S:Identification}, and proving uniqueness of solutions to the limiting equations in Section~\ref{S:Uniqueness}. An appendix proves several results that are used in Section~\ref{S:ProofMainTheorem}.

\section{Summary of fluctuation corrections for the $A + B \leftrightarrows C$ reaction} \label{S:summaryexample}
In order to introduce the main result of this paper, let us first present the result in the special case of the reversible $A + B \leftrightarrows C$ reaction. We first summarize the large population limit derived in~\cite{IMS:2022}, and then we show the fluctuation corrections we obtain in this work.

Let $\gamma$ denote a system size parameter (i.e. Avogadro's number, or in bounded domains the product of Avogadro's number and the domain volume). We assume all molecules move by Brownian Motion in $\R^d$, with species-dependent diffusivities, $\DA$, $\DB$ and $\DC$ respectively. Let $K_1^{\gamma}(x,y) = K_{1}(x,y)/\gamma$ denote the probability per time an individual \textrm{A} molecule at $x$ and \textrm{B} molecule at $y$ can react, with $m_{1}(z|x,y)$ giving the probability density that when the \textrm{A} and \textrm{B} molecules react they produce a \textrm{C} molecule at $z$. We define $K_2^{\gamma}(z) = K_2(z)$ and $m_2(x,y|z)$ similarly for the reverse reaction. Finally, denote by $A(t)$ the stochastic process for the number of species \textrm{A} molecules at time $t$, and label the position of $i$th molecule of species \textrm{A} at time $t$ by the stochastic process $\vQ_i^{A(t)}(t) \subset \R^d$. The generalized stochastic process
\begin{equation*}
A^{\gamma}(x,t) = \frac{1}{\gamma} \sum_{i=1}^{A(t)} \delta\paren{x - \vQ_i^{A(t)}(t)}
\end{equation*}
corresponds to the molar concentration of species \textrm{A} at $x$ at time $t$. We can similarly define $B^{\gamma}(x,t)$ and $C^{\gamma}(x,t)$. In~\cite{IMS:2022} we derived the large population (thermodynamic) limit where $\gamma \to \infty$ and $A^{\gamma}(x,0)$ converges to a well defined limiting molar concentration field (with similar limits for the molar concentrations of species \textrm{B} and \textrm{C}). We proved, in a weak sense, that as $\gamma \to \infty$,
\begin{align*}
 \paren{A^{\gamma}(x,t),B^{\gamma}(x,t),C^{\gamma}(x,t)} \to \paren{A(x,t), B(x,t), C(x,t)},
\end{align*}
where
\begin{align*}
  \partial_t A(x,t) &= \DA \lap A(x,t) - \int_{\R^d} K_1(x,y) A(x,t) B(y,t) \, dy
  + \int_{\R^{2d}} K_2(z) m_2(x,y|z) C(z,t) \,dy \, dz.\\
  \partial_t B(x,t) &= \DB \lap B(x,t) - \int_{\R^d} K_1(y,x) A(y,t) B(x,t) \, dy
  + \int_{\R^{2d}} K_2(z) m_2(y,x|z) C(z,t) \,dy \, dz.\\
  \partial_t C(x,t) &= \DC \lap C(x,t) - K_2(x) C(x,t) + \int_{\R^{2d}} K_1(y,z) m_1(x|y,z) A(y,t) B(z,t) \, dy \, dz.
\end{align*}
Here $(A(x,t),B(x,t),C(x,t))$ corresponds to the deterministic mean-field large-population limit.

In this work we derive the next order correction to these equations, obtaining an approximation as $\gamma \to \infty$ of
\begin{equation*}
A^{\gamma}(x,t) \sim A(x,t) + \frac{1}{\sqrt{\gamma}} \bar{A}(x,t), \quad \gamma \to \infty,
\end{equation*}
with similar equations for $B(x,t)$ and $C(x,t)$. We prove that in an appropriate (weak) sense,
\begin{equation*}
\lim_{\gamma \to \infty} \sqrt{\gamma} \paren{A^{\gamma}(x,t) - A(x,t)} = \bar{A}(x,t).
\end{equation*}
Here the fluctuation correction processes, $(\bar{A}(x,t),\bar{B}(x,t),\bar{C}(x,t))$, satisfy the stochastic PIDEs
\begin{equation*}
  \begin{aligned}
    \partial_t \bar{A}(x,t) &\begin{multlined}[t]
      =\DA \Delta_x \bar{A} + \xi^A_t
    - \int_{\R^d} K_1(x,y) \brac{A(x,t) \bar{B}(y,t) + \bar{A}(x,t) B(y,t)} dy \\
    + \int_{\R^{2d}} K_2(z) m_2(x,y|z) \bar{C}(z,t) \, dy \, dz,
    \end{multlined}\\
    \partial_t \bar{B}(x,t) &\begin{multlined}[t]
      =\DB \Delta_x \bar{B} + \xi^B_t
    - \int_{\R^d} K_1(y,x) \brac{A(y,t) \bar{B}(x,t) + \bar{A}(y,t) B(x,t)} \, dy \\
    + \int_{\R^{2d}} K_2(z) m_2(y,x|z) \bar{C}(z,t) \, dy \, dz,
    \end{multlined}\\
    \partial_t \bar{C}(x,t) &\begin{multlined}[t]
      =\DB \Delta_x \bar{C} + \xi^C_t - K_2(x) \bar{C}(x,t) \\
     + \int_{\R^{2d}} K_1(y,z) m_1(x|y,z) \brac{A(y,t) \bar{B}(z,t) + \bar{A}(y,t)B(z,t)} \, dy \, dz,
    \end{multlined}
  \end{aligned}
\end{equation*}
where the noise processes, $(\xi^A_t,\xi^B_t,\xi^C_t)$, are distribution-valued mean-zero Gaussian processes, specified precisely in Theorem \ref{T:FluctuationsThm}. For test functions $f$ and $g$ they have the covariance structure
\begin{align*}
  \Cov\brac{\la \xi^A_t, f\ra, \la \xi^A_s, g\ra} &\begin{multlined}[t]
    = 2 \DA \int_0^{s \wedge t} \paren{\nabla_x f(x) \cdot \nabla_x g(x)} A(x,s') \, dx \,ds'\\
    \quad + \int_0^{t \wedge s} \int_{\R^d} K_2(z) \paren{\int_{\R^{2d}} f(x) g(x) m_{2}(x,y|z) \, dx \, dy} C(z,s') \, dz \, ds' \\
    + \int_0^{t \wedge s} \int_{\R^{2d}} K_1(x,y) f(x) g(x) A(x,s') B(y,s') \, dx \, dy \, ds',
  \end{multlined}\\
  \Cov\brac{\la \xi^A_t, f\ra, \la \xi^B_s, g\ra} &\begin{multlined}[t]
    = \int_0^{t \wedge s} \int_{\R^d} K_2(z) \paren{\int_{\R^{2d}} f(x) g(y) m_{2}(x,y|z) \, dx \, dy} C(z,s') \, dz \, ds'\\
    +\int_0^{t \wedge s} \int_{\R^{2d}} K_1(x,y) f(x) g(y) A(x,s') B(y,s') \, dx \, dy \, ds',
  \end{multlined}\\
  \Cov\brac{\la \xi^C_t, f\ra, \la \xi^C_s, g\ra} &\begin{multlined}[t]
    = 2 \DC \int_0^{s \wedge t} \paren{\nabla_z f(z) \cdot \nabla_z g(z)} C(z,s') \, dz \,ds'\\
    \quad + \int_0^{t \wedge s} \int_{\R^{2d}} K_1(x,y) \paren{\int_{\R^{d}} f(z) g(z) m_{1}(z |x,y) \, dz} A(x,s') B(y,s') \, dx \, dy \, ds'\\
    + \int_0^{t \wedge s} \int_{\R^{d}} K_2(z) f(z) g(z) C(z,s')  \, dz \, ds',
  \end{multlined}\\
  \Cov\brac{\la \xi^A_t, f\ra, \la \xi^C_s, g\ra} &\begin{multlined}[t]
    = -\int_0^{t \wedge s} \int_{\R^d} K_2(z) \paren{\int_{\R^{2d}} f(x) g(z) m_{2}(x,y|z) \, dx \, dy} C(z,s') \, dz \, ds'\\
    -\int_0^{t \wedge s} \int_{\R^{2d}} K_1(x,y) \paren{\int_{\R^d} f(x) g(z) m_{1}(z|x,y)\, dz} A(x,s') B(y,s') \, dx \, dy \, ds',
  \end{multlined}
\end{align*}
with $\Cov\brac{\la \xi^B_t, f\ra, \la \xi^B_s, g\ra}$ and $\Cov\brac{\la \xi^B_t, f\ra, \la \xi^C_s, g\ra}$ defined analogously. In Section~\ref{S:numerics}, statistics of the PBSRD solution, $(A^\gamma,B^\gamma,C^\gamma)$, the mean-field model solution, $(A,B,C)$, and the fluctuation process approximation, $(A,B,C) + (\bar{A},\bar{B},\bar{C}) /\sqrt{\gamma}$, are compared numerically for specific choices of the reaction kernels, $K_1$ and $K_2$, and the placement kernels, $m_1$ and $m_2$.


\section{Notation and preliminary definitions}\label{S:GeneralSetup}
Our notation is similar to what we previously used in~\cite{IMS:2022}, however, for completeness it is fully described in this section.

We consider a collection of particles with $J$ possible different types. Note, in the following we will interchangeably use particle or molecule, and type or species. Let $\mathscr{S} = \{ S_{1},\cdots, S_{J}\}$ denote the set of different possible particle types, with $p_i \in \mathscr{S} $ the value of the type of the $i$-th particle. In the remainder, we also assume an underlying probability triple, $(\Omega,\filt,\BP)$, on which all random variables are defined.

Molecules are assumed to diffuse freely in $\R^d$, and undergo at most $L$ possible different types of reactions, denoted as $\mathscr{R}_1, \cdots , \mathscr{R}_L$. We describe the
$\mathscr{R}_\ell$th reaction,  $\ell \in \{1,\dots,L\}$, by
\begin{equation*}
\sum_{j = 1}^{J} \alpha_{\ell j}S_j \rightarrow \sum_{j = 1}^{J} \beta_{\ell j}S_j,
\end{equation*}
where we assume the stoichiometric coefficients $\{\alpha_{\ell j}\}_{j=1}^{J}$ and $\{\beta_{\ell j}\}_{j=1}^J$ are non-negative integers. Let $\vec{\alpha}^{(\ell)} = (\alpha_{\ell 1}, \alpha_{\ell 2}, \cdots, \alpha_{\ell J})$ and $\vec{\beta}^{(\ell)} = (\beta_{\ell 1}, \beta_{\ell 2}, \cdots, \beta_{\ell J})$ be multi-index vectors collecting the coefficients of the $\ell$th reaction. We denote the reactant and product orders of the reaction by $|\vec{\alpha}^{(\ell)}|\doteq\sum_{i = 1}^{J} \alpha_{\ell i} \leq 2$ and $|\vec{\beta}^{(\ell)}|\doteq\sum_{j = 1}^{J} \beta_{\ell j}\leq 2$, assuming that at most two reactants and two products participate in any reaction. We therefore implicitly assume all reactions are at most second order. For subsequent notational purposes, we order the reactions such that the first $\tilde{L}$ reactions correspond to those that have no products, i.e. annihilation reactions of the form
\begin{equation*}
\sum_{j = 1}^{J} \alpha_{\ell j}S_j \rightarrow \emptyset,
\end{equation*}
for $\ell \in \{1,\dots,\tilde{L}\}$. We assume the remaining $L - \tilde{L}$ reactions have one or more product particles.

Let $D^i$ label the diffusion coefficient for the $i$th molecule, taking values in $\{D_1, \dots, D_J\}$, where $D_j$ is the diffusion coefficient for species $S_j$, $ j = 1,\cdots, J$. We denote by $Q^{i}_{t}\in \R^d$ the position of the $i$th molecule, $i \in \N_+$, at time $t$. A particle's state can be represented as a vector in $\hat{P}= \mathbb{R}^{d}\times \mathscr{S}$, the combined space encoding particle position and type. This state vector is subsequently denoted by $\hat Q^i_t \Def (Q^{i}_{t},p_i)$.

We now formulate our representation for the (number) concentration, equivalently number density, fields of each species. Let $E$ be a complete metric space and $M(E)$ the collection of measures on $E$. Let $\mathcal{M}(E)$ be the subset of $M(E)$ consisting of all finite, non-negative point measures
\begin{align*}
\mathcal{M}(E)&=\left\{\sum_{i=1}^{N}\delta_{Q^{i}}, N\geq 1, Q^{1},\cdots, Q^{N}\in E\right\}.
\end{align*}

For  $f: E\mapsto \mathbb{R}$ and $\mu\in M(E)$, define
\begin{equation*}
\la f,\mu\ra_{E} = \int_{x \in E}f(x)\mu(d x).
\end{equation*}
We will frequently have $E=\mathbb{R}^{d}$. In this case we omit the subscript $E$ and simply write $\la f,\mu\ra$. For each $t\ge 0$, we define the concentration of particles in the system at time $t$ by the distribution
\begin{equation} \label{eq:densitymeasdef}
\nu_t = \sum_{i=1}^{N(t)}\delta_{\hat Q^i_t} = \sum_{i=1}^{N(t)}\delta_{Q^i_t} \delta_{p_i},
 \end{equation}
where, borrowing notation from \cite{BM:2015}, $N(t) = \la 1, \nu_t\ra_{\hat{P}}$ represents the stochastic process for the total number of particles at time $t$. To investigate the behavior of different types of particles, we denote the marginal distribution on the $j$th type, i.e. the concentration field for species $j$, by
\begin{equation*}
\nu^{j}_{t}(\cdot)= \nu^{}_{t}(\cdot\times \{S_j\}) \in M(\R^{d}),
\end{equation*}
a distribution on $\R^d$. $N_j(t) = \la 1, \nu^j_t\ra$ will label the total number of particles of type $S_j$ at time $t$. Note that in the remainder, in any rigorous calculation $\nu_{t}$ and $\nu_{t}^{j}$ will be measures and treated as such. However, we will abuse notation and also refer to them as concentration fields, i.e. number densities. Strictly speaking, the latter should refer to the densities associated with such measures, but we ignore this distinction in subsequent discussions. For $\nu$ any fixed particle distribution of the form~\eqref{eq:densitymeasdef}, we will also use an alternative representation in terms of the marginal distributions $\nu^j\in \mathcal{M}(\R^d)$ for particles of type $j$,
\begin{equation} \label{eq:densitymeasdefmargrep}
\nu = \sum_{j = 1}^{J} \nu^j\delta_{S_j}  \in \mathcal{M}(\hat{P}).
\end{equation}

In addition to having notations for representing particle concentration fields, we will also often make use of state vectors for the positions of particles of a given type.
  Following the notation established in \cite{BM:2015} (see Section 6.3 therein), we let $\mathbb{N}^{*}=\mathbb{N}\setminus\{0\}$ and let $H=(H^{1},\cdots, H^{k},\cdots ):\mathcal{M}\mapsto \left(\mathbb{R}^{d}\right)^{\mathbb{N}^{*}}$ be defined by
\begin{align}
H\left(\sum_{i=1}^{N}\delta_{Q^{i}}\right)=\left(Q^{\sigma(1)},\cdots, Q^{\sigma(N)}, 0,0,\cdots\right)
\end{align}
where $\sigma$ is a permutation such that $Q^{\sigma(1)}\preceq\cdots\preceq Q^{\sigma(N)}$, arising from an (assumed) fixed underlying ordering on $\mathbb{R}^{d}$. As commented in \cite{BM:2015}, this function $H$ allows us to address a notational issue. In particular, choosing a particle of a certain type uniformly among all particles in $\nu\in\mathcal{M}$ amounts to basically choosing uniformly in the set $\{1,\cdots,\la 1,\nu \ra\}$ and then choosing the individual particle from the arbitrary fixed ordering.

In our specific case, define the particle index maps $\{\sigma_j(k)\}_{k=1}^{N_j(t)}$, which encode a fixed ordering for particles of species $j$, $Q_t^{\sigma_j(1)}\preceq\cdots\preceq Q_t^{\sigma_j(N_j(t))}$, arising from an (assumed) fixed underlying ordering on $\mathbb{R}^{d}$, and we shall consider
\begin{align}
H(\nu^{j}_t)=\left(Q^{\sigma_j(1)}_{t},\cdots, Q^{\sigma_j(N_j(t))}_{t}, 0,0,\cdots\right)
\end{align} 
the position state vector for type $j$ particles, using the same ordering on $\mathbb{R}^{d}$, with $H^i(\nu^j_t)\in \R^d$ labeling the $i$th entry in $H(\nu^j_t)$. Note, as particles of the same type are assumed indistinguishable, there is no ambiguity in the value of $H(\nu_{t}^j)$ in the case that two particles of type $j$ have the same position.

With the preceding definitions, we finally introduce a system of notation to encode reactant and particle positions and configurations that are needed to later specify reaction processes.

\begin{definition}\label{def:reacIndSpace}
  In describing the dynamics of $\nu_t$, we will sample vectors containing the indices of the specific reactant particles participating in a single $\ell$-type reaction from the reactant index space
\begin{equation*}
  \mathbb{I}^{(\ell)} = \left(\N\setminus\{0\}\right)^{|\vec{\alpha}^{(\ell)}|}.
\end{equation*}
For the allowable reactions considered in this work we label the elements of $\mathbb{I}^{(\ell)}$ in a manner that shows which species they belong to:
\begin{enumerate}
  \item For $\mathscr{R}_{\ell}$ of the form $\varnothing \to \cdots$
  \begin{equation*}
    \mathbb{I}^{(\ell)} = \varnothing.
  \end{equation*}
  \item For $\mathscr{R}_{\ell}$ of the form $S_j \to \cdots$
  \begin{equation*}
    \mathbb{I}^{(\ell)} = \{i_1^{(j)} \in \N\setminus\{0\}\}
  \end{equation*}
  \item For $\mathscr{R}_{\ell}$ of the form $S_j + S_k \to \cdots$ with $j < k$
  \begin{equation*}
    \mathbb{I}^{(\ell)} = \{(i_1^{(j)},i_1^{(k)}) \in (\N\setminus\{0\})^2 \}
  \end{equation*}
  \item For $\mathscr{R}_{\ell}$ of the form $2 S_j \to \cdots$
  \begin{equation*}
    \mathbb{I}^{(\ell)} = \{(i_1^{(j)},i_2^{(j)}) \in (\N\setminus\{0\})^2  \}.
  \end{equation*}
\end{enumerate}
\end{definition}
A generic element $\bi \in \mathbb{I}^{(\ell)}$ can then be written as $\bi = (i_1^{(1)},\cdots,i_{\alpha_{\ell 1}}^{(1)},\cdots,i_1^{(J)},\cdots,i_{\alpha_{\ell J}}^{(J)})$.

\begin{definition}\label{def:prodIndSpace}
We define the product index space $\mathbb{J}^{(\ell)}$ analogously to $\mathbb{I}^{(\ell)}$, with $\bj \in \mathbb{J}^{(\ell)}$ given by $\bj = (j_1^{(1)},\cdots,j_{\beta_{\ell 1}}^{(1)},\cdots,j_1^{(J)},\cdots,j_{\beta{\ell J}}^{(J)})$.
\end{definition}

\begin{definition}\label{def:reacPosSpace}
We define the reactant particle position space analogously to $\mathbb{I}^{(\ell)}$,
\begin{equation*}
\mathbb{X}^{(\ell)} = \left(\R^d\right)^{|\vec{\alpha}^{(\ell)}|},
\end{equation*}
with an element $\vx \in \mathbb{X}^{(\ell)}$ written as $\vec{x} = (x^{(1)}_1, \cdots, x^{(1)}_{\alpha_{\ell 1}}, \cdots, x^{(J)}_1, \cdots, x^{(J)}_{\alpha_{\ell J}})$.
For $\vec{x}\in \mathbb{X}^{(\ell)}$ a sampled reactant position configuration for one individual $\mathscr{R}_{\ell}$ reaction, $x_r^{(j)}$ then labels the sampled position for the $r$th reactant particle of species $j$ involved in the reaction. Let $d\vec{x} = \left( \bigwedge_{j = 1}^J (\bigwedge_{r = 1}^{ \alpha_{\ell j}} d x_r^{(j)}) \right) $ be the corresponding volume form on $\mathbb{X}^{(\ell)}$, which also naturally defines an associated Lebesgue measure.
\end{definition}

\begin{definition}\label{def:prodPosSpace}
For reaction $\mathscr{R}_{\ell}$ with $\tilde{L} + 1\leq \ell\leq L$, i.e. having at least one product particle, define the product position space analogously to $\mathbb{X}^{(\ell)}$,
\begin{equation*}
\mathbb{Y}^{(\ell)} = \left(\R^d\right)^{|\vec{\beta}^{(\ell)}|},
\end{equation*}
with an element $\vy \in \mathbb{Y}^{(\ell)}$ written as $\vy = (y^{(1)}_1, \cdots, y^{(1)}_{\beta_{\ell 1}}, \cdots, y^{(J)}_1, \cdots, y^{(J)}_{\beta_{\ell J}})$. Let $d\vec{y} = \left( \bigwedge_{j = 1}^J (\bigwedge_{r = 1}^{ \beta_{\ell j}} d y_r^{(j)}) \right) $ be the corresponding volume form on $\mathbb{Y}^{(\ell)}$, which also naturally defines an associated Lebesgue measure.
\end{definition}

\begin{definition}\label{def:projMapping}
Consider a fixed reaction $\mathscr{R}_{\ell}$, with $\vec{i}\in \mathbb{I}^{(\ell)}$ and $\nu$ corresponding to a fixed particle distribution given by~\eqref{eq:densitymeasdef} with representation~\eqref{eq:densitymeasdefmargrep}. We define the $\ell$th projection mapping $\mathcal{P}^{(\ell)} :   \mathcal{M}(\hat{P})\times  \mathbb{I}^{(\ell)} \rightarrow  \mathbb{X}^{(\ell)}$ as
\begin{equation*}
\mathcal{P}^{(\ell)}(\nu,  \vec{i}) = (H^{i^{(1)}_1}(\nu^1), \cdots, H^{i^{(1)}_{\alpha_{\ell 1}}}(\nu^1) , \cdots , H^{i^{(J)}_1}(\nu^{J}), \cdots, H^{i^{(J)}_{\alpha_{\ell J}}}(\nu^{J})).
\end{equation*}
When reactants with indices $\vec{i}$ in particle distribution $\nu$ are chosen to undergo a reaction of type $\mathscr{\ell}$, $\mathcal{P}^{(\ell)}(\nu,\vec{i})$ then gives the vector of the corresponding reactant particles' positions. For simplicity of notation, in the remainder we will sometimes evaluate $\mathcal{P}^{(\ell)}$ with inconsistent particle distributions and index vectors. In all of these cases the inconsistency will occur in terms that are zero, and hence not matter in any practical way.
\end{definition}

\begin{definition}\label{def:effSamplingSpace}
Consider a fixed reaction $\mathscr{R}_{\ell}$, with $\nu$ a fixed particle distribution given by~\eqref{eq:densitymeasdef} with representation~\eqref{eq:densitymeasdefmargrep}. Using the notation of Def.~\ref{def:reacIndSpace}, we define the allowable reactant index sampling space $\Omega^{(\ell)}(\nu)\subset \mathbb{I}^{(\ell)}$ as
\begin{equation*}
\Omega^{(\ell)}(\nu) = \begin{cases}
\varnothing, & \abs{\vec{\alpha}^{(\ell)}}=0,\\
\{ \vec{i} = i_{1}^{(j)} \in \mathbb{I}^{(\ell)} \, | \, i_1^{(j)} \leq \la 1, \nu^j \ra \}, & \abs{\vec{\alpha}^{(\ell)}}=\alpha_{\ell j} = 1,\\
\{ \vec{i} = (i_{1}^{(j)},i_{2}^{(j)}) \in \mathbb{I}^{(\ell)} \, | \,   i_1^{(j)} < i_{2}^{(j)} \leq \la 1, \nu^j \ra \}, & \abs{\vec{\alpha}^{(\ell)}}=\alpha_{\ell j} = 2,\\
\{ \vec{i} = (i_{1}^{(j)},i_{1}^{(k)})\in \mathbb{I}^{(\ell)} \, | \, i_1^{(j)} \leq \la 1, \nu^j \ra, i_1^{(k)} \leq \la 1, \nu^k \ra \}, & \abs{\vec{\alpha}^{(\ell)}}=2, \quad \alpha_{\ell j} = \alpha_{\ell k} = 1, \quad j < k.
\end{cases}
\end{equation*}
Note that in the calculations that follow $\Omega^{(\ell)}(\nu)$ will change over time due to the fact that $\nu=\nu_{t}$ changes over time, but this will not be explicitly denoted for notational convenience.
\end{definition}

\begin{definition}\label{def:lambda}
Consider a fixed reaction $\mathscr{R}_{\ell}$, with $\nu$ any element of $M(\hat{P})$ having the representation~\eqref{eq:densitymeasdefmargrep}. We define the $\ell$th reactant measure mapping $\lambda^{(\ell)} [\, \cdot \,] : M(\hat{P}) \to M( \mathbb{X}^{(\ell)})$ evaluated at $\vx \in \mathbb{X}^{(\ell)}$ via $\lambda^{(\ell)}[\nu](d\vx) = \otimes_{j = 1}^J(\otimes_{ r = 1}^{\alpha_{\ell j }} \nu^j(dx_r^{(j)})) $.
\end{definition}

\begin{definition}\label{def:offDiagSpace}
For reaction $\mathscr{R}_{\ell}$, define a subspace $\tilde{\mathbb{X}}^{(\ell)}\subset\mathbb{X}^{(\ell)}$ by removing all particle reactant position vectors in $\mathbb{X}^{(\ell)}$ for which two particles of the same species have the same position. That is
\begin{equation*}
\tilde{\mathbb{X}}^{(\ell)} = \mathbb{X}^{(\ell)}\setminus\{ \vec{x} \in\mathbb{X}^{(\ell)} \, | \,  x_r^{(j)}  =  x_{k}^{(j)} \text{ for some } 1\leq j\leq J, 1\leq k \neq r \leq \alpha_{\ell j} \}.
\end{equation*}
\end{definition}


\section{Previous Results and Assumptions for the Mean Field Limit}\label{S:PreliminaryResults}
In this section we review our previous results on the mean field large-population limit from~\cite{IMS:2022} that are relevant to the current work.

Let us consider the time evolution of the processes $\nu^{\vec{{\vec{\zeta}}},j}_t(x) = \sum_{i=1}^{N^{\vec{\zeta}}(t)}\delta_{Q^i_t}(x) 1_{p_i=S_j}$ which give the spatial distribution of particles of type $j$ (i.e. number density or concentration). Here $N^{\vec{\zeta}}(t) = \sum_{j=1}^J N^{\vec{\zeta},j}(t)$ denotes the total number of particles in the system at time $t$, with $N^{\vec{\zeta},j}(t) = \la 1,\nu^{{\vec{\zeta}},j}_{t}\ra$ the number of particles of type $j$ at time $t$. ${\vec{\zeta}}= (\frac{1}{\gamma}, \eta)\in(0,1)^{2}$ is a two-vector consisting of a scaling parameter, $\gamma$, and a displacement range parameter, $\eta$ (which will be explained in Section~\ref{S:PreviousAssumptions}).

In the large population limit we consider $\gamma$ plays the role of a system size, and is considered to be large (e.g. Avogadro's number, or in bounded domains the product of Avogadro's number and the domain volume)~\cite{DK:2015}. On the other hand, $\eta$ is a regularizing parameter allowing us to be able to consider and rigorously handle delta-function placement measures for reaction products (a common choice in many PBSRD simulation methods). We will further clarify these parameters later on,  focusing on the (large-population) limit that $\gamma\to \infty$ and $\eta\to 0$ jointly, denoted as $\vec{\zeta}\to 0$.

\subsection{Generator and process level description}
To formulate the process-level model, it is necessary to specify more concretely the reaction process between individual particles. For reaction $\mathscr{R}_{\ell}$, denote by $K_\ell^{\gamma}(\vec{x})$ the rate (i.e. probability per time) that reactant particles with positions $\vec{x}\in \mathbb{X}^{(\ell)}$ react. As described in the next section, we assume this rate function has a specific scaling dependence on $\gamma$. Let $m^\eta_\ell( \vec{y} \, | \,  \vec{x})$ be the placement measure when the reactants at positions $\vec{x}\in \mathbb{X}^{(\ell)}$ react and generate products at positions $\vec{y}\in \mathbb{Y}^{(\ell)}$. We assume this placement measure depends on the displacement range parameter $\eta$.

To describe a reaction $\mathscr{R}_\ell$ with no products,  i.e. $1\leq \ell \leq \tilde{L}$,  we associate with it a Poisson point measure $dN_\ell(s, \vec{i}, \theta)$  on $\R_+\times \mathbb{I}^{(\ell)} \times\R_+$. Here $\vec{i}\in \mathbb{I}^{(\ell)}$ gives the sampled reactant configuration, with $i^{(j)}_r$ labeling the $r$th sampled index of species $j$. The corresponding intensity measure of $dN_{\ell}$ is given by $d\bar{N}_\ell(s, \vec{i}, \theta)= ds \, \left( \bigwedge_{ j = 1}^J \left( \bigwedge_{r = 1}^{\alpha_{\ell j}}\left(\sum_{k\geq 0} \delta_k( i_{r}^{(j)})\right) \right)\right)\, d\theta$. Analogously, for each reaction $\mathscr{R}_\ell$ with products,  i.e. $\tilde{L} + 1 \leq \ell \leq L$,  we associate with it a Poisson point measure $dN_\ell(s, \vec{i}, \vec{y}, \theta_1, \theta_2)$ on $\R_+\times \mathbb{I}^{(\ell)} \times \mathbb{Y}^{(\ell)} \times\R_+\times\R_+$. Here $\vec{i}\in \mathbb{I}^{(\ell)}$ gives the sampled reactant configuration, with $i^{(j)}_r$ labeling the $r$th sampled index of species $j$. $\vec{y}\in \mathbb{Y}^{(\ell)}$ gives the sampled product configuration, with $y_r^{(j)}$ labeling the sampled position for the $r$th newly created particle of species $j$. The corresponding intensity measure is given by $d\bar{N}_\ell(s, \vec{i}, \vec{y}, \theta_1, \theta_2)= ds \, \left( \bigwedge_{ j = 1}^J \left( \bigwedge_{r = 1}^{\alpha_{\ell j}}\left(\sum_{k\geq 0} \delta_k( i_{r}^{(j)})\right) \right)\right)\,d\vec{y}\, d\theta_1 \, d\theta_2$.

The existence of the Poisson point measure follows as the intensity measure is $\sigma$-finite (see Chapter I - Theorem 8.1 in \cite{NW:2014} or Corollary 9.7 in \cite{K:2001}). Let $d\tilde{N}_{\ell}(s, \vec{i}, \vec{y}, \theta_1, \theta_2) = dN_{\ell}(s, \vec{i}, \vec{y}, \theta_1, \theta_2) - d\bar{N}_{\ell}(s, \vec{i}, \vec{y}, \theta_1, \theta_2)$ be the compensated Poisson measure, for $\tilde{L}+1 \leq \ell \leq L$. For any measurable set $A\in  \mathbb{I}^{(\ell)} \times \mathbb{Y}^{(\ell)} \times\R_+\times\R_+$  such that $\bar{N}_{\ell}(\cdot,A)<\infty$, which is true if for example A is bounded, $N_\ell(\, \cdot \, , A)$ is a Poisson process and $\tilde{N}_\ell(\, \cdot \,, A)$ is a martingale (see Proposition 9.18 in \cite{K:2001}). Similarly, we can define $d\tilde{N}_{\ell}(s, \vec{i}, \theta) = dN_{\ell}(s, \vec{i}, \theta) - d\bar{N}_{\ell}(s, \vec{i}, \theta)$, for $1\leq \ell \leq \tilde{L}$. In this case, given any measurable set $A\in  \mathbb{I}^{(\ell)} \times\R_+$ such that $\bar{N}_{\ell}(\cdot,A) < \infty$, we then have that $N_\ell(\, \cdot \, , A)$ is a Poisson process and $\tilde{N}_\ell(\, \cdot \,, A)$ is a martingale.

With the preceding definitions, we can formulate a weak representation for the dynamics of $\nu_t^{\vec{\zeta}}$ and the concentration (i.e. number density) fields for each species, $\nu_t^{{\zeta},j}$, see~\cite{IMS:2022} for the resulting equations.  Note, here by weak representation we mean that the time evolution of $\mu^{{\vec{\zeta}}}_{t}$ is given in terms of pairings with appropriate test functions, see (\ref{Eq:EM_j_formula}). In this work we only require a weak representation for the time evolution of the scaled empirical measures, i.e. the molar concentration field for species $j$, $\mu^{{\vec{\zeta}}, j}_{t}=\frac{1}{\gamma}\nu^{{\vec{\zeta}}, j}_{t}$ with $j =1, \cdots, J$. Note, with this definition $\mu^{{\vec{\zeta}}}_{t}=\frac{1}{\gamma}\nu^{{\vec{\zeta}}}_{t}= \sum_{j = 1}^J\mu^{{\vec{\zeta}}, j}_{t}\delta_{S_j} $.

Let us denote $\{W^{n, j}_t\}_{n\in \N_+}$ as a countable collection of standard independent Brownian motions in $\R^d$ for species $j$, $j = 1, 2, \cdots, J$.
We can write the marginal distribution (molar concentration) of species $j$ as
\begin{equation*}
\mu_t^{{\vec{\zeta}}, j}(dx) = \frac{1}{{\gamma}}\sum_{i = 1}^{{\gamma}\la 1, \mu^{{\vec{\zeta}}, j}_t\ra} \delta_{H^i({\gamma}\mu_t^{{\vec{\zeta}}, j})}(dx), \quad j \in \{1,\dots,J\},
\end{equation*}
which, for $f\in C^{2}_{b}(\R^{d})$, satisfies the coupled system
\begin{equation}\label{Eq:EM_j_formula}
\begin{aligned}
\la f,\mu^{{\vec{\zeta}}, j}_{t}\ra
&
=\la f,\mu^{{\vec{\zeta}}, j}_{0}\ra + \frac{1}{{\gamma}}\sum_{i\geq 1}\int_{0}^{t}1_{\{i\leq \gamma \la 1, \mu_{s-}^{{\vec{\zeta}}, j}\ra\}}\sqrt{2D_{j}}\frac{\partial f}{\partial Q}(H^i(\gamma\mu_{s-}^{{\vec{\zeta}}, j}))dW^{i, j}_{s}
+\frac{1}{\gamma}\int_{0}^{t}\sum_{i=1}^{\gamma\la 1, \mu_{s-}^{{\vec{\zeta}}, j}\ra} D_{j}\frac{\partial^{2} f}{\partial Q^{2}}(H^i(\gamma\mu_{s-}^{{\vec{\zeta}}, j}))ds\\
&
\quad-\frac{1}{\gamma}\sum_{\ell = 1}^{\tilde{L}} \int_{0}^{t}\int_{\mathbb{I}^{(\ell)}} \int_{\mathbb{R}_{+}}\la f, \sum_{r = 1}^{\alpha_{\ell j}} \delta_{H^{i_r^{(j)}}(\gamma\mu^{{\vec{\zeta}},j}_{s-})}  \ra 1_{\{\vec{i} \in \Omega^{(\ell)}(\gamma\mu^{\vec{\zeta}}_{s-})\}} 1_{ \{ \theta \leq K_\ell^{\gamma}\left(\mathcal{P}^{(\ell)}(\gamma\mu_{s-}^{\vec{\zeta}}, \vec{i}) \right) \}}  dN_{\ell}(s,\vec{i}, \theta)\\
&
\quad+\frac{1}{\gamma}\sum_{\ell = \tilde{L}+1}^L \int_{0}^{t}\int_{\mathbb{I}^{(\ell)}} \int_{\mathbb{Y}^{(\ell)}}\int_{\mathbb{R}_{+}^2}\la f, \sum_{r = 1}^{\beta_{\ell j}} \delta_{y_r^{(j)}} - \sum_{r = 1}^{\alpha_{\ell j}} \delta_{H^{i_r^{(j)}}(\gamma\mu^{{\vec{\zeta}},j}_{s-})} \ra
1_{\{\vec{i} \in \Omega^{(\ell)}(\gamma\mu^{\vec{\zeta}}_{s-})\}}  \\
&
\qquad\qquad\qquad\qquad\times 1_{ \{ \theta_1 \leq K_\ell^{\gamma}\left(\mathcal{P}^{(\ell)}(\gamma\mu_{s-}^{\vec{\zeta}}, \vec{i}) \right) \}}  1_{ \{ \theta_2 \leq  m^\eta_\ell\left(\vec{y} \,  | \, \mathcal{P}^{(\ell)}(\gamma\mu_{s-}^{\vec{\zeta}}, \vec{i}) \right) \}}  dN_{\ell}(s,\vec{i}, \vec{y},\theta_1, \theta_2),
\end{aligned}
\end{equation}
for $j \in \{1,\dots,J\}$.

In this formulation, one important fact is that for fixed ${\vec{\zeta}}$, $\gamma\la 1,\mu^{{\vec{\zeta}}, j}_{s-}\ra$ is assumed finite, see Assumption~\ref{Assume:kernalBdd} in the next section, which provides exchangeability of the sum and Lebesgue integral. It also implies that the stochastic integrals with respect to Brownian motions in~\eqref{Eq:EM_j_formula} are martingales (for a fixed ${\vec{\zeta}}$).

Relation (\ref{Eq:EM_j_formula}) captures the dynamics of our particle system. We refer the interested reader to Remark 5.3 of \cite{IMS:2022} for a discussion on well posedness of (\ref{Eq:EM_j_formula}) and note that for the purposes of this paper, we assume that we work with reaction networks for which well posedness of (\ref{Eq:EM_j_formula}) holds. We also refer the interested reader to the very recent article \cite{C:2023}, where well-posedness of systems related to (\ref{Eq:EM_j_formula}) is proven.


\subsection{Assumptions for the Mean Field Limit}\label{S:PreviousAssumptions}
In this section we summarize the assumptions we previously used in proving the mean field large population limit~\cite{IMS:2022}. We will assume they hold in the remainder in studying the fluctuation corrections.

\begin{assumption}\label{Assume:kernalBdd0}
  We assume that for all $1\leq \ell \leq L$, the reaction rate kernel $K_\ell(\vec{x})$ is uniformly bounded for all $\vec{x}\in \mathbb{X}^{(\ell)}$. We denote generic constants that depend on this bound by $C(K)$.
  \end{assumption}

  \begin{assumption}\label{Assume:measureP}
  We assume that for any $\eta \geq 0$, $\tilde{L} + 1\leq \ell \leq L$, $\vec{y}\in \mathbb{Y}^{(\ell)}$ and $\vec{x}\in \mathbb{X}^{(\ell)}$,  the placement density $m^\eta_{\ell}(\vec{y} \, | \, \vec{x})$ is uniformly bounded in $\vec{x}$ and $\vec{y}$, and a  probability density in $\vec{y}$, i.e. $\int_{\mathbb{Y}^{(\ell)}} m^\eta_{\ell}(\vec{y} \, | \, \vec{x})\, d\vec{y}= 1$.
  \end{assumption}

  We want to allow for placement densities involving delta-functions. To do so in a mathematically rigorous way we introduced the displacement (i.e. smoothing) parameter $\eta$, through which we can define a corresponding mollifier in a standard way, as given by Definition~\eqref{def:molifier} below. This is needed for~\eqref{Eq:EM_j_formula} to be well-defined, since expressions like $\{\theta_2\leq m^{\eta}_\ell\left(\vec{y} \,  | \, \mathcal{P}^{(\ell)}(\nu_{s-}^{\vec{\zeta}}, \vec{i}) \right)\}$ are nonsensical when $\eta=0$ and the placement density is a Dirac delta function.
  \begin{definition}\label{def:molifier}
  For $x\in\R^d$, let $G(x)$ denote a standard positive mollifier and $G_\eta(x) = \eta^{-d}G(x/\eta)$. That is, $G(x)$ is a smooth function on $\R^d$ satisfying the following four requirements
  \begin{enumerate}
  \item $G(x)\geq 0$,
  \item $G(x)$ is compactly supported in $B(0, 1)$, the unit ball in $\R^d$,
  \item $\int_{\R^d} G(x)\, dx = 1$,
  \item $\lim_{\eta\to 0} G_\eta(x) = \lim_{\eta\to 0} \eta^{-d}G(x/\eta) = \delta_0(x) $, where $\delta_0(x)$ is the Dirac delta function and the limit is taken in the space of Schwartz distributions.
  \end{enumerate}
  \end{definition}

  The allowable forms of the placement density for each possible reaction are given by Assumptions \ref{Assume:measureOne2One}-\ref{Assume:measureOne2Two}:
  \begin{assumption}\label{Assume:measureOne2One}
  If $\mathscr{R}_{\ell}$ is a first order reaction of the form $S_i \rightarrow S_j$, we assume that the placement density  $m^\eta_{\ell}(y\,|\, x)$ takes the mollified form of
  $$m^\eta_{\ell}(y\, |\, x) =G_\eta(y-x).$$
  \end{assumption}
  Note that its distributional limit as $\eta \to 0$ is given by
  $$m_{\ell}(y\,|\, x) = \delta_x(y).$$
  This describes that the newly created $S_j$ particle is placed  at the position of the reactant $S_i$ particle.

  \begin{assumption}\label{Assume:measureTwo2One}
  If $\mathscr{R}_{\ell}$ is a second order reaction of the form $S_i + S_k \rightarrow S_j$, we assume that the binding placement density  $m_{\ell}(z\, |\, x,y)$ takes the mollified form of
  $$m^\eta_{\ell}(z\,|\,x,y) =\sum_{i=1}^{I}p_i \times G_\eta\left(z-(\alpha_i x +(1-\alpha_i)y)\right).$$
  \end{assumption}
  Note that its distributional limit as $\eta \to 0$ is given by
   $$m_{\ell}(z\,|\,x,y) = \sum_{i=1}^{I}p_i \times \delta\left(z-(\alpha_i x +(1-\alpha_i)y)\right),$$
    where $I$ is a fixed finite integer and $\sum_i p_i = 1$. This describes that the creation of particle $S_j$ is always on the segment connecting the reactant $S_i$ and reactant $S_k$ particles, but allows some random choice of position. A special case would be $I = 2$, $p_i = \tfrac{1}{2}$, $\alpha_1 = 0$ and $\alpha_2 = 1$, which corresponds to placing the particle randomly at the position of one of the two reactants. One common choice is taking $I = 1$, $p_1 = 1$ and choosing $\alpha_1$ to be the diffusion weighted center of mass~\cite{IZ:2018}.

  \begin{assumption}\label{Assume:measureTwo2Two}
  If $\mathscr{R}_{\ell}$ is a second order reaction of the form $S_i + S_k \rightarrow S_j + S_r$, we assume that the placement density  $m_\ell(z, w \, | \, x, y)$ takes the mollified form of
  $$m^\eta_\ell(z, w \, | \, x, y) = p\times G_\eta\left(x-z\right)\otimes G_\eta\left(y-w\right)   + (1-p)\times G_\eta\left(x-w\right)\otimes G_\eta\left(y-z\right).$$
  \end{assumption}
  Note that its distributional limit as $\eta \to 0$ is given by
  $$m_\ell(z, w \, | \, x, y) = p\times\delta_{(x, y)}\left((z, w)\right)  + (1-p)\times\delta_{(x, y)}\left((w, z)\right).$$
  This describes that newly created product $S_j$ and $S_r$ particles are always at the positions of the reactant $S_i$ and $S_k$ particles. $p$ is typically either $0$ or $1$, depending on the underlying physics of the reaction.

  \begin{assumption}\label{Assume:measureOne2Two}
  If $\mathscr{R}_{\ell}$ is a first order reaction of the form $S_i \rightarrow S_j + S_k$, we assume the unbinding displacement density is in the mollified form of $$m^\eta_{\ell}(x,y\,|\,z) = \rho(|x-y|) \sum_{i=1}^{I}p_i\times G_\eta\left(z-(\alpha_i x +(1-\alpha_i)y)\right),$$
       with $\sum_i p_i = 1$. Here we assume the relative separation of the product $S_j$ and $S_k$ particles, $\abs{x -y}$, is sampled from the probability density $\rho(|x-y|)$. Their (weighted) center of mass is sampled from the density encoded by the sum of $\delta$ functions. Such forms are common for detailed balance preserving reversible bimolecular reactions~\cite{IZ:2018}.
  \end{assumption}
   Note that the distributional limit of $m^\eta_{\ell}(x,y\,|\,z)$ as $\eta \to 0$ is given by
   $$m_{\ell}(x,y\,|\,z) = \rho(|x-y|) \sum_{i=1}^{I}p_i\times\delta\left(z-(\alpha_i x +(1-\alpha_i)y)\right).$$

  We further assume some regularity of the separation placement density, $\rho(r)$, introduced in Assumption~\ref{Assume:measureOne2Two}:
  \begin{assumption}\label{Assume:measureMrho}
  For Assumption~\ref{Assume:measureP} to be true, we'll need that the probability density $\rho$ is normalized, i.e.
  $$\int_{\R^d} \rho(|w|)\, dw = 1.$$
  In the remainder we abuse notation, and write  $\rho\in L^{1}(\mathbb{R}^{d})$ to mean $\rho(\abs{\cdot}) \in L^1(\R^d)$.
  \end{assumption}
  Though $m_{\ell}^{\eta}(x,y\,|\,z)$ is not a direct mollification of $m_{{\ell}}(x,y\,|\,z)$, with this assumption it is still properly normalized with respect to $x$ and $y$ as
   \begin{align*}
   \int_{\R^{2d}} m^\eta_{\ell}(x,y\,|\,z) \, dx \, dy = \sum_{i=1}^{I} p_{i} \int_{\R^{2d}} \rho(\abs{w}) \delta(z - y - \alpha_{i} w) \, dw \, dy = 1.
   \end{align*}

  Finally, to study the large-population limit of the population density measures, we must specify how the reaction kernels depend on the scaling parameter (i.e. system size parameter) $\gamma$. Motivated by the classical spatially homogeneous reaction network large-population limit~\cite{DK:2015}, we choose
  \begin{assumption}\label{Assume:rescaling}
  The reaction kernel is assumed to have the explicit $\gamma$ dependence that
  \begin{equation*}
  K_\ell^{\gamma}(\vec{x}) = \gamma^{1 - |\vec{\alpha}^{(\ell)}|}K_\ell(\vec{x})
  \end{equation*}
  for any $\vec{x}\in \mathbb{X}^{(\ell)}$, $1\leq \ell \leq L$.
  \end{assumption}
  Recall that $|\vec{\alpha}^{(\ell)}|$ represents the number of reactant particles needed for the $\ell$-th reaction.
  As we assume $\abs{\vec{\alpha}}^{\ell} \leq 2$, we obtain three scalings for the three allowable reaction orders:
  \begin{itemize}
  \item $|\vec{\alpha}^{(\ell)}| = 0$ corresponds to a pure birth reaction. By Assumption \ref{Assume:rescaling}, the scaling is $\gamma$; i.e. a larger system size implies more births. In a well-mixed model this would imply that as  $\gamma$ and the initial number of molecules are increased, we maintain a fixed rate \emph{with units of molar concentration per time} for the birth reaction to occur.
  \item $|\vec{\alpha}^{(\ell)}| = 1$ corresponds to a unimolecular reaction. By Assumption \ref{Assume:rescaling}, there's no rescaling as it's linear. We assume the rates of first order reactions are internal processes to particles, and as such independent of the system size.
  \item $|\vec{\alpha}^{(\ell)}| = 2$ corresponds to a bimolecular reaction. By Assumption \ref{Assume:rescaling}, the scaling of reaction kernel is $\gamma^{-1}$. As the system size increases it is harder for two \emph{individual} reactant particles to encounter each other and react.
  \end{itemize}
  See~\cite{IMS:2022} for a more detailed discussion on how such scalings could arise.

  Finally, we made two assumptions on the molar concentration fields, $\mu^j_t$, in~\cite{IMS:2022},
  requiring boundedness of the amount of particles in the system and convergence of the initial conditions in the large-population limit.
  \begin{assumption}\label{Assume:kernalBdd}  for all $t<\infty$
    We assume that the total (molar) population concentration satisfies $\sum_{j = 1}^J \la 1, \mu_t^{{\vec{\zeta}}, j} \ra \leq C_{\circ} < \infty$ for all $t<\infty$, i.e. are uniformly bounded in time by some deterministic constant  $C_{\circ}$. In the remainder we abuse notation and also denote generic constants that depend on this bound by $C_{\circ}$.
  \end{assumption}
  \begin{assumption}\label{Assume:initial}
    We assume that the initial distribution $\mu_0^{{\vec{\zeta}}, j}\to \xi_0^j$ weakly as $\vec{\zeta}\to 0$, where $\xi_0^j$ is a compactly supported measure with finite mass, for all $1\leq j\leq J$.
  \end{assumption}

\subsection{Mean Field Limit}
In \cite{IMS:2022}, we proved the following mean field limit result.  Let $M_F(\mathbb{R}^d)$ be the space of finite measures endowed with the weak topology and $\mathbb{D}_{M_F(\mathbb{R}^d)}[0, T]$ be the space of cadlag paths with values in $M_F(\mathbb{R}^d)$ endowed with Skorokhod topology.

\begin{theorem}{(Mean field large-population limit)} \label{thm:convergence}
Given Assumptions~\ref{Assume:kernalBdd0}-\ref{Assume:initial}, the sequence of measure-valued processes $\{(\mu^{{\vec{\zeta}}, 1}_t,\cdots,\mu^{{\vec{\zeta}}, J}_t)\}_{t\in [0, T]}\in \mathbb{D}_{\otimes_{j=1}^{J}M_F(\mathbb{R}^d)}([0, T])$ is relatively compact in $\mathbb{D}_{\otimes_{j=1}^{J}M_F(\mathbb{R}^d)}([0, T])$ for each $j= 1, 2,\cdots, J$. It converges in distribution to $\{(\bar{\mu}^1_t,\cdots,\bar{\mu}^{J}_{t})\}_{t\in [0, T]} \in C_{\otimes_{j=1}^{J} M_F(\mathbb{R}^d)}([0, T])$ as $\vec{\zeta}\to 0$, respectively being with $j=1,\cdots, J$ the unique solution to the system
\begin{align}
\la f,\bar{\mu}^j_{t}\ra
&
=\la f,\bar{\mu}^{ j}_{0}\ra + \int_{0}^{t} \la (\mathcal{L}_j f)(x), \bar{\mu}_{s}^{ j}(dx) \ra ds\nonumber\\
&
\quad-\sum_{\ell = 1}^{\tilde{L}} \int_{0}^{t} \int_{\tilde{\mathbb{X}}^{(\ell)}}     \frac{1}{\vec{\alpha}^{(\ell)}!}   K_\ell\left(\vec{x}\right)  \left(  \sum_{r = 1}^{\alpha_{\ell j}} f(x_r^{(j)}) \right)\,\lambda^{(\ell)}[\bar{\mu}_{s}](d\vec{x}) \, ds\nonumber\\
&
\quad+\sum_{\ell = \tilde{L} + 1}^L \int_{0}^{t} \int_{\tilde{\mathbb{X}}^{(\ell)}}     \frac{1}{\vec{\alpha}^{(\ell)}!}   K_\ell\left(\vec{x}\right) \left(\int_{\mathbb{Y}^{(\ell)}}    \left( \sum_{r = 1}^{\beta_{\ell j}} f(y_r^{(j)}) \right) m_\ell\left(\vec{y} \, |\, \vec{x} \right)\, d \vec{y} - \sum_{r = 1}^{\alpha_{\ell j}} f(x_r^{(j)}) \right)\,\lambda^{(\ell)}[\bar{\mu}_{s}](d\vec{x}) \, ds.\label{Eq:Limit_EM_formula2}
\end{align}
for each $f\in C^{2}_{b}(\R^{d})$.  Here $\mathcal{L}_j := D_j \Delta_x$ denotes the diffusion operator associated with a particle of type $j$.
\end{theorem}

 Note that since the limit (\ref{Eq:Limit_EM_formula2}) is deterministic, we have actually established convergence in probability since weak convergence to constants implies convergence in probability.

 In \cite{IMS:2021} assumptions on the structure of the chemical-reaction network are posed under which $\bar{\mu}^{j}$ has a density that is well-defined either locally or globally in time. For example, global well-posedness can be shown at least in the cases of $A+B\rightleftarrows C+D$ and $A+B\rightleftarrows C$ (the latter under specific choices for the placement measures), see \cite{IMS:2021}. In general though, only local in time well-posedness can be claimed, see \cite{IMS:2021}. 
\section{Main result: fluctuation theorem} \label{S:mainresult}
In this work, we aim to prove a central limit theorem result for the  particle-based stochastic reaction diffusion (PBSRD) model. We define the fluctuation process by
\begin{equation*}
\Xi^{\vec{\zeta}, j}_t = \sqrt{\gamma} (\mu^{{\vec{\zeta}}, j}_t - \bar{\mu}^j_t), \quad j =1, \cdots, J.
\end{equation*}
In Theorem \ref{T:FluctuationsThm} we show that the signed measure-valued process $\left\{\paren{\Xi^{\vec{\zeta}, 1}_t, \cdots,  \Xi^{\vec{\zeta}, J}_t}, t\in[0,T]\right\}_{\vec{\zeta}\in (0,1)^{2}}$ converges in law to a limit point $\{\paren{\bar{\Xi}^1_t, \cdots, \bar{\Xi}^J_t}, t\in[0, T]\}$ as $\vec{\zeta} \to 0$ in the appropriate space.

As is commonly found for problems involving fluctuation analysis of interacting particle systems, weighted Sobolev spaces are needed to control behavior of the fluctuation process as $\vx \to \infty$. Such spaces have been previously used to study central limit theorems of mean field systems in a variety of studies, include~\cite{Meleard}, \cite{KurtzXiong} and~\cite{FluctuationSpiliopoulosSirignanoGiesecke}. The weighted spaces introduced in~\cite{Meleard}, denoted by $W^{\Gamma,a}_{0}=W^{\Gamma,a}_{0}(\mathbb{R}^{d})$, are what is needed for our analysis. We subsequently denote by $W^{-\Gamma,a}=W^{-\Gamma,a}(\mathbb{R}^{d})$ the dual space to $W^{\Gamma,a}_{0}$. In Section~\ref{S:SobolevSpace} we review the precise definitions and key properties of such spaces.

A number of constants will appear in the sequel. For bookkeeping purposes and for easier reference to the reader we have gathered them and their relations in Assumption \ref{A:SobolevSpaceParameters} below.
\begin{assumption} \label{A:SobolevSpaceParameters}
  In the remainder we will make use of the non-negative integer parameters $\{D,\Gamma,\Gamma_1,a,b\}$ in specifying weighted Sobolev spaces. These parameters are chosen to satisfy the following constraints:
  \begin{enumerate}
    \item $D = 1 + \ceil*{d/2}$
    \item $\Gamma \geq 2D + 2$
    \item $\Gamma_1 = \Gamma - 1 \geq 2D + 1$
    \item $a \geq D$
    \item $b > \max\left\{a+\frac{d}{2}, 2D\right\}$
  \end{enumerate}
\end{assumption}

In order to prove the fluctuations theorem, we make a few additional assumptions:
\begin{assumption}\label{A:AssumptionParameterRelativeRates}
We assume the large population limit $\gamma\to \infty$ is taken such that $\sqrt{\gamma}\eta\to 0$. As such $\eta \to 0$ simultaneously and we write $\vec{\zeta} \to 0$ for the dual limit.
\end{assumption}

\begin{assumption}\label{A:AssumptionRho}
For the density $\rho$, in addition to Assumption~\ref{Assume:measureMrho} we assume $\int_{\mathbb{R}^{d}}|w|^{8D}\rho(w)dw<\infty$.
\end{assumption}

\begin{assumption}\label{A:AssumptionK}
For $\Gamma=2D+2$, we assume that $K_{\ell}\in C^{\Gamma}(\mathbb{X}^{(\ell)})$, that $\max_{j\leq\Gamma}\|\partial_{x}^{(j)}K_{\ell}\|_{\infty}<\infty$, and that $K_{\ell}(\vx)$ is symmetric with respect to permutations in the ordering of components of $\vx$ that correspond to the same species (e.g. to interchanging $x_k^{(j)}$ with $x_{k'}^{(j)}$).
\end{assumption}

Note that this assumption means our results do not rigorously encompass the commonly used Doi bimolecular reaction model for an $\textrm{A} + \textrm{B} \to \textrm{C}$ reaction, in which an \textrm{A} at $x$ and \textrm{B} at $y$ would react with probability per time $\lambda$ when separated by $\varepsilon$ or less, i.e. $K(x,y) = \lambda \ind_{\brac{0,\varepsilon}}(\abs{x-y})$. Similar to how we have introduced $\eta$ as a regularization of the $\delta$-functions that arise in the particle placement densities and then incorporated the limit $\eta \to 0$ in our analysis, we could have introduced a second regularization parameter into the Doi rate function $K(x,y)$. For simplicity of exposition we have chosen to ignore this additional complication, but note that our results would hold for any mollification of such a Doi kernel. We expect these smoothness requirements are an artifact of our method of proof, and not intrinsic to the existence of a limiting fluctuation process correction.

\begin{assumption}\label{A:AssumptionInitialCondition}
Let $\Gamma=2D+2$ and $a=D$. We assume that $\sup_{\vec{\zeta}\in(0,1)^{2}}\sum_{j=1}^{J}\E\|\Xi^{\vec{\zeta}, j}_0\|_{-\Gamma,a} <C$ for some $C<\infty$. In addition, we assume that $\sum_{j=1}^{J}\mathbb{E}\la |\cdot|^{8D},\mu^{\vec{\zeta},j}_{t}\ra<\infty$ for $t \in [0,T]$, and that this moment is well-defined via the density associated with the corresponding forward equation (i.e. the Fock-space representation for the forward equation, see Appendix~\ref{S:FockSpaceExpect}).
\end{assumption}
 Note that Assumption \ref{A:AssumptionInitialCondition} does not assume uniform boundedness of $\sum_{j=1}^{J}\mathbb{E}\la |\cdot|^{8D},\mu^{\vec{\zeta},j}_{t}\ra<\infty$ with respect to $\vec{\zeta}\in(0,1)^{2}$. It only assumes that for each $\vec{\zeta}\in(0,1)^{2}$, the spatial moment is finite. As a matter of fact we shall prove in Section \ref{S:MomentBounds} that these moments (assuming that they are well defined per Assumption \ref{A:AssumptionInitialCondition}) are indeed uniformly bounded in $\vec{\zeta}\in(0,1)^{2}$.

The assumed existence and finiteness of the spatial moment, $\sum_{j=1}^{J}\mathbb{E}\la |\cdot|^{8D},\mu^{\vec{\zeta},j}_{t}\ra$, is motivated by physical considerations. Particles move by diffusion, and we have assumed the underlying reaction network is one such that the number of particles stays uniformly bounded for all times. Moreover, the most commonly used reaction kernel $K(x,y)$, the Doi kernel, has compact support, which would then lead to $\rho(x)$ also having compact support for models that are consistent with detailed balance of reversible reactions~\cite{ZI:2022}. Combined with the standard delta-function based placement models described in the previous section, reactions should not induce particles to jump more than a bounded distance with each occurrence. As the number of reaction occurrences for typical reaction systems studied in practice should be bounded over any finite interval, we would not anticipate that the reaction components of the model could lead to unbounded spatial moments. That said, it is an open problem to rigorously establish the existence of spatial moments, which to our knowledge has not been investigated for the class of PBSRD models considered in this work. (In actuality, there is limited work establishing the boundedness of simpler population number moments such as $\E \la 1, \mu_{t}^{\vec{\zeta},j} \ra^{p}$, see for example \cite[Chapter 7]{BM:2015} and  the very recent work~\cite{C:2023} for some results in particular reaction systems.)

For reasons that will become clearer in Subsection \ref{SS:PreliminaryCalculations} we have the following definition.
\begin{definition}\label{def:Delta}
For any $\ell$-th reaction, particle distribution $\nu = \sum_{j = 1}^{J} \nu^j\delta_{S_j} \ \in M(\hat{P})$ and particle fluctuation distribution $\Xi = \sum_{j = 1}^{J} \Xi^j\delta_{S_j} \ \in M(\hat{P})$ , define the $\ell$-th reactant mapping  $\Delta^{(\ell)} [\, \cdot \,,  \cdot \,] : M(\hat{P})\times M(\hat{P}) \to M( \mathbb{X}^{(\ell)})$ via
\begin{equation}\label{eq:reactantmappingdef}
  \Delta^{(\ell)} [\nu,\, \Xi] =
  \begin{cases}
    \Xi^k(x)       & \quad \text{if the $\ell$-th reaction is of the form } S_k \to \cdots\\
    \Xi^k(x)\nu^r(y) + \nu^k(x) \Xi^r(y)  & \quad \text{if the $\ell$-th reaction is of the form } S_k + S_r \to \cdots
  \end{cases}
\end{equation}
\end{definition}

\begin{theorem}\label{T:FluctuationsThm}
Assume $T < \infty$. Given Assumptions~\ref{Assume:kernalBdd0}-\ref{Assume:initial} as well as Assumptions~~\ref{A:SobolevSpaceParameters}-\ref{A:AssumptionInitialCondition}, the sequence $\left\{\paren{\Xi^{\vec{\zeta}, 1}_t, \cdots,  \Xi^{\vec{\zeta}, J}_t}, t\in[0,T]\right\}_{\vec{\zeta}\in (0,1)^{2}}$ is relatively compact in $D_{(W^{-\Gamma,a}(\R^d))^{\otimes J}}([0, T])$. For any subsequence of this sequence, there is a further subsubsequence that converges in law to the distribution-valued stochastic process $\{\paren{\bar{\Xi}^1_t, \cdots, \bar{\Xi}^J_t}, t\in[0, T]\}$ as $\vec{\zeta} \to 0$  satisfying in $W^{-(2+\Gamma),a}$ the evolution equation
\begin{align}\label{eq:fluctlimitweakform}
\la f,\bar{\Xi}^j_t\ra&
=\la f,\bar{\Xi}^j_0\ra + \int_{0}^{t} \la (\mathcal{L}_jf)(x) , \bar{\Xi}^j_s (dx)\ra ds  + \bar{M}^{j}_t(f) \notag \\
&-\sum_{\ell = 1}^{\tilde{L}} \int_{0}^{t} \int_{\tilde{\mathbb{X}}^{(\ell)}}     \tfrac{1}{\vec{\alpha}^{(\ell)}!}   K_\ell\left(\vec{x}\right)  \left(  \sum_{r = 1}^{\alpha_{\ell j}} f(x_r^{(j)}) \right)\, \Delta^{(\ell)} [\bar{\mu}_s, \bar{\Xi}_s](d\vec{x})  \, ds\\
&+ \sum_{\ell = \tilde{L} + 1}^L \int_{0}^{t} \int_{\tilde{\mathbb{X}}^{(\ell)}}     \tfrac{1}{\vec{\alpha}^{(\ell)}!}   K_\ell\left(\vec{x}\right) \left(\int_{\mathbb{Y}^{(\ell)}}    \left( \sum_{r = 1}^{\beta_{\ell j}} f(y_r^{(j)}) \right) m_\ell\left(\vec{y} \, |\, \vec{x} \right)\, d \vec{y} - \sum_{r = 1}^{\alpha_{\ell j}} f(x_r^{(j)}) \right)\Delta^{(\ell)}[\bar{\mu}_s, \bar{\Xi}_s](d\vec{x}) \, ds \notag
\end{align}
for $j = 1, \cdots, J$ and any $f\in W^{2+\Gamma,a}_{0}(\mathbb{R}^{d})$. Here $\paren{\bar{M}^1_t, \cdots, \bar{M}^J_t}$ is a distribution-valued mean-zero Gaussian martingale with marginal variance
covariance structure for $0\leq s,t\leq T$ and $f,g\in W^{2+\Gamma,a}_{0}(\mathbb{R}^{d})$
\begin{equation}\label{eq:fluctnoiselimitweakform}
\begin{aligned}
\Cov[\bar{M}^j_t(f), \bar{M}^k_s(g)] &=  \sum_{\ell = 1}^{\tilde{L}} \int_{0}^{s\wedge t} \int_{\tilde{\mathbb{X}}^{(\ell)}}     \frac{1}{\vec{\alpha}^{(\ell)}!}   K_\ell\left(\vec{x}\right)  \left(  \sum_{r = 1}^{\alpha_{\ell j}} f(x_r^{(j)}) \right) \left(  \sum_{r = 1}^{\alpha_{\ell k}} g(x_r^{(k)}) \right)\,\lambda^{(\ell)}[\bar{\mu}_{s'}](d\vec{x}) \, ds'\\
&\quad+\sum_{\ell = \tilde{L} + 1}^L \int_{0}^{s\wedge t} \int_{\tilde{\mathbb{X}}^{(\ell)}}     \frac{1}{\vec{\alpha}^{(\ell)}!}   K_\ell\left(\vec{x}\right) \left(\int_{\mathbb{Y}^{(\ell)}}    \left( \sum_{r = 1}^{\beta_{\ell j}} f(y_r^{(j)}) - \sum_{r = 1}^{\alpha_{\ell j}} f(x_r^{(j)}) \right)\right.\\
&\qquad \left. \times \left( \sum_{r = 1}^{\beta_{\ell k}} g(y_r^{(k)}) - \sum_{r = 1}^{\alpha_{\ell k}} g(x_r^{(k)}) \right) m_\ell\left(\vec{y} \, |\, \vec{x} \right)\, d \vec{y} \right)\,\lambda^{(\ell)}[\bar{\mu}_{s'}](d\vec{x}) \, ds'\\
&+\int_{0}^{s\wedge t} \la 2D_{j} \frac{ \partial f}{\partial Q}(x)\frac{ \partial g}{\partial Q}(x), \bar{\mu}_{s'}^j(dx)\ra 1_{\{k=j\}}\,ds'.
\end{aligned}
\end{equation}

Finally, the limiting stochastic evolution equation has a unique solution in $\paren{W^{-(2+\Gamma),a}}^{\otimes J}$ for $t \in \brac{0,T}$, and thus the limit accumulation point $\{\paren{\bar{\Xi}^1_t, \cdots, \bar{\Xi}^J_t}, t\in[0, T]\}$ is unique.
\end{theorem}

Here by uniqueness for (\ref{eq:fluctlimitweakform}) we mean that for a given $(\bar{M}^{1}_t(f),\cdots, \bar{M}^{J}_t(f))$ there is a unique $\{(\Xi^{1}_{\cdot},\cdots, \Xi^{J}_{\cdot})\}$ in $\paren{W^{-(2+\Gamma),a}(\R^d)}^{\otimes J}$ (see Theorem \ref{T:Uniqueness} for the precise statement.) We also note that equations (\ref{eq:fluctlimitweakform})-(\ref{eq:fluctnoiselimitweakform})  characterize the limit of $\left\{\paren{\Xi^{\vec{\zeta}, 1}_t, \cdots,  \Xi^{\vec{\zeta}, J}_t}, t\in[0,T]\right\}_{\vec{\zeta}\in (0,1)^{2}}$ in a weak sense, i.e., in terms of appropriate test functions. Due to the presence of the reaction terms, the general existence of a strong formulation is an open problem. However, as we demonstrate in Subsection~\ref{SS:FluctuationsProcessModel_Discretization}, one can at least identify formal stronger representations for specific reaction systems.

\section{On the appropriate Sobolev space}\label{S:SobolevSpace}
We study convergence in Sobolev spaces, see for example \cite{Adams} for a general exposition. The weighted Sobolev spaces introduced in~\cite{Meleard} are designed to control growth as $x \to \infty$, and are just what is needed for the problem at hand. In this section, we recall their definitions and main properties.

For any integers $\Gamma,a\in\mathbb{N}$, consider the space of real valued functions $f$ with partial derivatives up to order $\Gamma$ which satisfy
\begin{eqnarray}
\norm{f}_{\Gamma,a} = \bigg{(} \sum_{|k| \leq \Gamma} \int_{\R^d}  \frac{\big{|}  D^k f(x) \big{|}^2}{1+|x|^{2a}} dx\bigg{)}^{1/2} < \infty.\notag
\end{eqnarray}
Define  the  space $W_0^{\Gamma,a}$ as the closure of functions of class $C_0^{\infty}$ in the norm defined above, where $C_0^{\infty}$ is the space of all functions in $C^{\infty}$ with compact support.  $W_0^{\Gamma,a}$ is a Hilbert space (see Theorem 3.5 and Remark 3.33 in \cite{Adams} and also \cite{Meleard} for the weighted version) and has the inner product
\begin{eqnarray}
\la f, g \ra_{\Gamma,a} =  \sum_{|k| \leq \Gamma} \int_{\R^d} \frac{D^k f(x) D^k g(x)}{1+|x|^{2a}} dx.\notag
\end{eqnarray}
When $\Gamma,a=0$, we write $\la f, g \ra_{0}=\la f, g \ra$. $W^{-\Gamma,a}$ denotes the dual space of $W_0^{\Gamma,a}$ that is equipped with the norm
\begin{eqnarray}
\norm{f}_{-\Gamma,a} = \sup_{\{g \in W_0^{\Gamma,a}\, |\, g \neq 0\}} \frac{ \big{|} \la f, g \ra \big{|} }{ \norm{g}_{\Gamma,a}}.\notag
\end{eqnarray}

Let $C^{\Gamma,a}$ denote the space of continuous functions $f$ that have continuous partial derivatives up to order $\Gamma$ such that
\begin{equation*}
  \lim_{|x|\rightarrow\infty}\frac{|D^{j}f(x)|}{1+|x|^a}=0, \text{ for all }j\leq \Gamma.
\end{equation*}
The norm of this space is
\begin{equation*}
\|f\|_{C^{\Gamma,a}}=\sum_{j\leq \Gamma}\sup_{x\in \R^d}\frac{|D^{j}f(x)|}{1+|x|^a}.
\end{equation*}
We refer the interested to \cite{Meleard} for details on this class of weighted spaces.

A few properties that we will use in this paper include the embeddings that
\begin{equation*}
W^{m+j,a}_{0}\hookrightarrow C^{j,a}, \quad m>d/2, j\geq 0, a\geq 0.
\end{equation*}
and
\begin{equation*}
W^{m+j,a}_{0}\hookrightarrow W^{j,a+b}_{0}, \quad m>d/2, j\geq 0, a\geq 0, b>d/2.
\end{equation*}
A key property we will subsequently make use of is that the latter is Hilbert-Schmidt, and implies that the embedding
\begin{equation*}
W^{-j,a+b}\hookrightarrow W^{-m-j,a}, \quad m>d/2, j\geq 0, a\geq 0, b>d/2
\end{equation*}
is also Hilbert-Schmidt.

The need to introduce weights becomes apparent when we derive the necessary a-priori bounds for $\{\Xi^{\zeta, j}_{t}\}$ that then lead to the tightness claims in Section \ref{S:Tightness}. A representative calculation where the need for weights is clear is the bound in (\ref{Eq:MapingPropertyD}). As illustrated there, by providing control over the $\abs{x} \to \infty$ behavior of integrands, the weights allow us to leverage uniform moment and kernel bounds when estimating integrals over $\R^{d}$.


\section{Numerical results} \label{S:numerics}
In this section, we numerically compare the mean field with fluctuation corrections SPIDE system to the underlying PBSRD model for a 1D ($d=1$) model of the three species
\begin{equation*}
  A + B \underset{\mu}{\stackrel{\lambda K(x,y)}{\rightleftharpoons}} C
\end{equation*}
reversible reaction. The model we consider corresponds to that presented in Section~\ref{S:summaryexample}, with the specific choice of reaction kernels and placement densities summarized in Section~\ref{S:numerics_example_def}.

We start by describing the model problem and discretization schemes for the SPIDE and PBSRD models in detail before giving numerical results. Note that our notation in this section differs slightly from the more abstract notation used previously; for example, we parametrize the particle types by uppercase letters $A,B,C$ instead of natural numbers $S_1,S_2,S_3$.

We demonstrate that for the total molar mass (i.e., the integral or $L^1$-norm of the molar concentration field) of the type $C$ particles, the fluctuation process gives an increasingly accurately approximation of the PBSRD variance as $\gamma$ increases. For $\gamma=8000$, the largest value of $\gamma$ that we consider, the variances for the two models agree to statistical error. Since the concentration field is a Gaussian process, the distribution of the molar mass is Gaussian at each time $t$. We observe empirically that the PBSRD molar mass is also approximately Gaussian for large $\gamma$, and so, given the close agreement of the means and variances, we also see close agreement between the entire statistical distribution for the molar mass for large enough $\gamma$.

Throughout the following section, we use the terms PBSRD and particle model interchangeably. We refer to the jump process discretization of the particle model that we actually simulate as the CRDME (see discretization section). Finally, we subsequently refer to the solution of the combined mean field with fluctuation correction SPIDE model as the fluctuation process.

\subsection{Description of model problem} \label{S:numerics_example_def}
We restrict the reaction system to the periodic domain $\Omega = (0,2\pi)$. The function $\lambda K(x,y)$ denotes the forward reaction kernel, i.e., the probability per unit time that a forward reaction $A+B \to C$ occurs given an $A$ at $x$ and a $B$ at $y,$ where $K(x,y)$ is a normalized function giving the spatial distribution and $\lambda > 0$ is the total reaction rate. Here we use the Gaussian
\begin{equation*}
K(x,y) = \frac{1}{Z}\frac{e^{-\frac{|x-y|^2}{2\varepsilon^2}}}{\sqrt{2\pi\varepsilon^2}},
\end{equation*}
with $\varepsilon$ a parameter determining the kernel width, $|x-y|$ the periodic distance between $x,y\in\Omega$, i.e.
\begin{equation*}
  |x-y| = \min\{|x-y|,2\pi-|x-y|\},
\end{equation*}
and $Z$ a normalization constant,
\begin{equation*}
Z = \frac{1}{\sqrt{2\pi\varepsilon^2}} \int_{0}^{2 \pi}  e^{-\frac{|x-y|^2}{2\varepsilon^2}} dx.
\end{equation*}
For the placement density $m(z | x,y)$, the probability of placing a $C$ at $z$ given a reaction between an $A$ at $x$ and $B$ at $y$, we use the combination of $\delta$-functions \begin{equation}\label{eq:deltafunctionplacement}m(z|x,y) = \frac{1}{2}\delta(x - z) + \frac{1}{2}\delta(y-z),\end{equation} meaning that in the event of an $A+B\to C$ reaction the product $C$ is placed at the location of the $A$ or the $B$ each with probability $\frac{1}{2}$. Note, in the simulations that follow we do not regularize this density, so that $\eta = 0$ in the notation of previous sections. This is common when simulating particle models, where such a regularization is not needed (it's use in the MVSP formulation is solely to ensure that representation of the model is well-defined as described in Section~\ref{S:PreliminaryResults}).

Reversibility (equivalently detailed balance) implies that the reverse reaction in which a $C$ at $z$ unbinds into an $A$ at $x$ and a $B$ at $y$ occurs with rate $\mu K(x,y)m(z|x,y)$, where $\mu > 0$ is a constant~\cite{ZI:2022}. Integrating over $x$ and $y$, we find that the probability per time of an unbinding reaction for a $C$ at $z$ is given by
\begin{equation*}
\mu \int_{\Omega^2} K(x,y) m(z|x,y)dxdy = \frac{\mu}{2}\int_\Omega K(z,y) dy + \frac{\mu}{2}\int_\Omega K(x,z) dz = \mu.
\end{equation*}
Similarly, we can derive that given an unbinding reaction of a $C$ at $z$, either the $A$ or $B$ is placed at $z$ with probability $\frac{1}{2}$ and the other molecule is placed with the Gaussian distribution $K(x,z)$ (respectively $K(z,y)$) centered at $z$.

In the results that will follow, we have used the parameters $D_A = 1$, $D_B=.5$, and $D_C=.1$. We also set $\lambda = 1$, $\mu = .05$, and $\varepsilon=2^{-7}$. We will also need an initial condition, which we set to be proportional to \begin{align*}
    A(x,0) &= e^{-10(x-1)^2},\\
    B(x,0) &= e^{-10(x-2)^2},\\
    C(x,0) &= 0.
\end{align*}
We use this initial condition for the PIDEs satisfied by the mean-field concentration fields, and to generate initial conditions for the particle model (see the next section). The initial conditions for the SPIDEs that correspond to the fluctuation corrections, a scaled difference between the particle and mean-field initial conditions, were therefore zero.

\subsection{Discretization of particle model.}
We numerically study the particle model using the convergent reaction-diffusion master equation (CRDME), a convergent spatial discretization of the forward Kolmogorov equation associated with the PBSRD model~\cite{I:2013,IZ:2018}. The CRDME corresponds to the forward equation for a system of continuous-time jump processes on a mesh, and we therefore simulate the particle system via simulations of these jump processes using optimized versions of the stochastic simulation algorithm (SSA), also known as Gillespie's method or Kinetic Monte Carlo~\cite{G:1976}. The PBSRD particles' Brownian motions are then approximated by continuous-time random walks on a grid, and their reactive interactions by jump processes that depend on the relative positions of reactants on the mesh. As discussed in~\cite{I:2013,IZ:2018}, statistics obtained from simulations of the CRDME should then converge to those of the underlying PBSRD model as the mesh spacing is taken to zero. For these simulations we use a uniform mesh with nodes $\{x_i\}_{i=1}^N \subseteq \Omega$, where $x_i = (i-1)h$, $i=1,...,N$, and $h = \frac{2\pi}{N}$. We denote the compartments, or voxels, that particles hop between by $V_i, i=1,...,N$, given by $V_i = (x_i-\frac{h}{2},x_i+\frac{h}{2})$.

As we require integer numbers of particles for each simulation, we converted from the molar concentration field at $x_i$ to the number of particles, $a_i$, in voxel $V_i$ as
  \begin{algorithmic}[1]
    \State{$m := \gamma h A(x_i,0)$, $\bar{m} := \lfloor m \rfloor$}
    \State{$r \sim \Bernoulli(m - \bar{m})$.}
    \State{$a_i = \bar{m} + r$.}
\end{algorithmic}
This choice ensured the average number of particles of species $A$ in the $i$th voxel is $\gamma h A(x_i,0)$, consistent with the mean-field model's initial condition. The mean number of particles of species $A$ in our initial condition is therefore $\gamma \sum_{i=1}^N h A(x_i,0)$, which is proportional (but not equal) to $\gamma$ for fixed $h$. The initial number of particles of species $B$ and $C$ were chosen similarly.

The diffusion rate for each species $s\in \{A,B,C\}$ from one voxel $V_i$ to either of its neighbors is given by $\frac{D_s}{h^2}$, where $h$ is the distance between the centers of neighboring voxels. Pair reaction rates between a molecule of type $A$ in voxel $V_i$ and a molecule of type $B$ in voxel $V_j$ are computed as $\lambda K_\gamma(x_i,y_j) = \frac{K(x_i,y_j)}{\gamma}$. The resulting $C$ is placed in voxel $V_i$ or $V_j$ with probability $\frac{1}{2}$ each. When a $C$ at location $V_k$ unbinds, an $A$ or $B$ is placed in $V_k$ with probability $\frac{1}{2}$. The location of the other product particle is determined by the (unnormalized) discrete distributions $\{K(z_k,y_j)\}_{j=1}^N  = \{K(x_i,z_k)\}_{i=1}^N$ over the voxel sites. The total rate of unbinding in the CRDME model then is given by the discrete marginal integral of the reaction kernel,
\begin{equation*}
\mu\sum_{i=1}^N hK(x_i,z_k),
\end{equation*}
which for our chosen parameters is essentially $\mu$ (i.e. to numerical precision). Specifying the rates with the midpoint rule quadrature scheme as we have just described preserves detailed balance and also $O(h^2)$ convergence for statistics as compared to the version of the CRDME presented in~\cite{IZ:2018}.  We note that the product-particle placement mechanisms in our discrete model directly simulate the unmollified $\delta$-function placement mechanism \eqref{eq:deltafunctionplacement} from the continuum model.

The CRDME is simulated as a continuous-time jump process using a variant of the direct Gillespie stochastic simulation algorithm (SSA) in which we view the simulation mesh as the top level (also referred to as the fine grid) of a nested hierarchy of uniform meshes $$\Omega_I \supseteq \Omega_{I-1} \supseteq ... \supseteq \Omega_1,$$ so that $|\Omega_i| = 2^i$. We keep track of the total reaction rate for each subvolume on each mesh throughout the simulation, which allows us to identify the next fine grid subvolume where a reaction occurs in logarithmic time using a binary search. When a reaction or a diffusion event occurs, we update the total reaction rate for each subvolume on each mesh that contains a fine grid subvolume involved in the reaction. The complexity for these updates is also logarithmic, although the total cost of rate updates after a $B$ molecule reacts or diffuses will be $O(N\log N)$ (with a small constant) due to the cost of updating the reaction rate for all $A$ in the range of the reaction kernel.

In our CRDME simulations, we chose $I = 10$ so that the simulations used $N = 2^{10}$ voxels. This choice was informed by looking at successive differences of statistics for simulations at different mesh resolutions to check for convergence, and also by the requirement for the mesh spacing $h = \frac{2\pi}{N}$ to be below the kernel spacing $\varepsilon$ in order to achieve asymptotic $O(h^2)$ convergence rates.  In these informal experiments, we found the required value of $h$ to be independent of the system size $\gamma$.

\subsection{Fluctuation Process Model and Discretization.}\label{SS:FluctuationsProcessModel_Discretization}\smallskip
The mean field model for the reversible reaction is given by the system of PIDEs
\begin{align*}
    \frac{\partial A}{\partial t}(x,t) &= D_A\Delta A(x,t) - \lambda (K*B)(x,t)A(x,t) + \frac{\mu}{2} [(K*C)(x,t) +  C(x,t)],\\
    \frac{\partial B}{\partial t}(y,t) &= D_B\Delta B(y,t) - \lambda (K*A)(y,t)B(y,t) + \frac{\mu}{2}[(K*C)(y,t) + C(y,t)],\\
    \frac{\partial C}{\partial t}(z,t) &= D_C\Delta C(z,t) - \mu C(z,t) + \frac{\lambda}{2}[(K*B)(z,t)A(z,t) + (K*A)(z,t)B(z,t)],
\end{align*}
where $A$, $B$, and $C$ are the mean field molar concentrations for the corresponding particle types, and we define the integral operator $(K*f)(x,t)$ for a function $f(x,t)$ by
\begin{equation*}
  (K * f)(x,t) = \int_{\Omega} K(x,y) f(y,t) \, dy.
\end{equation*}
This model has been previously explored and compared to CRDME simulations of the PBSRD model in~\cite{IMS:2021}.

We denote by $\Bar{A}$, $\Bar{B}$, and $\Bar{C}$ the fluctuation corrections that correspond to the limiting process in Theorem~\ref{T:FluctuationsThm}. By augmenting the mean field solution with the fluctuation corrections, we obtain the fluctuation processes
\begin{align*}
  A^\gamma(x,t) &:= A(x,t) + \frac{1}{\sqrt{\gamma}}\bar{A}(x,t), \\
  B^\gamma(y,t) &:= B(y,t) + \frac{1}{\sqrt{\gamma}}\bar{B}(y,t), \\
  C^\gamma(z,t) &:= C(z,t) + \frac{1}{\sqrt{\gamma}}\bar{C}(z,t),
\end{align*}
which we expect to provide an improved approximation of the particle model for finite $\gamma$ compared to the mean field model. We note that the fluctuation processes are Gaussian processes.

From the rigorous weak MVSP equations~\eqref{eq:fluctlimitweakform} and~\eqref{eq:fluctnoiselimitweakform} we obtain by integration by parts the following (formal) strong form representation the fluctuation corrections should satisfy
\begin{align*}
    \frac{\partial \bar{A}}{\partial t}(x,t) &= D_A\Delta \bar{A}(x,t) - \lambda [(K*B)(x,t)\bar{A}(x,t) + (K*\bar{B})(x,t)A(x,t)]  + \frac{\mu}{2}\brac{(K*\bar{C})(x,t) + \bar{C}(x,t)} + \xi^A_t,\\
    \frac{\partial \bar{B}}{\partial t}(y,t) &= D_B\Delta\bar{B}(y,t) - \lambda[(K*A)(y,t)\bar{B}(y,t) + (K*\bar{A})(x,t)B(x,t)  + \frac{\mu}{2}\brac{(K*\bar{C})(x,t) + \bar{C}(x,t)} + \xi^B_t,\\
    \frac{\partial \bar{C}}{\partial t}(z,t) &= D_C\Delta \bar{C}(z,t) - \mu \bar C(z,t) + \frac{\lambda}{2} \big[(K*B)(z,t)\bar{A}(z,t) + (K*A)(z,t)\bar{B}(z,t) + (K*\bar{A})(z,t)B(z,t) \\
    &\quad + (K*\bar{B})(z,t)A(z,t)\big] + \xi^C_t.
\end{align*}
Here the noise terms, $\xi_t^A,\xi^B_t,\xi^C_t$, are mean zero Gaussian processes with covariance structure
\begin{equation*}
\begin{aligned}
\Cov[\la f, \xi^A_t\ra\la g, \xi^A_t\ra] &= \int_0^t 2D_A\la \frac{\partial{f}}{\partial Q} A(s), \frac{\partial{g}}{\partial Q} \ra + \lambda  \la A(s)f,(K*B)(
s)g \ra + \frac{\mu}{2} [\la f(C * K)(s), g\ra + \la C(s)f, g\ra ]ds\\
\Cov[\la f, \xi^C_t\ra\la g, \xi^C_t\ra] &= \int_0^t 2D_C\la \frac{\partial{f}}{\partial Q}  C(s),   \frac{\partial{g}}{\partial Q}\ra + \mu \la fC(s), g\ra +\frac{\lambda}{2} [\la f A(s),g(K*B)(s)\ra  + \la f B(s),g(K*A)(s) \ra ds\\
\Cov[\la f, \xi^A_t\ra\la g, \xi^B_t\ra] &= \int_0^t \lambda\la (K*fA)(s),gB(s)\ra + \frac{\mu}{2} [\la (K * fC)(s), g\ra + \la (K * gC)(s), f\ra]ds\\
\Cov[\la f, \xi^A_t\ra\la g, \xi^C_t\ra] &= -\int_0^t \Big( \tfrac{\lambda}{2}[\la fA(s), g(K*B)(s) \ra + \la fA, (K * gB)(s) \ra] + \tfrac{\mu}{2}[\la K*f, C(s)g \ra + \la fC(s), g \ra] \Big) ds,
\end{aligned}
\end{equation*}
with the remaining covariances 
following by symmetry.  In deriving those equations we assume densities associated with the fluctuation measures exist. Note that in the preceding formula we have suppressed writing the spatial dependence of the mean field solutions and test functions, and we have denoted by $\la . , . \ra$ the $L^2$ inner-product on $\Omega$.

To compute the SPIDE solution first requires the mean field solution. We use Fourier spectral discretizations for both problems. For the mean field solver, we use a total of $2^9$ basis functions $\{e^{inx}\}_{n=0}^{2^9-1}$ to represent the concentration fields for each of the three species. We use fast Fourier transforms provided by \texttt{Scipy} to convert between Fourier representations and values at collocation points and compute all integrals using the midpoint rule on intervals centered at collocation points.

Solving for the fluctuation process is more computationally expensive because of the covariance matrix computation and factorization, and the large number of trials required to resolve statistics to sufficient accuracy for convergence testing. We therefore used 61 basis functions $\{\sin(nx),\cos(nx),1\}_{n=1}^{30}$ for approximating the fluctuation corrections. Quadratures were computed at the original $2^{9}$ collocation points for integrals involving the mean field solution using the midpoint rule.

For the time discretization we used a one step IMEX Euler method for both the mean field and fluctuation processes, treating the reaction and noise terms explicitly and the diffusion terms implicitly. In both solvers we used a time step of .001.

\subsection{Numerical results}
In this section we present numerical results for the reversible $A+B\leftrightarrows C$ reaction using the CRDME approximation to the particle model, and Fourier spectral approximation to the fluctuation process SPIDEs, as previously described.  We  examine several statistics that are not available from the (deterministic) mean-field approximation, to illustrate the additional information gained by having the fluctuation corrections. To generate the empirical variances and distributions, we used 280,000 simulations for the CRDME and 1,600,000 simulations for the fluctuation process for each individual set of parameter values.

We first compare the variance of the ($\sqrt{\gamma}$-scaled) molar mass within the interval for the fluctuation process at various times $t$. Denote by $V_{SPIDE}(t)$ the variance of the total (scaled) molar mass of $C$ molecules from the fluctuation process, i.e.
\begin{equation*}
  V_{SPIDE}(t) = \gamma \, \avg{\paren{\int_{\Omega} \paren{C^{\gamma}(x,t) - C(x,t)} \, dx}^2} = \avg{\paren{\int_{\Omega} \bar{C}(x,t) \, dx}^2}.
\end{equation*}
We denote by $V_{CRDME}^\gamma(t)$ the raw (unscaled) variance of the CRDME molar mass for a given $\gamma$ at time $t$, i.e.
\begin{equation*}
  V_{CRDME}^{\gamma}(t) = \avg{\paren{\int_{\Omega} \paren{C^{\gamma}_{CRDME}(x,t) - \avg{C^{\gamma}_{CRDME}(x,t)}} \, dx}^2},
\end{equation*}
where $C^{\gamma}_{CRDME}(x,t)$ is the molar concentration field process in the CRDME. Then $\gamma V_{CRDME}^\gamma(t)$ denotes the rescaled variance of the CRDME molar mass, which is directly comparable to $V_{SPIDE}(t)$. That is, we expect $\gamma V_{CRDME}^\gamma(t) \to V_{SPIDE}(t)$ as $\gamma \to \infty$. Directly comparing the unscaled variances would be more difficult as we expect the unscaled CRDME variance $V_{CRDME}^\gamma(t) \to 0$ and unscaled SPIDE variance $V_{SPIDE}/\gamma \to 0$ as $\gamma \to 0$.


\begin{figure}[tb]
  \centering
  \includegraphics[width=.45\textwidth]{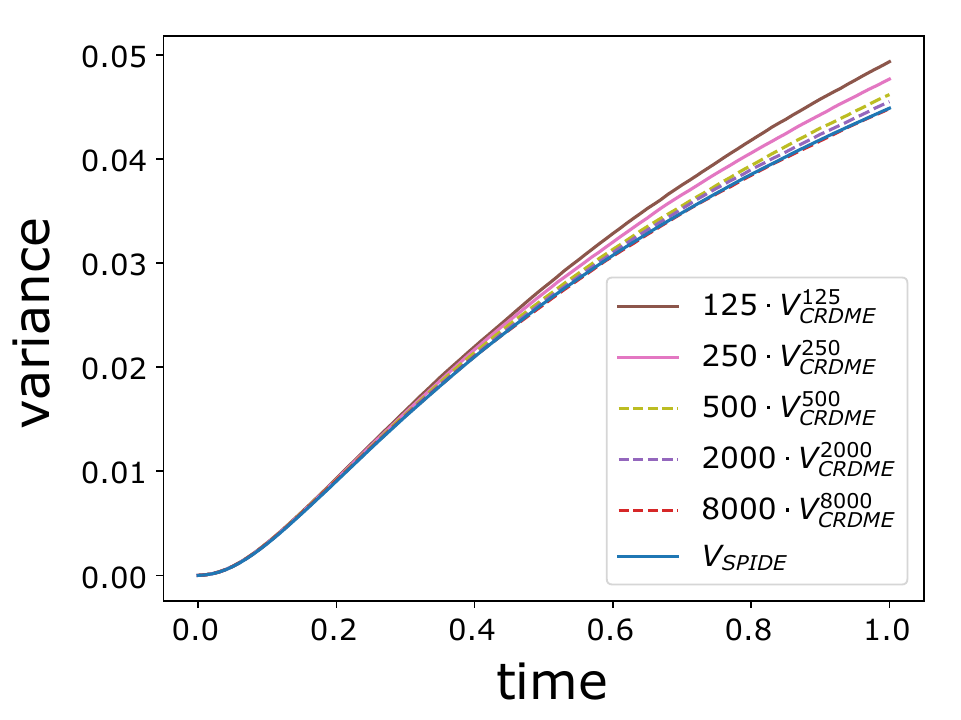}
  \includegraphics[width=.45\textwidth]{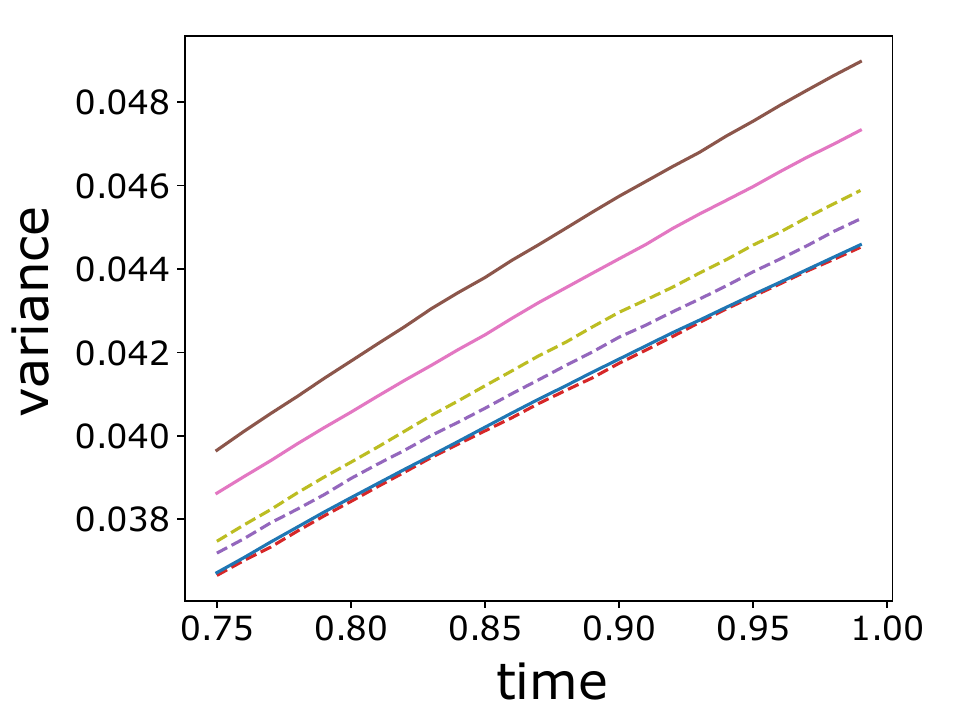}\\
  \includegraphics[width=.45\textwidth]{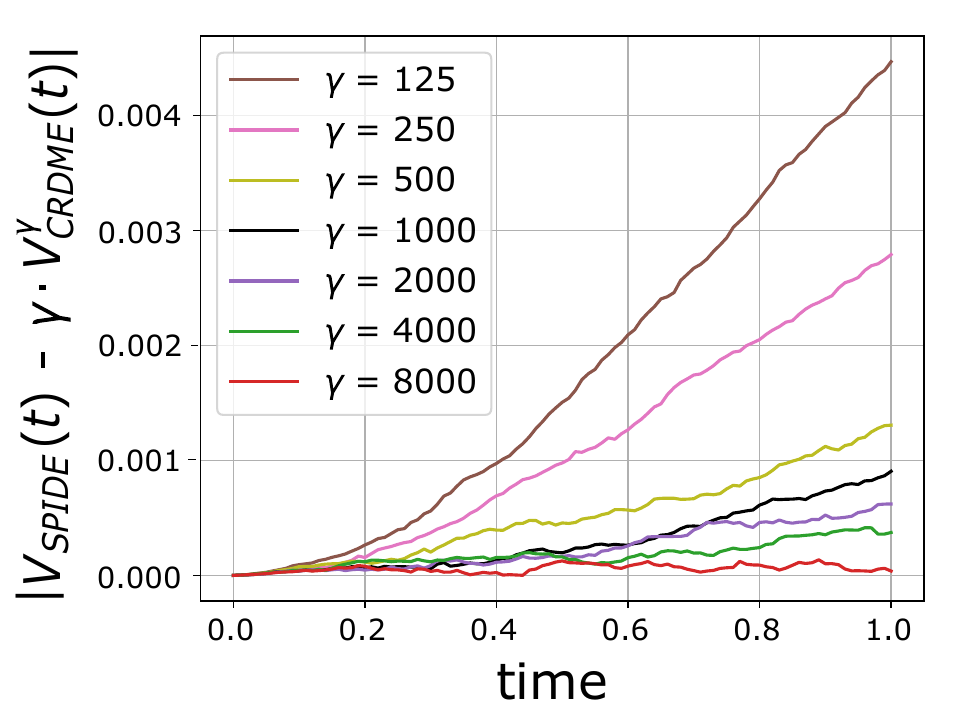}
  \includegraphics[width=.45\textwidth]{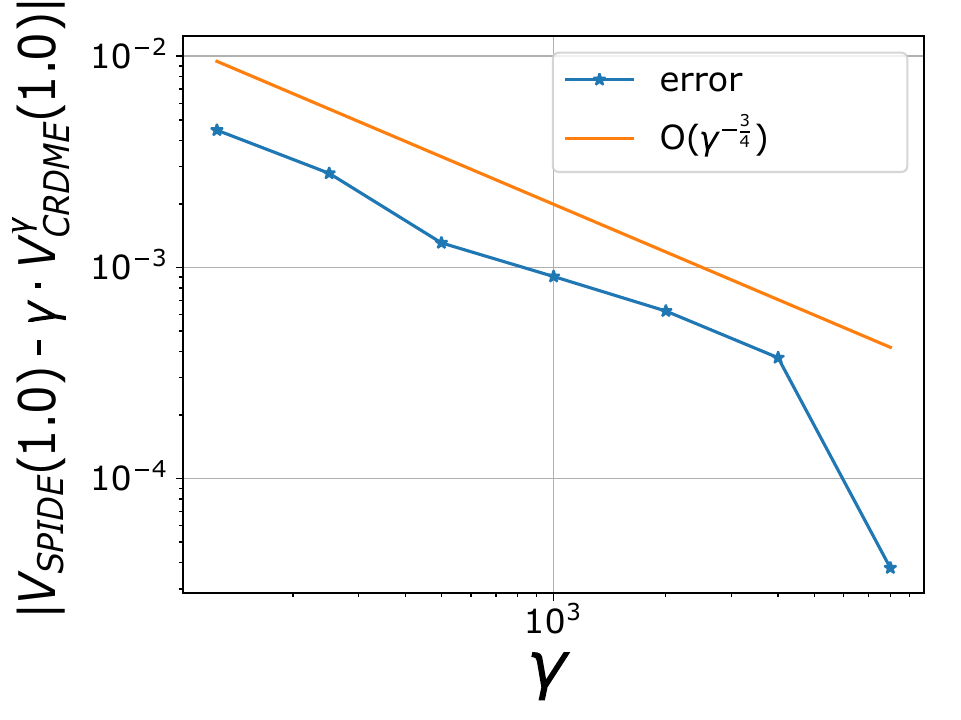}
  \caption{Scaled variance for fluctuation process, $V_{SPIDE}(t)$, compared to rescaled CRDME variances $\gamma V^{\gamma}_{CRDME}(t)$ for $t \in \brac{0,1}$ (top left) and $t \in \brac{.75,1}$ (top right, zoomed-in plot of the last quarter of the left panel). The bottom row shows the difference between these two statistics vs. time (bottom left) and
  at $t=1$ vs. $\gamma$ (bottom right). The latter illustrates how $\gamma V^{\gamma}_{CRDME}(t) \to V_{SPIDE}(t)$ as $\gamma \to \infty$.}
  \label{fig:variances}
\end{figure}
The top panels in Figure \ref{fig:variances} show the rescaled CRDME variances over the whole time interval $[0,1]$ (top left) and over the shorter time interval $[.75,1]$ (top right), alongside the (scaled) SPIDE variance $V_{SPIDE}$. The top right plot is a zoomed-in version of the top left plot, to show more clearly the increasing agreement between the variances computed of the particle model with the variances computed by the SPDE as $\gamma$ increases. We see that by $\gamma=8000$, the variances agree within sampling error. The bottom panels show the difference $|V_{SPIDE}(t) - \gamma V_{CRDME}^\gamma(t)|$ over time (left) and at $t=1$ (right). We see from the left bottom panel that over the entire interval $[0,1]$ the difference between the variances decreases as $\gamma$ increases. In the right bottom panel we observe a noisy but roughly consistent rate of convergence as we increase $\gamma$. The orange line represents an $O(\gamma^{-\frac{3}{4}})$ convergence rate, which is similar to the best fit rate of convergence for the first five values of $\gamma$ of $\gamma^{-.71}$.

\begin{figure}[thb]
  \centering
  \includegraphics[width=1\textwidth]{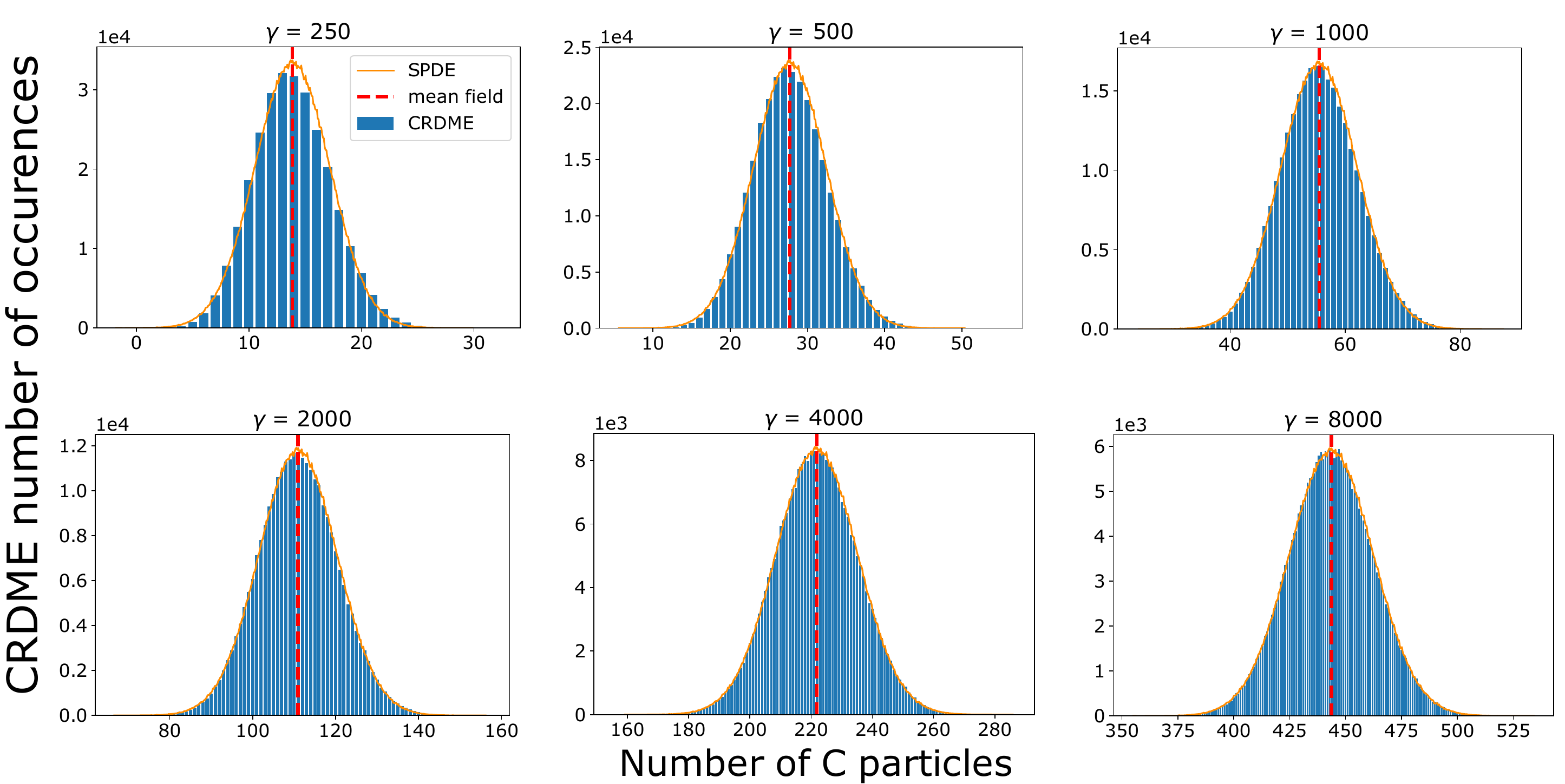}
\caption{CRDME bar plot of particle counts and fluctuation process particle count density for $\gamma=250,500,1000,2000,4000,8000$ at $t=1$ (plotted left to right top to bottom). The vertical dashed line shows the average number of $C$ particles predicted by the mean-field limit (i.e. the $\gamma \to \infty$ limit). The raw fluctuation process and mean-field data have been multiplied by $\gamma$ to convert their original units of molar mass to particle counts. Each fluctuation process curve has also been normalized to have the same integral as the corresponding CRDME box plot (the number of CRDME trials multiplied by $\frac{1}{\gamma}$).}
  \label{fig:histogram}
\end{figure}
As a final demonstration of the accuracy of the SPIDE solution, in Figure \ref{fig:histogram} we show the entire distribution of the CRDME $C$ particle counts, i.e., the distribution of
\begin{equation*}
 \gamma  \int_{\Omega}  C^{\gamma}_{CRDME}(x,t) \, dx
\end{equation*}
compared to the distribution of the $\gamma$-scaled fluctuation process molar mass, i.e., the distribution of
\begin{equation*}
  \gamma  \int_{\Omega}  C^{\gamma}(x,t) \, dx.
\end{equation*}
In the figure we have converted the raw molar mass data to have units of particle counts by multiplying by $\gamma$. We generated the SPIDE curve by partitioning the data produced by our simulations into equally sized bins using the \texttt{hist} function from the Python visualization library \texttt{matplotlib}, pairing each bin center with the corresponding bin count to create a series of equispaced data points, and interpolating linearly between consecutive data points. We then normalized the SPIDE curve to have an integral of (number of CRDME trials) / $\gamma$. We see that, as expected, the SPIDE particle counts have an approximately Gaussian distribution. For large $\gamma$, the SPIDE and CRDME distributions have very little visible difference.

In our tests, the SPIDE solver is about 100 times faster than the CRDME simulation for large enough $\gamma$; it took approximately 21 hours on 28 threads for 1,600,000 SPIDE simulations, or 588 compute hours, for an approximate time of 1.3s/simulation. On the other hand, the $\gamma=8000$ CRDME simulations took around 132s/trial. Hence, the SPIDE provides an intermediate fidelity model, more accurate and more computationally expensive than the mean field model, but less accurate and less computationally expensive than the PBSRD model. As the figures show, the fluctuation process can very accurately capture the overall distribution for statistics of interest at a reduced computational cost when $\gamma$ is sufficiently large.


\section{Proof of the fluctuations result, Theorem \ref{T:FluctuationsThm}}\label{S:ProofMainTheorem}
In this section we prove Theorem~\ref{T:FluctuationsThm}. We begin in the next section by formulating an evolution equation for the fluctuation process, $\Xi^{{\vec{\zeta}}, j}_{t}$. In Section~\ref{S:MomentBounds} we prove a lemma giving moment bounds for the pre-limit measures, $\mu_t^{\vec{\zeta},j}$, and the post-limit mean-field measures, $\bar{\mu}_t^{\vec{\zeta},j}$. We then prove tightness for the fluctuation process in Section~\ref{S:Tightness}, identify the limiting equation it satisfies in Section~\ref{S:Identification}, and finally prove uniqueness of the solution to this equation in Section~\ref{S:Uniqueness}.

\subsection{Preliminary calculations for the prelimit fluctuation process}\label{SS:PreliminaryCalculations}

From \cref{Eq:EM_j_formula,Eq:Limit_EM_formula2}, we directly obtain
\begin{align}
\la f,\Xi^{{\vec{\zeta}}, j}_{t}\ra
&
=\la f,\Xi^{{\vec{\zeta}}, j}_{0}\ra + \int_{0}^{t} \la (\mathcal{L}_jf)(x) , \Xi^{\vec{\zeta}, j}_{s-} (dx)\ra ds + \frac{1}{{\sqrt{\gamma}}}\sum_{i\geq 1}\int_{0}^{t}1_{\{i\leq \gamma \la 1, \mu_{s-}^{{\vec{\zeta}}, j}\ra\}}\sqrt{2D_{j}}\frac{\partial f}{\partial Q}(H^i(\gamma\mu_{s-}^{{\vec{\zeta}}, j}))dW^{i, j}_{s}\nonumber\\
&
\quad-\frac{1}{\gamma} \sqrt{\gamma}\sum_{\ell = 1}^{\tilde{L}} \int_{0}^{t}\int_{\mathbb{I}^{(\ell)}} \int_{\mathbb{R}_{+}}\la f,\sum_{r = 1}^{\alpha_{\ell j}} \delta_{H^{i_r^{(j)}}(\gamma\mu^{{\vec{\zeta}},j}_{s-})}  \ra  1_{\{\vec{i} \in \Omega^{(\ell)}(\gamma\mu^{\vec{\zeta}}_{s-})\}} 1_{ \{ \theta \leq K_\ell^{\gamma}\left(\mathcal{P}^{(\ell)}(\gamma\mu_{s-}^{\vec{\zeta}}, \vec{i}) \right) \}}  d\tilde{N}_{\ell}(s,\vec{i}, \theta)\nonumber\\
&
\quad+\sqrt{\gamma} \frac{1}{\gamma} \sum_{\ell = 1 + \tilde{L}}^L \int_{0}^{t}\int_{\mathbb{I}^{(\ell)}} \int_{\mathbb{Y}^{(\ell)}}\int_{\mathbb{R}_{+}^2} \la f, \sum_{r = 1}^{\beta_{\ell j}} \delta_{y_r^{(j)}} -\sum_{r = 1}^{\alpha_{\ell j}} \delta_{H^{i_r^{(j)}}(\gamma\mu^{{\vec{\zeta}},j}_{s-})} \ra 1_{\{\vec{i} \in \Omega^{(\ell)}(\gamma\mu^{\vec{\zeta}}_{s-})\}}   \nonumber\\
&
\qquad\qquad\qquad\qquad
\times 1_{ \{ \theta_1 \leq K_\ell^{\gamma}\left(\mathcal{P}^{(\ell)}(\gamma\mu_{s-}^{\vec{\zeta}}, \vec{i}) \right) \}}  \times 1_{ \{ \theta_2 \leq  m^\eta_\ell\left(\vec{y} \,  | \, \mathcal{P}^{(\ell)}(\gamma\mu_{s-}^{\vec{\zeta}}, \vec{i}) \right) \}}  d\tilde{N}_{\ell}(s,\vec{i}, \vec{y},\theta_1, \theta_2),\nonumber\\
&
\quad-\sum_{\ell = 1}^{\tilde{L}} \int_{0}^{t} \int_{\tilde{\mathbb{X}}^{(\ell)}}     \frac{1}{\vec{\alpha}^{(\ell)}!}   K_\ell\left(\vec{x}\right)  \left(  \sum_{r = 1}^{\alpha_{\ell j}} f(x_r^{(j)}) \right)\, \sqrt{\gamma}\paren{ \lambda^{(\ell)}[\mu_{s}^{\vec{\zeta}}](d\vec{x})- \lambda^{(\ell)}[\bar{\mu}_{s}](d\vec{x}) } \, ds\nonumber\\
&
\quad+\sqrt{\gamma} \sum_{\ell = \tilde{L} + 1}^L \int_{0}^{t} \int_{\tilde{\mathbb{X}}^{(\ell)}}     \frac{1}{\vec{\alpha}^{(\ell)}!}   K_\ell\left(\vec{x}\right) \left(\int_{\mathbb{Y}^{(\ell)}}    \left( \sum_{r = 1}^{\beta_{\ell j}} f(y_r^{(j)}) \right) m^\eta_\ell\left(\vec{y} \, |\, \vec{x} \right)\, d \vec{y} - \sum_{r = 1}^{\alpha_{\ell j}} f(x_r^{(j)}) \right)\,\lambda^{(\ell)}[\mu_{s}^{\vec{\zeta}}](d\vec{x})\, ds \nonumber\\
&
\quad-\sqrt{\gamma} \sum_{\ell = \tilde{L} + 1}^L \int_{0}^{t} \int_{\tilde{\mathbb{X}}^{(\ell)}}     \frac{1}{\vec{\alpha}^{(\ell)}!}   K_\ell\left(\vec{x}\right) \left(\int_{\mathbb{Y}^{(\ell)}}    \left( \sum_{r = 1}^{\beta_{\ell j}} f(y_r^{(j)}) \right) m_\ell\left(\vec{y} \, |\, \vec{x} \right)\, d \vec{y} - \sum_{r = 1}^{\alpha_{\ell j}} f(x_r^{(j)}) \right)\,\lambda^{(\ell)}[\bar{\mu}_{s}](d\vec{x})\, ds \nonumber\\
&
=\la f,\Xi^{{\vec{\zeta}}, j}_{0}\ra + \int_{0}^{t} \la (\mathcal{L}_jf)(x) , \Xi^{\vec{\zeta}, j}_{s-} (dx)\ra ds  + M^{\vec{\zeta}, j}_t(f) \nonumber\\
&
\quad-\sum_{\ell = 1}^{\tilde{L}} \int_{0}^{t} \int_{\tilde{\mathbb{X}}^{(\ell)}}     \frac{1}{\vec{\alpha}^{(\ell)}!}   K_\ell\left(\vec{x}\right)  \left(  \sum_{r = 1}^{\alpha_{\ell j}} f(x_r^{(j)}) \right)\, \sqrt{\gamma}\paren{ \lambda^{(\ell)}[\mu_{s}^{\vec{\zeta}}](d\vec{x})- \lambda^{(\ell)}[\bar{\mu}_{s}](d\vec{x}) } \, ds\nonumber\\
&
\quad+ \sum_{\ell = \tilde{L} + 1}^L \int_{0}^{t} \int_{\tilde{\mathbb{X}}^{(\ell)}}     \frac{1}{\vec{\alpha}^{(\ell)}!}   K_\ell\left(\vec{x}\right) \left(\int_{\mathbb{Y}^{(\ell)}}    \left( \sum_{r = 1}^{\beta_{\ell j}} f(y_r^{(j)}) \right) m_\ell\left(\vec{y} \, |\, \vec{x} \right)\, d \vec{y} - \sum_{r = 1}^{\alpha_{\ell j}} f(x_r^{(j)}) \right)\,\nonumber\\
&
\qquad\times \sqrt{\gamma} \paren{\lambda^{(\ell)}[\mu_{s}^{\vec{\zeta}}](d\vec{x})-\lambda^{(\ell)}[\bar{\mu}_{s}](d\vec{x})}\, ds \nonumber\\
&
\quad+ \sum_{\ell = \tilde{L} + 1}^L \int_{0}^{t} \int_{\tilde{\mathbb{X}}^{(\ell)}}     \frac{1}{\vec{\alpha}^{(\ell)}!}   K_\ell\left(\vec{x}\right) \left(\int_{\mathbb{Y}^{(\ell)}}    \left( \sum_{r = 1}^{\beta_{\ell j}} f(y_r^{(j)}) \right) \sqrt{\gamma}\paren{m^\eta_\ell\left(\vec{y} \, |\, \vec{x} \right)-m_\ell\left(\vec{y} \, |\, \vec{x} \right)}\, d \vec{y}  \right)\, \lambda^{(\ell)}[\mu_{s}^{\vec{\zeta}}](d\vec{x})\, ds. \nonumber\\
\label{eq:fluctuationProcess}
\end{align}
Note that in the last equality of \cref{eq:fluctuationProcess}, we are using the notation
\begin{equation}\label{eq:MartingaleDef}
M^{\vec{\zeta}, j}_t(f) =  \cC^{\vec{\zeta}, j}_t(f)  + \cD^{\vec{\zeta}, j}_t(f) ,
\end{equation}
as the martingale part with quadratic variation
\begin{equation}
\la M^{\vec{\zeta}, j}\ra_t(f) =  \la \cC^{\vec{\zeta}\gamma, j}\ra_t(f)  + \la \cD^{\vec{\zeta}, j}\ra_t(f),
\end{equation}
where
\begin{equation}\label{eq:contMartingaleDef}
\cC^{\vec{\zeta}, j}_t(f) = \frac{1}{{\sqrt{\gamma}}}\sum_{i\geq 1}\int_{0}^{t}1_{\{i\leq \gamma \la 1, \mu_{s-}^{{\vec{\zeta}}, j}\ra\}}\sqrt{2D_{j}}\frac{\partial f}{\partial Q}(H^i(\gamma\mu_{s-}^{{\vec{\zeta}}, j}))dW^{i, j}_{s},
\end{equation}
denotes the continuous part, i.e. Brownian motion, of the martingale and
\begin{equation}\label{eq:disMartingaleDef}
\begin{aligned}
\cD^{\vec{\zeta}, j}_t(f) &=-\frac{1}{\sqrt{\gamma}}\sum_{\ell = 1}^{\tilde{L}} \int_{0}^{t}\int_{\mathbb{I}^{(\ell)}} \int_{\mathbb{R}_{+}} \la f, \sum_{r = 1}^{\alpha_{\ell j}} \delta_{H^{i_r^{(j)}}(\gamma\mu^{{\vec{\zeta}},j}_{s-})}  \ra
 1_{\{\vec{i} \in \Omega^{(\ell)}(\gamma\mu^{\vec{\zeta}}_{s-})\}}   1_{ \{ \theta \leq K_\ell^{\gamma}\left(\mathcal{P}^{(\ell)}(\gamma\mu_{s-}^{\vec{\zeta}}, \vec{i}) \right) \}}  d\tilde{N}_{\ell}(s,\vec{i}, \theta)\\
&\quad+\frac{1}{\sqrt{\gamma}}\sum_{\ell = \tilde{L}+1}^L \int_{0}^{t}\int_{\mathbb{I}^{(\ell)}} \int_{\mathbb{Y}^{(\ell)}}\int_{\mathbb{R}_{+}^2}\la f, \sum_{r = 1}^{\beta_{\ell j}} \delta_{y_r^{(j)}}  -\sum_{r = 1}^{\alpha_{\ell j}} \delta_{H^{i_r^{(j)}}(\gamma\mu^{{\vec{\zeta}},j}_{s-})}\ra
1_{\{\vec{i} \in \Omega^{(\ell)}(\gamma\mu^{\vec{\zeta}}_{s-})\}}   \\
&\qquad\qquad\qquad\qquad\qquad
\times 1_{ \{ \theta_1 \leq K_\ell^{\gamma}\left(\mathcal{P}^{(\ell)}(\gamma\mu_{s-}^{\vec{\zeta}}, \vec{i}) \right) \}}  1_{ \{ \theta_2 \leq  m^\eta_\ell\left(\vec{y} \,  | \, \mathcal{P}^{(\ell)}(\gamma\mu_{s-}^{\vec{\zeta}}, \vec{i}) \right) \}}  d\tilde{N}_{\ell}(s,\vec{i}, \vec{y},\theta_1, \theta_2),
\end{aligned}
\end{equation}
denotes the discrete part, i.e. Poisson process, of the martingale.

The quadratic variation of $\cC^{\vec{\zeta}, j}_t(f)$ is
\begin{equation}\label{eq:contMartingaleQV}
\la \cC^{\vec{\zeta}, j}\ra_t(f) = \int_{0}^{t} \la 2D_{j} \paren{\frac{ \partial f}{\partial Q}(x)}^2, \mu_{s-}^{{\vec{\zeta}}, j}(dx)\ra\,ds,
\end{equation}
and the quadratic variation of $\cD^{\vec{\zeta}, j}_t(f)$ is
\begin{align}\label{eq:disMartingaleQV}
\la \cD^{\vec{\zeta}, j}\ra_t(f) &=\sum_{\ell = 1}^{\tilde{L}} \int_{0}^{t} \int_{\tilde{\mathbb{X}}^{(\ell)}}     \frac{1}{\vec{\alpha}^{(\ell)}!}   K_\ell\left(\vec{x}\right)  \left(  \sum_{r = 1}^{\alpha_{\ell j}} f(x_r^{(j)}) \right)^2\,\lambda^{(\ell)}[\mu_{s-}^{{\vec{\zeta}}}](d\vec{x}) \, ds\nonumber\\
&
\quad+\sum_{\ell = \tilde{L} + 1}^L \int_{0}^{t} \int_{\tilde{\mathbb{X}}^{(\ell)}}     \frac{1}{\vec{\alpha}^{(\ell)}!}   K_\ell\left(\vec{x}\right) \left(\int_{\mathbb{Y}^{(\ell)}}    \left( \sum_{r = 1}^{\beta_{\ell j}} f(y_r^{(j)}) - \sum_{r = 1}^{\alpha_{\ell j}} f(x_r^{(j)}) \right)^2 m^\eta_\ell\left(\vec{y} \, |\, \vec{x} \right)\, d \vec{y} \right)\,\lambda^{(\ell)}[\mu_{s-}^{{\vec{\zeta}}}](d\vec{x}) \, ds.
\end{align}


Notice that in \cref{eq:fluctuationProcess}, $\{ \Xi^{{\vec{\zeta}}, j}_{t}\}_{j = 1}^J$ evolves through $\sqrt{\gamma}\paren{\lambda^{(\ell)}[\mu_{s}^{\vec{\zeta}}](\vx)-\lambda^{(\ell)}[\bar{\mu}_{s}](\vx)}$. In particular, we have the following cases. If the $\ell$-th reaction is a first order reaction, $A_k \to \cdots$ ,
$$\sqrt{\gamma}\paren{\lambda^{(\ell)}[\mu_{s}^{\vec{\zeta}}](x)-\lambda^{(\ell)}[\bar{\mu}_{s}](x)} =  \Xi^{{\vec{\zeta}},k}_{s}.$$
If the $\ell$-th reaction is a second order reaction, $A_k + A_r \to \cdots$ ,
$$\sqrt{\gamma}\paren{\lambda^{(\ell)}[\mu_{s}^{\vec{\zeta}}](x,y)-\lambda^{(\ell)}[\bar{\mu}_{s}](x,y)}
=  \sqrt{\gamma} \paren{\mu_{s}^{\vec{\zeta}, k}(x) \mu_{s}^{\vec{\zeta}, r}(y) - \bar{\mu}^k_{s}(x)\bar{\mu}^r_{s}(y)}
= \Xi^{{\vec{\zeta}},k}_{s}(x)\mu_{s}^{\vec{\zeta}, r}(y) +  \bar{\mu}^k_{s}(x)\Xi^{{\vec{\zeta}},r}_{s}(y).$$
When $\vec{\zeta}\to 0$, we have that $\mu_{s}^{\vec{\zeta}, j}\to \bar{\mu}^j_{s}$ for all $j = 1, \cdots, J$ in probability. These considerations led to the notation used in Definition~\ref{def:Delta}.

For the remainder of the proof of Theorem~\ref{T:FluctuationsThm}, and without loss of generality, we assume that $\tilde{L} = 0$. The case when $\tilde{L} > 0$ follows by similar arguments.

\subsection{Preliminary moment bounds}\label{S:MomentBounds}
The purpose of this section is to obtain moment bounds on $\mu_{t}^{{\vec{\zeta}}, j}$ that are uniform with respect to $\vec{\zeta}$, and on the limit measure $\bar{\mu}_{t}^{j}$, for all $t \in [0,T]$ for any $T<\infty$.

We are interested in obtaining a uniform bound with respect to $\vec{\zeta}\in(0,1)^{2}$ of the quantity
\[
\theta^{\vec{\zeta},j,4D}_{t}=\mathbb{E}\sup_{s\in[0,t]}\la |\cdot|^{4D},\mu^{\vec{\zeta},j}_{s}\ra.
\]
For this purpose, we have the following lemma.

\begin{lemma}\label{L:MomentBound}
Let us assume $\rho\in L^{1}(\mathbb{R}^{d})$ and that Assumptions \ref{Assume:kernalBdd}, \ref{A:AssumptionRho}, \ref{A:AssumptionK}, \ref{A:AssumptionInitialCondition} hold.
Then, we have that there exists a finite constant $C(C_{\circ},K,\rho,T,D)<\infty$ that depends on the upper bounds of the quantities that appear in those assumptions, such that the following moment bound holds
\begin{align*}
\sum_{j=1}^{J} \theta^{\vec{\zeta},j,4D}_{T}&=\sum_{j=1}^{J}\mathbb{E}\sup_{t\in[0,T]}\la |\cdot|^{4D},\mu^{\vec{\zeta},j}_{t}\ra \leq C(C_{\circ},K,\rho,T,D)<\infty,
\end{align*}
uniformly with respect to $\vec{\zeta}$. In addition, we also have that
\begin{align*}
\sup_{t\in[0,T]}\sum_{j=1}^{J}\la |\cdot|^{8D},\bar{\mu}^{j}_{t}\ra \leq C(C_{\circ},K,\rho,T,D)<\infty.
\end{align*}
\end{lemma}

\begin{proof}[Proof of Lemma \ref{L:MomentBound}]
We will only prove the statement for the moments associated with the prelimit measure $\mu^{\vec{\zeta},j}_{\cdot}$. The statement for the moments of the limit measure $\bar{\mu}^{j}_{\cdot}$ follows by a similar argument.

Without loss of generality we assume $\tilde{L} = 0$. We start by defining
\[
h^{\vec{\zeta},j,8D}_{t}=\mathbb{E}\la |\cdot|^{8D},\mu^{\vec{\zeta},j}_{t}\ra.
\]
Note that by Assumption \ref{A:AssumptionInitialCondition} we have that $h^{\vec{\zeta},j,8D}_{0}<\infty$. We will first prove that $\sup_{t\in[0,T]}\sum_{j=1}^{J} h^{\vec{\zeta},j,8D}_{t}<C$ for some finite constant $C<\infty$ that is uniform with respect to $\zeta\in(0,1)^{2}$. Note that in Assumption \ref{A:AssumptionInitialCondition} we merely assume that such spatial moments are finite (not necessarily uniformly bounded in $\zeta\in(0,1)^{2}$). Now, we shall prove they in fact are uniformly bounded in $\zeta\in(0,1)^{2}$.

Beginning with the equivalent Fock space forward equation representation of~\eqref{Eq:EM_j_formula}, i.e. the forward Kolmogorov equation for the stochastic processes for the number of particles of each type and their positions, we show in Appendix~\ref{S:FockSpaceExpect} that for $f(x) = \abs{x}^{8D}$ we have that $\mathbb{E}\la f,\mu^{\vec{\zeta},j}_{t}\ra$ satisfies the equation
\begin{multline} \label{Eq:Expectation_Mean_Formula}
\mathbb{E}\la f,\mu^{\vec{\zeta},j}_{t}\ra
=\mathbb{E}\la f,\mu^{\vec{\zeta},j}_{0}\ra + \int_{0}^{t} \mathbb{E}\la (\mathcal{L}_j f)(x), \mu^{\vec{\zeta},j}_{s}(dx) \ra ds\\
+\sum_{\ell =  1}^L \mathbb{E}\int_{0}^{t} \int_{\tilde{\mathbb{X}}^{(\ell)}}     \frac{1}{\vec{\alpha}^{(\ell)}!}   K_\ell\left(\vec{x}\right) \left(\int_{\mathbb{Y}^{(\ell)}}    \left( \sum_{r = 1}^{\beta_{\ell j}} f(y_r^{(j)}) \right) m^{\eta}_\ell\left(\vec{y} \, |\, \vec{x} \right)\, d \vec{y} - \sum_{r = 1}^{\alpha_{\ell j}} f(x_r^{(j)}) \right)\,\lambda^{(\ell)}[\mu^{\vec{\zeta}}_{s}](d\vec{x}) \, ds.
\end{multline}
Note that implicit in deriving this equation is the assumed existence of the spatial moment via Assumption~\ref{A:AssumptionInitialCondition}.

Now we derive an a-priori bound that is uniform with respect to $\zeta\in(0,1)^{2}$ and $t\in[0,T]$. We use (\ref{Eq:Expectation_Mean_Formula}) with  $f(x)=|x|^{8D}$ and we note that by Assumption \ref{Assume:kernalBdd} we have that $\la 1, \mu_t^{{\vec{\zeta}}, j} \ra \leq C_{\circ}$ and by Assumption \ref{A:AssumptionInitialCondition} that $h^{\vec{\zeta},j,8D}_{0}\leq C<\infty$. We then obtain
\begin{align}
h^{\vec{\zeta},j,8D}_{t}
&
=h^{\vec{\zeta},j,8D}_{0} + 8D(8D+d-2)D_{j}\int_{0}^{t} h^{\vec{\zeta},j,8D-2}_{s} ds\notag\\
&
\quad+\sum_{\ell =  1}^L \mathbb{E}\int_{0}^{t} \int_{\tilde{\mathbb{X}}^{(\ell)}}     \frac{1}{\vec{\alpha}^{(\ell)}!}   K_\ell\left(\vec{x}\right) \left(\int_{\mathbb{Y}^{(\ell)}}    \left( \sum_{r = 1}^{\beta_{\ell j}} |y_r^{(j)}|^{8D} \right) m_\ell\left(\vec{y} \, |\, \vec{x} \right)\, d \vec{y} - \sum_{r = 1}^{\alpha_{\ell j}} |x_r^{(j)}|^{8D} \right)\,\lambda^{(\ell)}[\mu^{\vec{\zeta}}_{s}](d\vec{x}) \, ds \notag\\
&\quad+\sum_{\ell =  1}^L \mathbb{E}\int_{0}^{t} \int_{\tilde{\mathbb{X}}^{(\ell)}} \frac{1}{\vec{\alpha}^{(\ell)}!}   K_\ell\left(\vec{x}\right) \left(\int_{\mathbb{Y}^{(\ell)}} \left( \sum_{r = 1}^{\beta_{\ell j}} |y_r^{(j)}|^{8D} \right) \left(m^{\eta}_\ell\left(\vec{y} \, |\, \vec{x} \right)-m_\ell\left(\vec{y} \, |\, \vec{x} \right)\right)\, d \vec{y} \right)\,\lambda^{(\ell)}[\mu^{\vec{\zeta}}_{s}](d\vec{x}) \, ds.\notag \\
&\leq h^{\vec{\zeta},j,8D}_{0} + 8D(8D+d-2)D_{j}\int_{0}^{t} h^{\vec{\zeta},j,8D-2}_{s} ds \label{Eq:Expectation_Mean_Formula2}\\
&\quad+\sum_{\ell =  1}^L \mathbb{E}\int_{0}^{t} \int_{\tilde{\mathbb{X}}^{(\ell)}} \frac{1}{\vec{\alpha}^{(\ell)}!} K_\ell\left(\vec{x}\right) \left(\int_{\mathbb{Y}^{(\ell)}}    \left( \sum_{r = 1}^{\beta_{\ell j}} |y_r^{(j)}|^{8D} \right) m_\ell\left(\vec{y} \, |\, \vec{x} \right)\, d \vec{y} \right)\,\lambda^{(\ell)}[\mu^{\vec{\zeta}}_{s}](d\vec{x}) \, ds \notag \\
&\quad+\sum_{\ell =  1}^L \mathbb{E}\int_{0}^{t} \int_{\tilde{\mathbb{X}}^{(\ell)}}     \frac{1}{\vec{\alpha}^{(\ell)}!}   K_\ell\left(\vec{x}\right) \left(\int_{\mathbb{Y}^{(\ell)}}    \left( \sum_{r = 1}^{\beta_{\ell j}} |y_r^{(j)}|^{8D} \right) \left(m^{\eta}_\ell\left(\vec{y} \, |\, \vec{x} \right)-m_\ell\left(\vec{y} \, |\, \vec{x} \right)\right)\, d \vec{y} \right)\,\lambda^{(\ell)}[\mu^{\vec{\zeta}}_{s}](d\vec{x}) \, ds.\notag
\end{align}

We first estimate
\begin{equation*}
  I := \E \int_{\tilde{\mathbb{X}}^{(\ell)}}     \frac{1}{\vec{\alpha}^{(\ell)}!}   K_\ell\left(\vec{x}\right) \left(\int_{\mathbb{Y}^{(\ell)}}    \left( \sum_{r = 1}^{\beta_{\ell j}} |y_r^{(j)}|^{8D} \right) m_\ell\left(\vec{y} \, |\, \vec{x} \right)\, d \vec{y}  \right) \lambda^{(\ell)}[\mu_{s}^{\vec{\zeta}}](d\vec{x})
\end{equation*}
by analyzing all possible cases.
\begin{enumerate}
\item If the $\ell$-th reaction is of the form $S_i \to S_j$, so that $\alpha_{\ell,i} = \beta_{\ell,j} = 1$, then
\begin{equation*}
 I = \E\int_{\R^d}  K_\ell\left(x\right) |x|^{8D} \mu^{\vec{\zeta},i}_{s} (dx) \leq C(K) h^{\vec{\zeta},i,8D}_{s}.
\end{equation*}

\item If the $\ell$-th reaction is of the form $S_i \to S_j + S_n$,
\begin{align*}
I &=
\begin{cases}
\E\int_{\R^d} K_\ell(x) \left(\int_{\R^{2d}}  |y|^{8D}  m_\ell\left(y, z \, |\, x \right)\,  dy \, dz \right)  \mu_{s}^{\vec{\zeta},i} (dx), & (\beta_{\ell j} = 1)\\
\E\int_{\R^d} K_\ell\left(x\right) \left(\int_{\R^{2d}}    \left( |y|^{8D} + |z|^{8D}\right) m_\ell\left(y, z \, |\, x \right)\,  dy\, dz   \right)  \mu_{s}^{\vec{\zeta},i} (dx), &(\beta_{\ell j} = 2)
\end{cases}\\
&\leq \E\int_{\R^d} K_\ell(x) \left(\int_{\R^{2d}} \left( |y|^{8D} + |z|^{8D}\right) \rho(y-z) \sum_{i = 1}^{I} p_i \delta\paren{x - \paren{\alpha_i y + (1-\alpha_i) z}} \, dy\, dz  \right)  \mu_{s}^{\vec{\zeta},i} (dx)\\
&= \sum_{i = 1}^{I} p_i  \E\int_{\R^d}  K_\ell(x) \left(\int_{\R^{2d}}  \paren{|w+z|^{8D} + |z|^{8D}}  \rho(w) \delta\paren{x - \paren{\alpha_i w + z}} \, dz\, dw  \right)  \mu_{s}^{\vec{\zeta},i} (dx) \\
&= \sum_{i = 1}^{I} p_i \E \int_{\R^d} K_\ell(x) \left(\int_{\R^{d}} \paren{|x + \paren{1-\alpha_i}w|^{8D} + \abs{x - \alpha_i w}^{8D}}  \rho(w) \, dw  \right)  \mu_{s}^{\vec{\zeta},i} (dx) \\
&\leq C \sum_{i = 1}^{I} p_i \E \int_{\R^d} K_\ell(x) \left(\int_{\R^{d}} \paren{|x|^{8D} + \abs{w}^{8D}}  \rho(w) \, dw  \right)  \mu_{s}^{\vec{\zeta},i} (dx) \\
&\leq C \|K_\ell\|_{\infty}\sum_{i = 1}^{I} p_i \E \left[\la |x|^{8D}, \mu_{s}^{\vec{\zeta},i}(dx)\ra+ \int_{\R^{d}} \abs{w}^{8D} \rho(w) \, dw \la 1, \mu_{s}^{\vec{\zeta},i}\ra\right]
 \\
&\leq C(C_{\circ},K)(1+h^{\vec{\zeta},i,8D}_{s}).
\end{align*}
Here in obtaining the second to last inequality we made use of Assumption~\ref{A:AssumptionRho}.
\item If the $\ell$-th reaction is of the form $S_i + S_n \to S_j$,
\begin{align*}
  I &= \E\int_{\R^{2d}} K_\ell\left(x, y\right) \left(\int_{\R^d}    |z|^{8D} m_\ell\left(z |\, x, y\right)\, d z  \right)   \mu_{s}^{\vec{\zeta}, n}(dx)\mu_{s}^{\vec{\zeta}, i}(dy) \\
  &= \sum_{i = 1}^{I} p_i \E\int_{\R^{2d}}     K_\ell\left(x, y\right)   |\alpha_i x + (1-\alpha_i) y|^{8D}     \mu_{s}^{\vec{\zeta}, n}(dx)\mu_{s}^{\vec{\zeta}, i}(dy) \\
  & \leq C(C_{\circ},K)\sum_{j=1}^{J}h^{\vec{\zeta},j,8D}_{s}.
\end{align*}
\item If the $\ell$-th reaction is of the form $S_i + S_n \to S_j + S_r$,
\begin{align*}
I &= \begin{cases}
  \frac{1}{2} \E\int_{\R^{2d}} K_\ell\left(x,y\right) \left(\int_{\R^{2d}} |z|^{8D} \, m_\ell\left(z, w \, |\,x, y \right)\, d z\, dw  \right) \mu_{s}^{\vec{\zeta}, n}(dx)\mu_{s}^{\vec{\zeta}, i}(dy), & (\beta_{\ell j} = 1)\\
  \frac{1}{2} \E\int_{\R^{2d}} K_\ell\left(x,y\right) \left(\int_{\R^{2d}}    \left( |z|^{8D} + |w|^{8D} \right) m_\ell\left(z, w \, |\,x, y \right)\, d z\, dw  \right) \mu_{s}^{\vec{\zeta}, n}(dx)\mu_{s}^{\vec{\zeta}, i}(dy), & (\beta_{\ell j} = 2)
\end{cases} \\
&\leq \tfrac{1}{2}\E\int_{\R^{2d}} K_\ell(x, y)    \left( |x|^{8D} + |y|^{8D} \right)  \mu_{s}^{\vec{\zeta}, n}(dx)\mu_{s}^{\vec{\zeta}, i}(dy) \\
&\leq  C(C_{\circ},K)\sum_{j=1}^{J}h^{\vec{\zeta},j,8D}_{s}.
\end{align*}
\end{enumerate}
Thus, we obtain that
\begin{align}\label{eq:DiffMeasPart2}
  I \leq C(C_{\circ},K) \paren{1 +\sum_{j=1}^{J}h^{\vec{\zeta},j,8D}_{s}}.
\end{align}

It remains to bound the second term in (\ref{Eq:Expectation_Mean_Formula2}),
\begin{equation*}
  \mathbb{E}\int_{0}^{t} \int_{\tilde{\mathbb{X}}^{(\ell)}}     \frac{1}{\vec{\alpha}^{(\ell)}!}   K_\ell\left(\vec{x}\right) \left(\int_{\mathbb{Y}^{(\ell)}}    \left( \sum_{r = 1}^{\beta_{\ell j}} |y_r^{(j)}|^{8D} \right) \left(m^{\eta}_\ell\left(\vec{y} \, |\, \vec{x} \right)-m_\ell\left(\vec{y} \, |\, \vec{x} \right)\right)\, d \vec{y} \right)\,\lambda^{(\ell)}[\mu^{\vec{\zeta}}_{s}](d\vec{x}) \, ds.
\end{equation*}
The analysis of the innermost integral is done as in the proof of Lemma B.2 of \cite{IMS:2022} using the decomposition in terms of the different types of reactions (similar to the approach above). The role of the test function $f(y)$ in that lemma is played here by  $f(y)= \sum_{r = 1}^{\beta_{\ell j}}|y_r^{(j)}|^{8D}$. The primary change in the estimates is that we can not use the uniform $C_{b}^{1}(\R^{d})$ norm when bounding $\abs{f(y) - f(x)}$ over balls about $x$, since the moments are unbounded over free-space. We instead use $x$-dependent estimates over the balls combined with Assumption~\ref{A:AssumptionRho} to  obtain bounds of the form
\begin{equation*}
\abs{\int_{\mathbb{Y}^{(\ell)}}   |y_r^{(j)}|^{8D}  \left(m^{\eta}_\ell\left(\vec{y} \, |\, \vec{x} \right)-m_\ell\left(\vec{y} \, |\, \vec{x} \right)\right)\, d \vec{y}}  \leq C(D) \eta \paren{1 + \abs{x_{r}^{(j)}}^{8D-1}}.
\end{equation*}
Using such modifications, and the assumption that $\sum_{t\in[0,T]}\sum_{j=1}^{J}\la 1, \mu^{\vec{\zeta},j}_{t} \ra$ is bounded, one gets that there exists a constant $C=C(C_{\circ},K,D)$ which  depends on $\sum_{t\in[0,T]}\sum_{j=1}^{J}\la 1, \mu^{\vec{\zeta},j}_{t} \ra$ and on the upper uniform bound for the reaction kernel $K_\ell\left(\vec{x}\right)$ such that
\begin{align*}
&\mathbb{E}\int_{0}^{t} \int_{\tilde{\mathbb{X}}^{(\ell)}}     \frac{1}{\vec{\alpha}^{(\ell)}!}   K_\ell\left(\vec{x}\right) \left(\int_{\mathbb{Y}^{(\ell)}}    \left( \sum_{r = 1}^{\beta_{\ell j}} |y_r^{(j)}|^{8D} \right) \left(m^{\eta}_\ell\left(\vec{y} \, |\, \vec{x} \right)-m_\ell\left(\vec{y} \, |\, \vec{x} \right)\right)\, d \vec{y} \right)\,\lambda^{(\ell)}[\mu^{\vec{\zeta}}_{s}](d\vec{x}) \, ds \\
&\qquad \leq \eta C (C_{\circ},K,D)\int_{0}^{t} \paren{1 + \sum_{j=1}^{J} h^{\vec{\zeta},j,8D-1}_{s}} ds.
\end{align*}

Combining now the last bound with~\eqref{eq:DiffMeasPart2} implies the following bound for~\eqref{Eq:Expectation_Mean_Formula2}
\begin{align*}
h^{\vec{\zeta},j,8D}_{s} &\leq h^{\vec{\zeta},j,8D}_{0}+ 8D(8D+d-2)D_{j}\int_{0}^{t} h^{\vec{\zeta},j,8D-2}_{s} ds+ C(C_{\circ},K,D)\left[t + \sum_{j=1}^{J} \int_{0}^{t} \paren{h^{\vec{\zeta},j,8D}_{s} + \eta h^{\vec{\zeta},j,8D-1}_{s}} ds\right] \\
&\leq C(C_{\circ},K,D) \paren{h^{\vec{\zeta},j,8D}_{0} + t + \sum_{j=1}^J \int_{0}^{t} h^{\vec{\zeta},j,8D}_{s} ds}
\end{align*}
by bounding the lower moments by $C(1 + h^{\vec{\zeta},j,8D}_{s})$.
Summing over all $j\in\{1,\cdots, J\}$ and using Grownwall's inequality, we find that there exists some constant $C(C_{\circ},K,\rho,T,D)<\infty$ that depends on $\sup_{t\in[0,T]}\sum_{j=1}^{J}\la 1, \mu^{\vec{\zeta},j}_{t}\ra$, on the upper uniform bound for the reaction kernel $K_\ell\left(\vec{x}\right)$, on $\int_{\mathbb{R}^{d}}|w|^{8D}\rho(w)dw$, on $\norm{\rho}_{L^{1}(\mathbb{R}^{d})}$, and on $T<\infty$, such that
\begin{align}
\sup_{t\in[0,T]}\sum_{j=1}^{J} h^{\vec{\zeta},j,8D}_{t}\leq C(C_{\circ},K,\rho,T,D)<\infty.\label{Eq:MomentBoundExpOutside}
\end{align}

So, indeed starting with equation (\ref{Eq:Expectation_Mean_Formula2}) and recalling the uniform bound at time $t=0$ by Assumption \ref{A:AssumptionInitialCondition}, we obtain in (\ref{Eq:MomentBoundExpOutside}) the a-priori uniform bound in $\zeta\in(0,1)^{2}$ and $t\in[0,T]$ of $h^{\vec{\zeta},j,8D}_{t}$.

We now prove the uniform bound for $\theta^{\vec{\zeta},j,4D}_{t}$. Recall that we can assume, without loss of generality, that $\tilde{L}=0$. Due to the preceding bound, we can now write (\ref{Eq:EM_j_formula}) for $f(x)=|x|^{4D}$
\begin{equation}  \label{eq:MuProcess}
\begin{aligned}
\la f,\mu^{{\vec{\zeta}}, j}_{t}\ra&
=\la f,\mu^{{\vec{\zeta}}, j}_{0}\ra + \int_{0}^{t} \la (\mathcal{L}_jf)(x) , \mu^{\vec{\zeta}, j}_{s-} (dx)\ra ds  + \frac{1}{\sqrt{\gamma}}M^{\vec{\zeta}, j}_t(f) \\
&\phantom{=}+ \sum_{\ell = 1}^L \int_{0}^{t} \int_{\tilde{\mathbb{X}}^{(\ell)}}     \frac{1}{\vec{\alpha}^{(\ell)}!}   K_\ell\left(\vec{x}\right) \left(\int_{\mathbb{Y}^{(\ell)}}  \left( \sum_{r = 1}^{\beta_{\ell j}} f(y_r^{(j)}) \right) \paren{m^\eta_\ell\left(\vec{y} \, |\, \vec{x} \right)-m_\ell\left(\vec{y} \, |\, \vec{x} \right)}\, d \vec{y}  \right)\, \lambda^{(\ell)}[\mu_{s}^{\vec{\zeta}}](d\vec{x})\, ds,
\end{aligned}
\end{equation}
where $M^{\vec{\zeta}, j}_t(f)= \cC^{\vec{\zeta}, j}_t(f)  + \cD^{\vec{\zeta},
j}_t(f)$ is the martingale given by (\ref{eq:MartingaleDef}). By the
Burkoholder-Davis-Gundy inequality with $f(x)=|x|^{4D}$, we have that there
exists a finite constant $C$ that depends on $D_{j}$ and $D$, such that
\begin{align}
\E\sup_{t\in[0,T]}|\cC^{\vec{\zeta}, j}_t(|\cdot|^{4D})|^{2}
&\leq C \int_{0}^{T}  h_{s-}^{{\vec{\zeta}}, j,8D-2} ds.
\end{align}

Similarly, using Doob's maximal inequality and calculations for the different reaction possibilities as before, we obtain that there exists a constant $C<\infty$ that may depend on $\sum_{t\in[0,T]}\sum_{j=1}^{J}\la1, \mu^{\vec{\zeta},j}_{t}\ra$, on the upper uniform bound for the reaction kernel $K_\ell\left(\vec{x}\right)$ , on $\int_{\mathbb{R}^{d}}|w|^{8D}\rho(w)dw$, on $\norm{\rho}_{L^{1}(\mathbb{R}^{d})}$, and on $T<\infty$ such that
\begin{align*}
&\E\sup_{t\in[0,T]}|\cD^{\vec{\zeta}, j}_t(|\cdot|^{4D})|^{2} \\
&\leq C \E\sum_{\ell = 1}^L \int_{0}^{T} \int_{\tilde{\mathbb{X}}^{(\ell)}}     \frac{1}{\vec{\alpha}^{(\ell)}!}   K_\ell\left(\vec{x}\right) \left(\int_{\mathbb{Y}^{(\ell)}}    \left( \sum_{r = 1}^{\beta_{\ell j}} |y_r^{(j)}|^{4D} - \sum_{r = 1}^{\alpha_{\ell j}} |x_r^{(j)}|^{4D} \right)^2 m^\eta_\ell\left(\vec{y} \, |\, \vec{x} \right)\, d \vec{y} \right)\,\lambda^{(\ell)}[\mu_{s-}^{{\vec{\zeta}}}](d\vec{x}) \, ds.\\
&\leq C \sum_{j=1}^{J}\int_{0}^{T}  h_{s}^{{\vec{\zeta}}, j,8D} ds.
\end{align*}
Using then the bound (\ref{Eq:MomentBoundExpOutside}), we obtain that
\begin{align}
\E\sup_{t\in[0,T]}|M^{\vec{\zeta}, j}_t(|\cdot|^{4D})|^{2}&\leq C,\label{Eq:MartBound}
\end{align}
for some constant $C<\infty$ that may depend on $\sum_{t\in[0,T]}\sum_{j=1}^{J}\la 1, \mu^{\vec{\zeta},j}_{t}\ra$, on the upper uniform bound for the reaction kernel $K_\ell\left(\vec{x}\right)$ , on $\int_{\mathbb{R}^{d}}|w|^{8D}\rho(w)dw$, on $\norm{\rho}_{L^{1}(\mathbb{R}^{d})}$, and on $T<\infty$, but not on $\vec{\zeta}$.

Taking now supremum over $t\in[0,T]$ and then the expectation operator in (\ref{eq:MuProcess}), using (\ref{Eq:MartBound}) for the martingale terms and (\ref{Eq:MomentBoundExpOutside}) to bound the supremum of the Riemann integral terms, we obtain the desired bound
\begin{align*}
\sum_{j=1}^{J} \theta^{\vec{\zeta},j,4D}_{T}&=\sum_{j=1}^{J}\mathbb{E}\sup_{t\in[0,T]}\la |\cdot|^{4D},\mu^{\vec{\zeta},j}_{t}\ra \leq C(C_{\circ},K,\rho,T,D)<\infty,
\end{align*}
concluding the proof of the lemma.
\end{proof}

\subsection{Tightness} \label{S:Tightness}
In this section, we aim to rigorously show the relative compactness of the fluctuation process $\left\{\paren{\Xi^{\vec{\zeta}, 1}_t, \cdots, \Xi^{\vec{\zeta}, J}_t},t\in[0, T]\right\}_{\vec{\zeta}\in(0,1)^{2}}$ and  of the martingale vector $\left\{\paren{M^{\vec{\zeta},1}_t, \cdots, M^{\vec{\zeta}, J}_t},t\in[0, T]\right\}_{\vec{\zeta}\in(0,1)^{2}}$ in $D_{(W^{-\Gamma,a}(\R^d))^{\otimes J}}([0, T])$.

Let $\kappa>0$, $q\geq 0$ and define the stopping time
\begin{align}
\theta^{q}_{\vec{\zeta},\kappa}&=\inf\left\{t\geq 0: \sum_{j=1}^{J}\la |\cdot|^{q},\mu^{\vec{\zeta},j}_{t}\ra\geq \kappa\right\}.
\end{align}
The estimate in Lemma \ref{L:MomentBound} implies that for $T<\infty$ and $q=4D\geq 0$, we have that
\begin{align}
\lim_{\kappa\rightarrow\infty}\sup_{\vec{\zeta}\in(0,1)^{2}}\mathbb{P}(\theta^{q}_{\vec{\zeta},\kappa}\leq T)=0.
\end{align}
Thus, for an appropriately chosen $q\geq 0$, it is enough to prove tightness for $\left\{\paren{\Xi^{\vec{\zeta}, 1}_{t\wedge\theta^{q}_{\vec{\zeta},\kappa}}, \cdots, \Xi^{\vec{\zeta}, J}_{t\wedge\theta^{q}_{\vec{\zeta},\kappa}}}\right\}_{t\in[0, T]} $ and  of the martingales $\left\{\paren{M^{\vec{\zeta},1}_{t\wedge\theta^{q}_{\vec{\zeta},\kappa}}, \cdots, M^{\vec{\zeta}, J}_{t\wedge\theta^{q}_{\vec{\zeta},\kappa}}}\right\}_{t\in[0, T]} $ in $D_{(W^{-\Gamma,a}(\R^d))^{\otimes J}}([0, T])$.


\subsubsection{Uniform Bound on the Fluctuation Process $\left\{\paren{\Xi^{\vec{\zeta}, 1}_{t\wedge\theta^{2b}_{\vec{\zeta},\kappa}}, \cdots,  \Xi^{\vec{\zeta}, J}_{t\wedge\theta^{2b}_{\vec{\zeta},\kappa}}}\right\}_{t\in[0, T]}$}

Before we prove our main Lemma \ref{lem:MartingaleNegBdd}, let us first present the technical \cref{lem:contMartingale,lem:poissonBd1,lem:MartingalePoisson,lem:densityIneq}.
\begin{lemma}\label{lem:contMartingale}
Let $\{ f_p \}_{p\geq 1}$ be a complete orthonormal system in $W_0^{\Gamma_{1},a}(\R^d)$ of class $C^\infty_0(\R^d)$. For a fixed $\gamma$, and assuming from Assumption~\ref{A:SobolevSpaceParameters} that $b>a+d/2$, $\Gamma_1\geq 2D+1$, and $a\geq D$, we have that
\begin{equation*}
  \sum_{p\geq 1}\int_0^{t\wedge\theta^{2b}_{\vec{\zeta},\kappa}}  \la f_p, \Xi^{{\vec{\zeta}}, j}_{s-}\ra d\cC^{\vec{\zeta}, j}_s (f_p)
\end{equation*}
is a square integrable martingale with the expected value of its quadratic variation bounded by
\begin{equation*}
  C(D_j,\kappa) \, \mathbb{E}\int_0^{t\wedge\theta^{2b}_{\vec{\zeta},\kappa}}   || \Xi^{{\vec{\zeta}}, j}_{s-}||_{-\Gamma_{1},a}^2  \, ds
\end{equation*}
for a finite constant $C<\infty$.
\end{lemma}
\begin{proof}
For a fixed $\gamma$,
\begin{align}\label{eq:fixGammaBd}
 ||\Xi^{{\vec{\zeta}}, j}_{s}||_{-\Gamma_{1},a} &= \sup_{||f||_{\Gamma_{1},a} = 1}\la f, \Xi^{{\vec{\zeta}}, j}_{s}\ra \nonumber\\
 & = \sqrt{\gamma}\sup_{||f||_{\Gamma_{1},a} = 1}\la f, \mu^{{\vec{\zeta}}, j}_{s} - \bar{\mu}^j_{s}\ra\nonumber\\
 & \leq C \sqrt{\gamma}\sup_{||f||_{\Gamma_{1},a} = 1} ||f||_{C^{0,a}} \leq C \sqrt{\gamma}  < \infty,
\end{align}
where we used the moment bound from Lemma \ref{L:MomentBound} and Sobolev embedding since $\Gamma_{1}>d/2$. For a fixed $\gamma$, $b>a+d/2$, by \cref{eq:contMartingaleQV}, we have that $\sum_{p\geq 1}\int_0^{t\wedge\theta^{2b}_{\vec{\zeta},\kappa}}  \la f_p, \Xi^{{\vec{\zeta}}, j}_{s-}\ra d\cC^{\vec{\zeta}, j}_s (f_p)$ is a square integrable martingale with quadratic variation having expected value
\begin{align*}
\mathbb{E}\sum_{p\geq 1}\int_0^{t\wedge\theta^{2b}_{\vec{\zeta},\kappa}}  \la f_p, \Xi^{{\vec{\zeta}}, j}_{s-}\ra^2 d\la\cC^{\vec{\zeta}, j}\ra_s (f_p)
&= 2D_{j}\mathbb{E}\sum_{p\geq 1}\int_0^{t\wedge\theta^{2b}_{\vec{\zeta},\kappa}}   \la f_p, \Xi^{{\vec{\zeta}}, j}_{s-}\ra^2 \la \paren{\frac{ \partial f_p}{\partial Q}(x)}^2, \mu_{s-}^{{\vec{\zeta}}, j}(dx)\ra \, ds \\
&\leq 2D_{j} \paren{\sum_{p\geq 1}||f_p||^{2}_{C^{1,b}}}\mathbb{E}\int_0^{t\wedge\theta^{2b}_{\vec{\zeta},\kappa}}   \left(|| \Xi^{{\vec{\zeta}}, j}_{s-}||_{-\Gamma_{1},a}^2\right) \kappa  \, ds \\
&\leq C(D_j,\kappa) \paren{\sum_{p\geq 1}||f_p||^{2}_{1+D,b}}\mathbb{E}\int_0^{t\wedge\theta^{2b}_{\vec{\zeta},\kappa}}   \left(|| \Xi^{{\vec{\zeta}}, j}_{s-}||_{-\Gamma_{1},a}^2 \right) \, ds  < \infty.
\end{align*}
The derivation of the last line used the Sobolev embedding theorem, and that since $\Gamma_{1}-(1+D)>d/2$ and $b-a>d/2$ the embedding $W^{\Gamma_{1},a}_{0}\hookrightarrow W^{1+D,b}_{0}$ is of Hilbert-Schmidt type, so $\sum_{p\geq 1}||f_p||^{2}_{1+D,b}<\infty$. 
This concludes the proof of the lemma.
\end{proof}

We now develop several useful bounds involving the function
\begin{multline}\label{eq:jump}
g^{\ell,j, f, \mu^{\vec{\zeta}}}(s,\vec{i}, \vec{y},\theta_1, \theta_2)=\frac{1}{\sqrt{\gamma}}\la f, \sum_{r = 1}^{\beta_{\ell j}} \delta_{y_r^{(j)}}  -\sum_{r = 1}^{\alpha_{\ell j}} \delta_{H^{i_r^{(j)}}(\gamma\mu^{{\vec{\zeta}},j}_{s-})} \ra \\
\times 1_{\{\vec{i} \in \Omega^{(\ell)}(\gamma\mu^{\vec{\zeta}}_{s-})\}}  1_{ \{ \theta_1 \leq K_\ell^{\gamma}\left(\mathcal{P}^{(\ell)}(\gamma\mu_{s-}^{\vec{\zeta}}, \vec{i}) \right) \}}  1_{ \{ \theta_2 \leq  m^\eta_\ell\left(\vec{y} \,  | \, \mathcal{P}^{(\ell)}(\gamma\mu_{s-}^{\vec{\zeta}}, \vec{i}) \right) \}},
\end{multline}
representing the "jump" at time $s$. Notice that it is uniformly bounded and of order $\mathcal{O}(\frac{1}{\sqrt{\gamma}})$ for $f \in C_0^{\infty}(\R^d)$.

\begin{lemma}\label{lem:poissonBd1}
For any $f_p \in W_0^{D,b}(\R^d)\cap C^\infty_{0}(\R^d)$, there is a finite constant $C$ such that
\begin{align}
&\mathbb{E}\sup_{t\in [0, T]}\sum_{\ell = 1}^L \int_{0}^{t\wedge\theta^{2b}_{\vec{\zeta},\kappa}}\int_{\mathbb{I}^{(\ell)}} \int_{\mathbb{Y}^{(\ell)}}\int_{\mathbb{R}_{+}^2}\paren{ \paren{\la f_p,\Xi^{{\vec{\zeta}}, j}_{s-}\ra + g^{\ell, j, f_p, \mu^{\vec{\zeta}}}(s,\vec{i}, \vec{y},\theta_1, \theta_2) }^2 - \la f_p,\Xi^{{\vec{\zeta}}, j}_{s-}\ra^2 \right. \nonumber\\
&
\qquad\qquad\left.  - 2g^{\ell, j, f_p, \mu^{\vec{\zeta}}}(s,\vec{i}, \vec{y},\theta_1, \theta_2) \la f_p,\Xi^{{\vec{\zeta}}, j}_{s-}\ra} \,  d\bar{N}_{\ell}(s,\vec{i}, \vec{y},\theta_1, \theta_2)  \nonumber\\
&\leq     TC  ||f_p||_{D,b}^2\nonumber
\end{align}
is uniformly bounded in $\vec{\zeta}\in(0,1)^{2}$. The constant $C$ depends on $L,C(K),22),\kappa$ and the upper bound from the moment bound of Lemma \ref{L:MomentBound}.
\end{lemma}
\begin{proof}
We write
\begin{equation}\label{eq:poissonBdDerive}
\begin{aligned}
&\mathbb{E}\sup_{t\in [0, T]}\sum_{\ell = 1}^L \int_{0}^{t\wedge\theta^{2b}_{\vec{\zeta},\kappa}}\int_{\mathbb{I}^{(\ell)}} \int_{\mathbb{Y}^{(\ell)}}\int_{\mathbb{R}_{+}^2}\paren{ \paren{\la f_p,\Xi^{{\vec{\zeta}}, j}_{s-}\ra + g^{\ell, j, f_p, \mu^{\vec{\zeta}}}(s,\vec{i}, \vec{y},\theta_1, \theta_2) }^2 - \la f_p,\Xi^{{\vec{\zeta}}, j}_{s-}\ra^2 \right. \\
&\qquad\qquad\qquad\qquad\qquad\qquad\qquad\qquad\qquad\qquad\left.  - 2g^{\ell, j, f_p, \mu^{\vec{\zeta}}}(s,\vec{i}, \vec{y},\theta_1, \theta_2) \la f_p,\Xi^{{\vec{\zeta}}, j}_{s-}\ra} \,  d\bar{N}_{\ell}(s,\vec{i}, \vec{y},\theta_1, \theta_2)  \\
&
=\mathbb{E} \sup_{t\in [0, T]}\sum_{\ell = 1}^L \int_{0}^{t\wedge\theta^{2b}_{\vec{\zeta},\kappa}}\int_{\mathbb{I}^{(\ell)}} \int_{\mathbb{Y}^{(\ell)}}\int_{\mathbb{R}_{+}^2}\paren{ g^{\ell, j, f_p, \mu^{\vec{\zeta}}}(s,\vec{i}, \vec{y},\theta_1, \theta_2) }^2   \,  d\bar{N}_{\ell}(s,\vec{i}, \vec{y},\theta_1, \theta_2) \\
&
=\gamma^{-1} \mathbb{E} \sup_{t\in [0, T]} \sum_{\ell = 1}^L \int_{0}^{t\wedge\theta^{2b}_{\vec{\zeta},\kappa}}\int_{\mathbb{I}^{(\ell)}} \int_{\mathbb{Y}^{(\ell)}}\int_{\mathbb{R}_{+}^2} \left(\sum_{r = 1}^{\beta_{\ell j}}f_p(y_r^{(j)}) - \sum_{r = 1}^{\alpha_{\ell j}} f_p (H^{i_r^{(j)}(\gamma\mu^{{\vec{\zeta}},j}_{s-})}) \right)^2\\
&\qquad\qquad\qquad\qquad\qquad
\times 1_{\{\vec{i} \in \Omega^{(\ell)}(\gamma\mu^{\vec{\zeta}}_{s-})\}} 1_{ \{ \theta_1 \leq K_\ell^{\gamma}\left(\mathcal{P}^{(\ell)}(\gamma\mu_{s-}^{\vec{\zeta}}, \vec{i}) \right) \}}  1_{ \{ \theta_2 \leq  m^\eta_\ell\left(\vec{y} \,  | \, \mathcal{P}^{(\ell)}(\gamma\mu_{s-}^{\vec{\zeta}}, \vec{i}) \right) \}}
  \,  d\bar{N}_{\ell}(s,\vec{i}, \vec{y},\theta_1, \theta_2) \\
&=\mathbb{E}\sup_{t\in [0, T]}\sum_{\ell = 1}^L \int_{0}^{t\wedge\theta^{2b}_{\vec{\zeta},\kappa}}\int_{\tilde{\mathbb{X}}^{(\ell)}}\frac{1}{\vec{\alpha}^{(\ell)}!}\paren{\int_{\mathbb{Y}^{(\ell)}} \left(\sum_{r = 1}^{\beta_{\ell j}}f_p(y_r^{(j)}) - \sum_{r = 1}^{\alpha_{\ell j}} f_p ( x_r^{(j)})  \right)^2  m^\eta_\ell(\vy \,  | \, \vx)  \, \, d\vy}  K_\ell(\vx) \lambda^{(\ell)}[\mu_{s}^{\vec{\zeta}}](d\vec{x})\,  ds \\
& < C T ||f_p||_{C^{0,b}}^2  \\
 &< C T ||f_p||_{D,b}^2  < \infty
\end{aligned}
\end{equation}
which follows using the Sobolev embedding theorem as $D>d/2$, using Assumption~\ref{Assume:rescaling} on the $\gamma$-scaling of $K_\ell^{\gamma}(\vx)$, and using the uniform bound on the moments of $\mu_{s}^{\vec{\zeta}}$ from Lemma~\ref{L:MomentBound}.

Note, in the preceding estimate the final bound has no $\gamma$ dependence.
\end{proof}

\begin{lemma}\label{lem:MartingalePoisson}
For a fixed $\gamma$, and any $f_p \in W_0^{D,b}(\R^d)\cap C^\infty_{0}(\R^d)$,
\begin{align*}
&\sum_{\ell = 1}^L \int_{0}^{t\wedge\theta^{2b}_{\vec{\zeta},\kappa}}\int_{\mathbb{I}^{(\ell)}} \int_{\mathbb{Y}^{(\ell)}}\int_{\mathbb{R}_{+}^2} \paren{ \paren{\la f_p,\Xi^{{\vec{\zeta}}, j}_{s-}\ra +  g^{\ell, j, f_p, \mu^{\vec{\zeta}}}(s,\vec{i}, \vec{y},\theta_1, \theta_2) }^2 - \la f_p,\Xi^{{\vec{\zeta}}, j}_{s-}\ra^2} \,d\tilde{N}_{\ell}(s,\vec{i}, \vec{y},\theta_1, \theta_2)
\end{align*}
is a square integrable martingale with quadratic variation bounded by
$$\frac{1}{\gamma}C t ||f_p||_{D,b}^4 + C  ||f_p||_{D,b}^2 \int_0^{t\wedge\theta^{2b}_{\vec{\zeta},\kappa}} \paren{ \la f_p,\Xi^{{\vec{\zeta}}, j}_{s-}\ra}^2\, ds.$$
The constant $C$ depends on $L, C(K), C_{\circ}, \kappa$, and the upper bound from the moment bound of Lemma~\ref{L:MomentBound}.
\end{lemma}

\begin{proof}
As shown in \cref{eq:fixGammaBd}, for a fixed $\gamma$,  $||\Xi^{\vec{\zeta}, 1}_{s}||_{-\Gamma_{1},a} < \infty$, so that the quadratic variation
\begin{align}\label{eq:QVpoissonBd}
&
\sum_{\ell = 1}^L \int_{0}^{t\wedge\theta^{2b}_{\vec{\zeta},\kappa}}\int_{\mathbb{I}^{(\ell)}} \int_{\mathbb{Y}^{(\ell)}}\int_{\mathbb{R}_{+}^2} \paren{ \paren{\la f_p,\Xi^{{\vec{\zeta}}, j}_{s-}\ra +  g^{\ell, j, f_p, \mu^{\vec{\zeta}}}(s,\vec{i}, \vec{y},\theta_1, \theta_2) }^2 - \la f_p,\Xi^{{\vec{\zeta}}, j}_{s-}\ra^2}^2 \,d\bar{N}_{\ell}(s,\vec{i}, \vec{y},\theta_1, \theta_2)    \nonumber\\
&
 =\sum_{\ell = 1}^L \int_{0}^{t\wedge\theta^{2b}_{\vec{\zeta},\kappa}}\int_{\mathbb{I}^{(\ell)}} \int_{\mathbb{Y}^{(\ell)}}\int_{\mathbb{R}_{+}^2} \paren{ \paren{ g^{\ell, j, f_p, \mu^{\vec{\zeta}}}(s,\vec{i}, \vec{y},\theta_1, \theta_2) }^2 + 2 \la f_p,\Xi^{{\vec{\zeta}}, j}_{s-}\ra g^{\ell, j, f_p, \mu^{\vec{\zeta}}}(s,\vec{i}, \vec{y},\theta_1, \theta_2)}^2 \,d\bar{N}_{\ell}(s,\vec{i}, \vec{y},\theta_1, \theta_2)  \nonumber\\
&
\leq 2\sum_{\ell = 1}^L \int_{0}^{t\wedge\theta^{2b}_{\vec{\zeta},\kappa}}\int_{\mathbb{I}^{(\ell)}} \int_{\mathbb{Y}^{(\ell)}}\int_{\mathbb{R}_{+}^2}  \paren{g^{\ell, j, f_p, \mu^{\vec{\zeta}}}(s,\vec{i}, \vec{y},\theta_1, \theta_2)}^4 \,  d\bar{N}_{\ell}(s,\vec{i}, \vec{y},\theta_1, \theta_2) \nonumber\\
&
\qquad + 8\sum_{\ell = 1}^L \int_{0}^{t\wedge\theta^{2b}_{\vec{\zeta},\kappa}}\int_{\mathbb{I}^{(\ell)}} \int_{\mathbb{Y}^{(\ell)}}\int_{\mathbb{R}_{+}^2} \paren{g^{\ell, j, f_p, \mu^{\vec{\zeta}}}(s,\vec{i}, \vec{y},\theta_1, \theta_2)  \la f_p,\Xi^{{\vec{\zeta}}, j}_{s-}\ra}^2 \,  d\bar{N}_{\ell}(s,\vec{i}, \vec{y},\theta_1, \theta_2)  \nonumber\\
& \text{(Similarly as \cref{eq:poissonBdDerive}, we can get)} \nonumber\\
&
\leq \frac{2}{\gamma}C t ||f_p||_{D,b}^4 + C ||f_p||_{D,b}^2 \int_0^{t\wedge\theta^{2b}_{\vec{\zeta},\kappa}} \paren{ \la f_p,\Xi^{{\vec{\zeta}}, j}_{s-}\ra}^2\, ds < \infty.
\end{align}
\end{proof}

\begin{lemma}\label{lem:densityIneq}
For $\Xi \in W^{-\Gamma_{1},a}$ and $a\geq 1$, recalling that $\mathcal{L}_j = D_j \Delta_x$, we have that
\begin{equation*}
  \la \Xi,  \mathcal{L}_j^* \Xi\ra_{-\Gamma_{1},a} \leq C||\Xi||^2_{-\Gamma_{1},a}.
\end{equation*}
\end{lemma}

\begin{proof}
By the Riesz Representation Theorem for Hilbert spaces, there exists a unique $\Psi \in W_0^{\Gamma_{1},a}$ such that
$$\la f, \Xi \ra = \la f, \Psi\ra_{\Gamma_{1},a}, \text{for all } f\in W_0^{\Gamma_{1},a}.$$
Let us denote $F(\Xi) = \Psi$, and note that $S := \{\Xi \in W^{-\Gamma_{1},a}: F(\Xi) \in W^{\Gamma_{1}+2,a}_0 \}$ is dense in $W^{-\Gamma_{1},a}$. We will first focus on such $\Xi \in S$. Then $ \mathcal{L}_j F(\Xi) \in W_0^{\Gamma_{1},a}$ and
$$\la \Xi,  \mathcal{L}_j^* \Xi\ra_{-\Gamma_{1},a} =  \la F(\Xi),  \mathcal{L}_j^* \Xi\ra = \la \mathcal{L}_j F(\Xi),   \Xi\ra = \la \mathcal{L}_j F(\Xi), F(\Xi)\ra_{\Gamma_{1},a}.$$
Then, by Lemma \ref{L:AuxiliaryBound1} we obtain that for some constant $C<\infty$
$$\la \Xi,  \mathcal{L}_j^* \Xi\ra_{-\Gamma_{1},a}  =  \la \mathcal{L}_j F(\Xi), F(\Xi)\ra_{\Gamma_{1},a}\leq C||F(\Xi)||^2_{\Gamma_{1},a} = C||\Xi||^2_{-\Gamma_{1},a}$$
By a density argument, the above result holds for any $\Xi \in W^{-\Gamma_{1},a}$.
\end{proof}

\begin{lemma}\label{lem:DiffMeas}
Let $\{ f_p \}_{p\geq 1}$ be a complete orthonormal system in $W_0^{\Gamma_{1},a}(\R^d)$ of class $C^\infty_0(\R^d)$. Under Assumptions~\ref{A:SobolevSpaceParameters} and~\ref{A:AssumptionRho} we have that there exists a finite constant $C<\infty$ such that
\begin{align}\label{eq:DiffMeas}
 &\E\sup_{t\in [0, T]}\sum_{p\geq 1} \sum_{\ell =  1}^L \int_{0}^{t\wedge\theta^{2a}_{\vec{\zeta},\kappa}}\la f_p,\Xi^{{\vec{\zeta}}, j}_{s}\ra \int_{\tilde{\mathbb{X}}^{(\ell)}}     \frac{1}{\vec{\alpha}^{(\ell)}!}   K_\ell\left(\vec{x}\right) \left(\int_{\mathbb{Y}^{(\ell)}}    \left( \sum_{r = 1}^{\beta_{\ell j}} f_p(y_r^{(j)}) \right) m_\ell\left(\vec{y} \, |\, \vec{x} \right)\, d \vec{y} - \sum_{r = 1}^{\alpha_{\ell j}} f_p(x_r^{(j)}) \right)\,\nonumber\\
&
\qquad\times \sqrt{\gamma} \paren{\lambda^{(\ell)}[\mu_{s}^{\vec{\zeta}}](d\vec{x})-\lambda^{(\ell)}[\bar{\mu}_{s}](d\vec{x})}\, ds \nonumber\\
 &
 \leq C \E \int_{0}^{T\wedge\theta^{2a}_{\vec{\zeta},\kappa}}
 \sum_{j=1}^{J}||\Xi_{s}^{{\vec{\zeta}},j}||^{2}_{-\Gamma_{1},a}\, ds
\end{align}
\end{lemma}

\begin{proof}

We begin by considering the term
\begin{equation*}
  I = \int_{\tilde{\mathbb{X}}^{(\ell)}}     \frac{1}{\vec{\alpha}^{(\ell)}!}   K_\ell\left(\vec{x}\right) \left(\sum_{r = 1}^{\alpha_{\ell j}} f_p(x_r^{(j)}) \right) \sqrt{\gamma} \paren{\lambda^{(\ell)}[\mu_{s}^{\vec{\zeta}}](d\vec{x})-\lambda^{(\ell)}[\bar{\mu}_{s}](d\vec{x})}.
\end{equation*}
\begin{enumerate}
\item If the $\ell$-th reaction is a unimolecular reaction $S_j \to \cdots$,
\begin{align*}
 I &=
 \int_{\R^d} K_\ell\left(x\right) f_p(x) \sqrt{\gamma} \paren{\mu_{s}^{\vec{\zeta},j}(dx)-\bar{\mu}_{s}^j(dx)}\,
= \int_{\R^d} K_\ell(x) f_p(x) \, \Xi^{{\vec{\zeta}},j}_{s}(dx) = \la K_\ell f_{p}, \Xi^{{\vec{\zeta}},j}_{s}\ra.
\end{align*}
\item If the $\ell$-th reaction is a bimolecular reaction $2 S_j \to \cdots$,
\begin{align*}
 I &= \tfrac{1}{2} \int_{\R^{2d}}  K_\ell\left(x,y\right) \left(f_p(x) + f_p(y)\right) \paren{\Xi^{{\vec{\zeta}},j}_{s}(dx)\mu_{s}^{\vec{\zeta},j}(dy) +  \bar{\mu}^j_{s}(dx)\Xi^{{\vec{\zeta}},j}_{s}(dy) }\, \\
&=\tfrac{1}{2}\left[\la \la K_\ell, \mu^{\vec{\zeta},j}_{s}\ra f_p, \Xi^{{\vec{\zeta}},j}_{s}\ra +
\la \la K_\ell f_{p}, \bar{\mu}^{j}_{s}\ra, \Xi^{\vec{\zeta},j}_{s}\ra+
\la \la K_\ell f_{p}, \mu^{\vec{\zeta},j}_{s}\ra, \Xi^{\vec{\zeta},j}_{s}\ra+\la \la K_\ell, \bar{\mu}^{j}_{s}\ra f_p, \Xi^{{\vec{\zeta}},j}_{s}\ra\right]
\end{align*}
\item If the $\ell$-th reaction is a bimolecular reaction $S_i + S_j \to \cdots$, with $i \neq j$
\begin{align*}
 I &= \int_{\R^{2d}} K_\ell\left(x,y\right) f_p(y) \paren{\Xi^{{\vec{\zeta}},i}_{s}(dx)\mu_{s}^{\vec{\zeta}, j}(dy) +  \bar{\mu}^i_{s}(dx)\Xi^{{\vec{\zeta}},j}_{s}(dy) }
=\left[\la \la K_\ell f_{p}, \mu^{\vec{\zeta},j}_{s}\ra, \Xi^{\vec{\zeta},i}_{s}\ra+\la \la K_\ell, \bar{\mu}^{i}_{s}\ra f_p, \Xi^{{\vec{\zeta}},j}_{s}\ra\right]
\end{align*}
\end{enumerate}

Next, we simplify
\begin{equation*}
  II = \int_{\tilde{\mathbb{X}}^{(\ell)}}     \frac{1}{\vec{\alpha}^{(\ell)}!}   K_\ell\left(\vec{x}\right) \left(\int_{\mathbb{Y}^{(\ell)}}    \left( \sum_{r = 1}^{\beta_{\ell j}} f_p(y_r^{(j)}) \right) m_\ell\left(\vec{y} \, |\, \vec{x} \right)\, d \vec{y}  \right) \sqrt{\gamma} \paren{\lambda^{(\ell)}[\mu_{s}^{\vec{\zeta}}](d\vec{x})-\lambda^{(\ell)}[\bar{\mu}_{s}](d\vec{x})}
\end{equation*}
by expanding out all the allowable reaction types:
\begin{enumerate}
\item If the $\ell$-th reaction is of the form $S_i \to S_j$,
\begin{align*}
  II &= \int_{\tilde{\mathbb{X}}^{(\ell)}}  K_\ell\left(x\right) \left(\int_{\R^d}  f_p(y)  m_\ell\left(y \, |\, x \right)\, dy \right) \sqrt{\gamma} \paren{\mu_{s}^{\vec{\zeta},i}(dx)-\bar{\mu}_{s}^i(dx)}\\
  &= \int_{\R^d} K_\ell\left(x\right)  f_p(x)  \Xi^{{\vec{\zeta}},i}_{s} (dx)
  = \la K_\ell f_{p}, \Xi^{{\vec{\zeta}},i}_{s}\ra.
\end{align*}
\item If the $\ell$-th reaction is of the form $S_i \to 2 S_j$, we set the operator
\begin{equation*}
  \mathcal{D}_{1}^{(1-\alpha_i)} f_{p}(x)=K_\ell\left(x\right) \left(\int_{\R^{d}}    f_p\paren{x + \paren{1-\alpha_i}w}  \rho(w) \, dw  \right)
\end{equation*}
and we have
\begin{align*}
 II &= \int_{\R^d} K_\ell\left(x\right) \left(\int_{\R^{2d}}  \left( f_p(y) + f_p(z)\right) m_\ell\left(y, z \, |\, x \right)\,  dy\, dz   \right)  \Xi^{{\vec{\zeta}},i}_{s} (dx)\\
 &= \int_{\R^d} K_\ell\left(x\right) \left(\int_{\R^{2d}} \left( f_p(y) + f_p(z)\right) \rho(y-z) \sum_{k = 1}^{I} p_k \delta\paren{x - \paren{\alpha_k y + (1-\alpha_k) z}} \, dy\, dz  \right)  \Xi^{{\vec{\zeta}},i}_{s} (dx)\\
 &= \sum_{k = 1}^{I} p_k  \int_{\R^d} K_\ell(x) \left(\int_{\R^{2d}} \paren{f_p(w+z) + f(z)} \rho(w) \delta\paren{x - \paren{\alpha_k w + z}} \, dz\, dw  \right)  \Xi^{{\vec{\zeta}},i}_{s} (dx)\\
 &= \sum_{k = 1}^{I} p_k  \int_{\R^d}  K_\ell(x) \left(\int_{\R^{d}} \paren{f_p\paren{x + \paren{1-\alpha_k}w} + f(x - \alpha_k w)} \rho(w) \, dw  \right)  \Xi^{{\vec{\zeta}},i}_{s} (dx)\\
 &= \sum_{k = 1}^{I} p_k \brac{ \la \mathcal{D}_{1}^{(1-\alpha_k)}f_{p},\Xi^{{\vec{\zeta}},i}_{s}\ra + \la \mathcal{D}_{1}^{(-\alpha_k)}f_{p},\Xi^{{\vec{\zeta}},i}_{s}\ra}.
\end{align*}
\item If the $\ell$-th reaction is of the form $S_i \to S_j + S_k$ for $j \neq n$, mimicking the previous case we find
\begin{align*}
  II &= \int_{\R^d} K_\ell\left(x\right) \left(\int_{\R^{2d}} f_p(y) m_\ell\left(y, z \, |\, x \right)\,  dy\, dz   \right)  \Xi^{{\vec{\zeta}},i}_{s} (dx)\\
  &= \sum_{n = 1}^{I} p_n \la \mathcal{D}_{1}^{(1-\alpha_n)}f_{p},\Xi^{{\vec{\zeta}},i}_{s}\ra.
 \end{align*}
\item If the $\ell$-th reaction is of the form $S_i + S_k \to S_j$,
\begin{align*}
 II &= \tfrac{1}{\vec{\alpha}^{(\ell)}!} \int_{\R^{2d}} K_\ell\left(x, y\right) \left(\int_{\R^d}    f_p(z) m_\ell\left(z |\, x, y\right)\, d z  \right)   \paren{\Xi^{{\vec{\zeta}},k}_{s}(dx)\mu_{s}^{\vec{\zeta}, i}(dy) +  \bar{\mu}^k_{s}(dx)\Xi^{{\vec{\zeta}},i}_{s}(dy) }\\
&= \tfrac{1}{\vec{\alpha}^{(\ell)}!}  \sum_{n = 1}^{I} p_n \int_{\R^{2d}} K_\ell\left(x, y\right)   f_p\paren{\alpha_n x + (1-\alpha_n) y} \paren{\Xi^{{\vec{\zeta}},k}_{s}(dx)\mu_{s}^{\vec{\zeta}, i}(dy) +  \bar{\mu}^k_{s}(dx)\Xi^{{\vec{\zeta}},i}_{s}(dy) }\\
&= \tfrac{1}{\vec{\alpha}^{(\ell)}!}  \sum_{n = 1}^{I} p_n \Big[\la \la K_\ell(x, y)   f_p\paren{\alpha_n x + (1-\alpha_n) y} \mu^{\vec{\zeta},i}_{s}(dy)\ra, \Xi^{{\vec{\zeta}},k}_{s}(dx)\ra \\
&\qquad\qquad\qquad\qquad\qquad\qquad +\la \la K_\ell(x, y)   f_p\paren{\alpha_n x + (1-\alpha_n) y} \bar{\mu}^{k}_{s}(dx)\ra, \Xi^{{\vec{\zeta}},i}_{s}(dy)\ra \Big].
\end{align*}
\item If the $\ell$-th reaction is of the form $S_i + S_k \to 2S_j$,
\begin{align*}
 II&= \tfrac{1}{\vec{\alpha}^{(\ell)}!}\int_{\R^{2d}} K_\ell(x,y) \left(\int_{\R^{2d}}    \left( f_p(z) + f_p(w) \right) m_\ell\left(z, w \, |\,x, y \right)\, d z\, dw  \right) \paren{\Xi^{{\vec{\zeta}},k}_{s}(dx)\mu_{s}^{\vec{\zeta}, i}(dy) +  \bar{\mu}^k_{s}(dx)\Xi^{{\vec{\zeta}},i}_{s}(dy) }\\
 &=  \tfrac{1}{\vec{\alpha}^{(\ell)}!} \int_{\R^{2d}} K_\ell(x, y)    \left( f_p(x) + f_p(y) \right)  \paren{\Xi^{{\vec{\zeta}},k}_{s}(dx)\mu_{s}^{\vec{\zeta}, i}(dy) +  \bar{\mu}^k_{s}(dx)\Xi^{{\vec{\zeta}},i}_{s}(dy)}\\
 &= \tfrac{1}{\vec{\alpha}^{(\ell)}!}\left[\la \la K_\ell, \mu^{\vec{\zeta},i}_{s}\ra f_p, \Xi^{{\vec{\zeta}},k}_{s}\ra +
\la \la K_\ell f_{p}, \bar{\mu}^{k}_{s}\ra, \Xi^{\vec{\zeta},i}_{s}\ra+
\la \la K_\ell f_{p}, \mu^{\vec{\zeta},i}_{s}\ra, \Xi^{\vec{\zeta},k}_{s}\ra+\la \la K_\ell, \bar{\mu}^{k}_{s}\ra f_p, \Xi^{{\vec{\zeta}},i}_{s}\ra\right]
\end{align*}
\item If the $\ell$-th reaction is of the form $S_i + S_k \to S_j + S_r$ for $r \neq j$, mimicking the previous case we find
\begin{align*}
 II&= \tfrac{1}{\vec{\alpha}^{(\ell)}!}\int_{\R^{2d}} K_\ell(x,y) \left(\int_{\R^{2d}} f_p(z) m_\ell\left(z, w \, |\,x, y \right)\, d z\, dw  \right) \paren{\Xi^{{\vec{\zeta}},k}_{s}(dx)\mu_{s}^{\vec{\zeta}, i}(dy) +  \bar{\mu}^k_{s}(dx)\Xi^{{\vec{\zeta}},i}_{s}(dy) }\\
 &= \tfrac{1}{\vec{\alpha}^{(\ell)}!}\paren{\la \la K_\ell, \mu^{\vec{\zeta},i}_{s}\ra f_p, \Xi^{{\vec{\zeta}},k}_{s}\ra +
\la \la K_\ell f_{p}, \bar{\mu}^{k}_{s}\ra, \Xi^{\vec{\zeta},i}_{s}\ra}.
\end{align*}
\end{enumerate}

In each case studied above, the upper bound is of the form $\la \mathcal{D} f_{p}, \Xi_{s}^{\zeta}\ra$ where the operator $\mathcal{D}$ is either a local operator as $\mathcal{D}f_{p}(x)=K_\ell(x) f_{p}(x)$, $\mathcal{D}f_{p}(x,s)=\la K_\ell(x), \mu^{\vec{\zeta},i}_{s}\ra f_p(x)$ or a non-local operator of the form $\mathcal{D}f_{p}(x)=\mathcal{D}_{1}^{(1-\alpha_n)}f_{p}(x)$ (as defined above) or of the form
$\mathcal{D}f_{p}(x,s)=\la K_\ell f_{p}, \mu^{\vec{\zeta},i}_{s}\ra$. Due to Assumption \ref{A:AssumptionK} and for the operator $\mathcal{D}=\mathcal{D}_{1}^{(1-\alpha_n)}$ Assumption \ref{A:AssumptionRho} as well, in all of the cases, the operators $\mathcal{D}$ are linear operators from
$W_0^{\Gamma_{1},a}$ into $W_0^{\Gamma_{1},a}$.

Examining the left hand side of (\ref{eq:DiffMeas}), for such operators $\mathcal{D}$ we see that in all of these cases estimates of the following form hold:
\begin{equation}
  \begin{aligned}
    \int_{0}^{t\wedge\theta^{2a}_{\vec{\zeta},\kappa}}\sum_{p\geq 1}\la f_p,\Xi^{{\vec{\zeta}}, j}_{s}\ra\la \mathcal{D} f_p,\Xi^{{\vec{\zeta}}, i}_{s}\ra ds
    &= \int_{0}^{t\wedge\theta^{2a}_{\vec{\zeta},\kappa}}\la \Xi^{{\vec{\zeta}}, j}_{s}, \mathcal{D}^{*}\Xi^{{\vec{\zeta}}, i}_{s}\ra_{-\Gamma_{1},a} ds\\
    &\leq C \int_{0}^{t\wedge\theta^{2a}_{\vec{\zeta},\kappa}}\left[\|\Xi^{{\vec{\zeta}}, j}_{s}\|^{2}_{-\Gamma_{1},a}+\|\Xi^{{\vec{\zeta}}, i}_{s}\|^{2}_{-\Gamma_{1},a}\right]ds.
    \end{aligned}
\end{equation}

Let us prove this for one possible form of the operators $\mathcal{D}$, say for $\mathcal{D}f_{p}(x,s;\zeta,i)=\la K_\ell(x,y) f_{p}(y), \mu^{\vec{\zeta},i}_{s}(dy)\ra$. The rest of the cases can being treated similarly\footnote{Even though we do not show this here, showing the desired bound for the choice $\mathcal{D}f_{p}(x)=\mathcal{D}_{(1-\alpha_n)}^{i}f_{p}(x)$ leads to the requirements $\|\rho\|_{L^{1}}<\infty$  and $\int_{\mathbb{R}^{d}}|w|^{2a}\rho(w)dw<\infty$.}. Let us first show that $\mathcal{D}:W_0^{\Gamma_{1},a}\mapsto W_0^{\Gamma_{1},a}$ and for $f_{p}\in W_0^{\Gamma_{1},a}$ the bound $\|\mathcal{D}f_{p}\|^{2}_{\Gamma_{1},a}\leq C \|f_{p}\|^{2}_{\Gamma_{1},a}$ holds, for some unimportant constant $C<\infty$. Indeed, the following computation holds
\begin{align}
\|\mathcal{D}f_{p}\|^{2}_{\Gamma_{1},a}&=\|\la K_\ell(x,\cdot) f_{p}, \mu^{\vec{\zeta},i}_{s}\ra\|^{2}_{\Gamma_{1},a}\nonumber\\
&=\sum_{k\leq\Gamma_{1}}\int_{\mathbb{R}^{d}}\frac{1}{1+|x|^{2a}}\left(\la K_\ell (x,\cdot)f_{p}, \mu^{\vec{\zeta},i}_{s}\ra^{(k)}\right)^{2}dx\nonumber\\
&=\sum_{k\leq\Gamma_{1}}\int_{\mathbb{R}^{d}}\frac{1}{1+|x|^{2a}}\left(\la K_\ell^{(k)} (x,\cdot)f_{p}, \mu^{\vec{\zeta},i}_{s}\ra\right)^{2}dx\nonumber\\
&\leq C \sum_{k\leq\Gamma_{1}}\|\partial^{(k)}_{x}K_{\ell}\|^{2}_{\infty}\int_{\mathbb{R}^{d}}\frac{1}{1+|x|^{2a}}\la (f_{p})^{2}, \mu^{\vec{\zeta},i}_{s}\ra dx\nonumber\\
&\leq C \sum_{k\leq\Gamma_{1}}\|\partial^{(k)}_{x}K_{\ell}\|^{2}_{\infty}
\paren{\sup_{x\in\mathbb{R}^{d}}\frac{(f_{p}(x))^{2}}{1+|x|^{2a}}}\la 1+|x|^{2a}, \mu^{\vec{\zeta},i}_{s}\ra\int_{\mathbb{R}^{d}}\frac{1}{1+|x|^{2a}} dx\nonumber\\
&\leq C \left[\sum_{k\leq\Gamma_{1}}\|\partial^{(k)}_{x}K_{\ell}\|^{2}_{\infty}
\la 1+|x|^{2a}, \mu^{\vec{\zeta},i}_{s}\ra\int_{\mathbb{R}^{d}}\frac{1}{1+|x|^{2a}} dx\right]\|f_{p}\|^{2}_{\Gamma_{1},a}\nonumber\\
&\leq C\left[\sum_{k\leq\Gamma_{1}}\|\partial^{(k)}_{x}K_{\ell}\|^{2}_{\infty}
\la 1+|x|^{2a}, \mu^{\vec{\zeta},i}_{s}\ra\right]\|f_{p}\|^{2}_{\Gamma_{1},a}\nonumber\\
&= C\|f_{p}\|^{2}_{\Gamma_{1},a},
\label{Eq:MapingPropertyD}
\end{align}
where we have used the Sobolev embedding theorem, the boundedness of $\max_{k\leq\Gamma_{1}}\|\partial^{(k)}_{x}K_{\ell}\|^{2}_{\infty}$, and that $2a>d$ by assumption. This estimate also demonstrates why weighted spaces are needed. In particular, without weights (take for example $a=0$) the bound on the second inequality would have been in terms of $\int_{\mathbb{R}^{d}} dx=\infty$ instead of $\int_{\mathbb{R}^{d}}\frac{1}{1+|x|^{2a}} dx<\infty$.

Then, as in Lemma \ref{lem:densityIneq}, the Riesz representation theorem for Hilbert spaces allows us to show that there exists a unique $\Psi^{j}, \Psi^{i} \in W_0^{\Gamma_{1},a}$ such that (omitting the $\vec{\zeta}$ for notational simplicity)
\begin{equation*}
  \la \Xi^{j},  \mathcal{D}^* \Xi^{i}\ra_{-\Gamma_{1},a} =  \la \Psi^{j},  \mathcal{D}^* \Xi^{i}\ra = \la \mathcal{D} \Psi^{j},   \Xi^{i}\ra = \la \mathcal{D} \Psi^{j}, \Psi^{i}\ra_{\Gamma_{1},a}.
\end{equation*}

Hence, using Young's inequality, pulling the supremum of $\partial^{(k)}_{x}K_{\ell}$ outside the integration and omitting the time dependence for now, we shall have (following a similar computation as in (\ref{Eq:MapingPropertyD}))
\begin{align}
\la \Xi^{j}, \mathcal{D}^{*}\Xi^{ i}\ra_{-\Gamma_{1},a}&=\la \mathcal{D} \Psi^{j}, \Psi^{i}\ra_{\Gamma_{1},a}\nonumber\\
&=\sum_{k\leq\Gamma_{1}}\int_{\mathbb{R}^{d}}\frac{1}{1+|x|^{2a}}\la K_\ell(x,\cdot) \Psi^{j}, \mu^{\vec{\zeta},i}_{s}\ra^{(k)} (\Psi^{i}(x))^{(k)}dx\nonumber\\
&\leq \frac{1}{2}\sum_{k\leq\Gamma_{1}}\int_{\mathbb{R}^{d}}\frac{1}{1+|x|^{2a}}\left(\la K_\ell (x,\cdot)\Psi^{j}, \mu^{\vec{\zeta},i}_{s}\ra^{(k)}\right)^{2}dx+\frac{1}{2}\|\Psi^{i}\|^{2}_{\Gamma_{1},a}\nonumber\\
&\leq C\left[\sum_{k\leq\Gamma_{1}}\|\partial^{(k)}_{x}K_{\ell}\|^{2}_{\infty}
\la 1+|x|^{2a}, \mu^{\vec{\zeta},i}_{s}\ra\right]\|\Psi^{j}\|^{2}_{\Gamma_{1},a}+\frac{1}{2}\|\Psi^{i}\|^{2}_{\Gamma_{1},a}\nonumber\\
&\leq C\left[\sum_{k\leq\Gamma_{1}}\|\partial^{(k)}_{x}K_{\ell}\|^{2}_{\infty}
\la 1+|x|^{2a}, \mu^{\vec{\zeta},i}_{s}\ra\right]\|\Xi^{j}\|^{2}_{-\Gamma_{1},a}+\frac{1}{2}\|\Xi^{i}\|^{2}_{-\Gamma_{1},a},\nonumber
\end{align}
where to go from the first inequality to the second we did the same computation as in  (\ref{Eq:MapingPropertyD}).  From the preceding estimate we find
\begin{align}
\mathbb{E} \sup_{t \in \brac{0,T}} \int_{0}^{t\wedge\theta^{2a}_{\vec{\zeta},\kappa}}\la \Xi^{j}_{s}, \mathcal{D}^{*}\Xi^{ i}_{s}\ra_{-\Gamma_{1},a}ds&\leq C\mathbb{E}\int_{0}^{T\wedge\theta^{2a}_{\vec{\zeta},\kappa}}\left(\|\Xi^{j}_{s}\|^{2}_{-\Gamma_{1},a}+\|\Xi^{i}_{s}\|^{2}_{-\Gamma_{1},a}\right)ds.
\end{align}

We conclude that there is a constant $C<\infty$ that depends on $\kappa,C(K), C_{\circ}, \|\rho\|_{L^{1}}, \int_{\mathbb{R}^{d}}(1+|x|^{2a})\rho(x)dx$, and the moment bound from Lemma \ref{L:MomentBound} such that
\begin{align*}
 &\E\sup_{t\in [0, T]}\sum_{p\geq 1} \sum_{\ell =  1}^L \int_{0}^{t\wedge\theta^{2a}_{\vec{\zeta},\kappa}}\la f_p,\Xi^{{\vec{\zeta}}, j}_{s}\ra \int_{\tilde{\mathbb{X}}^{(\ell)}}     \frac{1}{\vec{\alpha}^{(\ell)}!}   K_\ell\left(\vec{x}\right) \left(\int_{\mathbb{Y}^{(\ell)}}    \left( \sum_{r = 1}^{\beta_{\ell j}} f_p(y_r^{(j)}) \right) m_\ell\left(\vec{y} \, |\, \vec{x} \right)\, d \vec{y} - \sum_{r = 1}^{\alpha_{\ell j}} f_p(x_r^{(j)}) \right)\,\\
&\qquad\times \sqrt{\gamma} \paren{\lambda^{(\ell)}[\mu_{s}^{\vec{\zeta}}](d\vec{x})-\lambda^{(\ell)}[\bar{\mu}_{s}](d\vec{x})}\, ds \nonumber\\
& \leq  C \E \int_{0}^{T\wedge\theta^{2a}_{\vec{\zeta},\kappa}}
 \sum_{j=1}^{J}||\Xi^{{\vec{\zeta}},j}_{s}||^{2}_{-\Gamma_{1},a}\, ds.
\end{align*}

\end{proof}


\begin{lemma}\label{lem:DiffPlaceMeas}
Let $\{ f_p \}_{p\geq 1}$ be a complete orthonormal system in $W_0^{\Gamma_{1},a}(\R^d)$ of class $C^\infty_0(\R^d)$. Then, with $b>a+d/2$, we have that there exists a finite constant $C<\infty$ such that
\begin{align}
&\E\sup_{t\in [0, T]}\sum_{p\geq 1}\sum_{\ell =  1}^L \int_{0}^{t\wedge\theta^{2b}_{\vec{\zeta},\kappa}}\la f_p,\Xi^{{\vec{\zeta}}, j}_{s}\ra \int_{\tilde{\mathbb{X}}^{(\ell)}}     \frac{1}{\vec{\alpha}^{(\ell)}!}   K_\ell\left(\vec{x}\right) \left(\int_{\mathbb{Y}^{(\ell)}}    \left( \sum_{r = 1}^{\beta_{\ell j}} f_p(y_r^{(j)}) \right) \sqrt{\gamma}\paren{m^\eta_\ell\left(\vec{y} \, |\, \vec{x} \right)-m_\ell\left(\vec{y} \, |\, \vec{x} \right)}\, d \vec{y}  \right)\, \lambda^{(\ell)}[\mu_{s}^{\vec{\zeta}}](d\vec{x})\, ds\nonumber\\
&\leq C\left[\E \int_{0}^{T\wedge\theta^{2b}_{\vec{\zeta},\kappa}} ||\Xi^{{\vec{\zeta}}, j}_{s}||^{2}_{-\Gamma_{1},a} ds + |\sqrt{\gamma}\eta|^{2} \right].
\end{align}
The constant $C<\infty$ depends on $\kappa,C(K), C_{\circ}, \|\rho\|_{L^{1}}, \int_{\mathbb{R}^{d}}(1+|x|^{2a})\rho(x)dx$ and the moment bound from Lemma \ref{L:MomentBound}.
\end{lemma}

\begin{proof}
By Young's inequality, Lemma B.2 and the proof of Lemma B.1 in~\cite{IMS:2022}, and the moment bound Lemma~\ref{L:MomentBound}, we have for $b>a+d/2$ that
\begin{align*}
&2 \, \E\sup_{t\in [0, T]}\sum_{p\geq 1}\sum_{\ell =  1}^L \int_{0}^{t\wedge\theta^{2b}_{\vec{\zeta},\kappa}}\la f_p,\Xi^{{\vec{\zeta}}, j}_{s}\ra \int_{\tilde{\mathbb{X}}^{(\ell)}}     \frac{1}{\vec{\alpha}^{(\ell)}!}   K_\ell\left(\vec{x}\right) \left(\int_{\mathbb{Y}^{(\ell)}}    \left( \sum_{r = 1}^{\beta_{\ell j}} f_p(y_r^{(j)}) \right) \sqrt{\gamma}\paren{m^\eta_\ell\left(\vec{y} \, |\, \vec{x} \right)-m_\ell\left(\vec{y} \, |\, \vec{x} \right)}\, d \vec{y}  \right)\, \lambda^{(\ell)}[\mu_{s}^{\vec{\zeta}}](d\vec{x})\, ds\\
&\leq \E\sup_{t\in [0, T]}\sum_{p\geq 1}\sum_{\ell =  1}^L \int_{0}^{t\wedge\theta^{2b}_{\vec{\zeta},\kappa}}\la f_p,\Xi^{{\vec{\zeta}}, j}_{s}\ra^{2} ds \\
 &+
 \E\sup_{t\in [0, T]}\sum_{p\geq 1}\sum_{\ell =  1}^L \int_{0}^{t\wedge\theta^{2b}_{\vec{\zeta},\kappa}} \paren{\int_{\tilde{\mathbb{X}}^{(\ell)}}     \frac{1}{\vec{\alpha}^{(\ell)}!}   K_\ell\left(\vec{x}\right) \left(\int_{\mathbb{Y}^{(\ell)}}    \left( \sum_{r = 1}^{\beta_{\ell j}} f_p(y_r^{(j)}) \right) \sqrt{\gamma}\paren{m^\eta_\ell\left(\vec{y} \, |\, \vec{x} \right)-m_\ell\left(\vec{y} \, |\, \vec{x} \right)}\, d \vec{y}  \right)\ \lambda^{(\ell)}[\mu_{s}^{\vec{\zeta}}](d\vec{x})}^{2}\, ds\\
&\leq C \left[\E \int_{0}^{T\wedge\theta^{2b}_{\vec{\zeta},\kappa}} ||\Xi^{{\vec{\zeta}}, j}_{s}||^{2}_{-\Gamma_{1},a} ds+\gamma \eta^{2} \sum_{p\geq 1}\|f_{p}\|^{2}_{C^{1,b}}\right]\\
&\leq C \left[\E \int_{0}^{T\wedge\theta^{2b}_{\vec{\zeta},\kappa}} ||\Xi^{{\vec{\zeta}}, j}_{s}||^{2}_{-\Gamma_{1},a} ds+\gamma \eta^{2} \sum_{p\geq 1}\|f_{p}\|^{2}_{1+D,b}\right],
\end{align*}
where the last line follows by Sobolev embedding with $D>d/2$. Now because $\Gamma_{1}>d/2+1+D$ and $b-a>d/2$, the embedding $W^{\Gamma_{1},a}_{0}\hookrightarrow W^{1+D,b}_{0}$ is of Hilbert-Schmidt type, so $\sum_{p\geq 1}||f_p||^{2}_{1+D,b}<\infty$. This concludes the proof.
\end{proof}
\begin{lemma}\label{lem:negativeNormBd}
Let $a\geq D$, $b\geq 2D$ and $\Gamma_{1}\geq 2D+1$, $D=1+\ceil*{d/2}$. For all $\sqrt{\gamma} \eta$ sufficiently small, see Assumption~\ref{A:AssumptionParameterRelativeRates}, there exists a constant $C$ such that
\begin{equation}\label{eq:negativeNormBd}
\sup_{\vec{\zeta}\in(0,1)^{2}}\E [\sup_{t\in[0, T]}||\Xi^{\vec{\zeta}, j}_{t\wedge\theta^{2b}_{\vec{\zeta},\kappa}}||^2_{-\Gamma_{1},a}] < C,
\end{equation}
for $j = 1, 2, \cdots, J$.
\end{lemma}

\begin{proof}
Let $\{ f_p \}_{p\geq 1}$ be a complete orthonormal system in $W_0^{\Gamma_{1},a}(\R^d)$ of class $C^\infty_0(\R^d)$. By applying Ito's formula on \cref{eq:fluctuationProcess}, we can obtain
\begin{align}\label{eq:fp_squared}
&\la f_p,\Xi^{{\vec{\zeta}}, j}_{t}\ra^2
=\la f_p,\Xi^{{\vec{\zeta}}, j}_{0}\ra^2 + 2\int_{0}^{t} \la f_p,\Xi^{{\vec{\zeta}}, j}_{s}\ra \la (\mathcal{L}_j f_p)(x) , \Xi^{\vec{\zeta}, j}_{s-} (dx)\ra ds  \nonumber\\
 &
\quad + 2\int_0^t  \la f_p, \Xi^{\vec{\zeta}, j}_{s-}\ra d\cC^{\vec{\zeta}, j}_s (f_p)+ \la \cC^{\vec{\zeta}, j}\ra_t (f_p)\nonumber\\
&
\quad+ 2\sum_{\ell =  1}^L \int_{0}^{t}\la f_p,\Xi^{{\vec{\zeta}}, j}_{s}\ra \int_{\tilde{\mathbb{X}}^{(\ell)}}     \frac{1}{\vec{\alpha}^{(\ell)}!}   K_\ell\left(\vec{x}\right) \left(\int_{\mathbb{Y}^{(\ell)}}    \left( \sum_{r = 1}^{\beta_{\ell j}} f_p(y_r^{(j)}) \right) m_\ell\left(\vec{y} \, |\, \vec{x} \right)\, d \vec{y} - \sum_{r = 1}^{\alpha_{\ell j}} f_p(x_r^{(j)}) \right)\,\nonumber\\
&
\qquad\times \sqrt{\gamma} \paren{\lambda^{(\ell)}[\mu_{s}^{\vec{\zeta}}](d\vec{x})-\lambda^{(\ell)}[\bar{\mu}_{s}](d\vec{x})}\, ds \nonumber\\
&
\quad+ 2\sum_{\ell =  1}^L \int_{0}^{t}\la f_p,\Xi^{{\vec{\zeta}}, j}_{s}\ra \int_{\tilde{\mathbb{X}}^{(\ell)}}     \frac{1}{\vec{\alpha}^{(\ell)}!}   K_\ell\left(\vec{x}\right) \left(\int_{\mathbb{Y}^{(\ell)}}    \left( \sum_{r = 1}^{\beta_{\ell j}} f_p(y_r^{(j)}) \right) \sqrt{\gamma}\paren{m^\eta_\ell\left(\vec{y} \, |\, \vec{x} \right)-m_\ell\left(\vec{y} \, |\, \vec{x} \right)}\, d \vec{y}  \right)\, \lambda^{(\ell)}[\mu_{s}^{\vec{\zeta}}](d\vec{x})\, ds\nonumber\\
&
\quad + \sum_{\ell = 1}^L \int_{0}^{t}\int_{\mathbb{I}^{(\ell)}} \int_{\mathbb{Y}^{(\ell)}}\int_{\mathbb{R}_{+}^2} \paren{ \paren{\la f_p,\Xi^{{\vec{\zeta}}, j}_{s-}\ra +  g^{\ell, j, f_p, \mu^{\vec{\zeta}}}(s,\vec{i}, \vec{y},\theta_1, \theta_2) }^2 - \la f_p,\Xi^{{\vec{\zeta}}, j}_{s-}\ra^2} \,d\tilde{N}_{\ell}(s,\vec{i}, \vec{y},\theta_1, \theta_2)  \nonumber\\
&
\quad + \sum_{\ell = 1}^L \int_{0}^{t}\int_{\mathbb{I}^{(\ell)}} \int_{\mathbb{Y}^{(\ell)}}\int_{\mathbb{R}_{+}^2}\paren{ \paren{\la f_p,\Xi^{{\vec{\zeta}}, j}_{s-}\ra + g^{\ell, j, f_p, \mu^{\vec{\zeta}}}(s,\vec{i}, \vec{y},\theta_1, \theta_2) }^2 - \la f_p,\Xi^{{\vec{\zeta}}, j}_{s-}\ra^2 \right. \nonumber\\
&
\qquad\qquad\left.  - 2g^{\ell, j, f_p, \mu^{\vec{\zeta}}}(s,\vec{i}, \vec{y},\theta_1, \theta_2) \times \la f_p,\Xi^{{\vec{\zeta}}, j}_{s-}\ra} \,  d\bar{N}_{\ell}(s,\vec{i}, \vec{y},\theta_1, \theta_2)  \nonumber \\
\end{align}

We will sum over all $p\geq 1$, take the supremum over $t\in[0, T]$ and then take expectations on both sides of \cref{eq:fp_squared}. Using Parseval's identity, the first line of \cref{eq:fp_squared}'s right side becomes
\begin{align*}
&\E \sum_{p\geq1} \la f_p, \Xi^{{\vec{\zeta}}, j}_0 \ra^2  + 2\E \sup_{t\in[0, T]}\sum_{p\geq1}\int_0^{t\wedge\theta^{2b}_{\vec{\zeta},\kappa}} \la f_p, \Xi^{{\vec{\zeta}}, j}_{s-}\ra\la (\mathcal{L}_1f_p)(x) , \Xi^{{\vec{\zeta}}, j}_{s-} (dx)\ra \, ds\\
&= \E \sum_{p\geq1} \la f_p, \Xi^{{\vec{\zeta}}, j}_0 \ra^2  + 2\E\sup_{t\in[0, T]} \int_0^{t\wedge\theta^{2b}_{\vec{\zeta},\kappa}}  \sum_{p\geq1} \la f_p, \Xi^{{\vec{\zeta}}, j}_{s-}\ra\la f_p ,\mathcal{L}_1^* \Xi^{{\vec{\zeta}}, j}_{s-}\ra \, dt\nonumber\\
&= \E \sum_{p\geq1} \la f_p, \Xi^{{\vec{\zeta}}, j}_0 \ra^2  + 2\E\sup_{t\in[0, T]} \int_0^{t\wedge\theta^{2b}_{\vec{\zeta},\kappa}}   \la\Xi^{{\vec{\zeta}}, j}_{s-}, \mathcal{L}_1^* \Xi^{{\vec{\zeta}}, j}_{s-}\ra \, dt\nonumber\\
&\leq  \E ||\Xi^{{\vec{\zeta}}, j}_{0}||_{-\Gamma_{1},a}^2  + C  \E \int_0^{T\wedge\theta^{2b}_{\vec{\zeta},\kappa}}  ||\Xi^{{\vec{\zeta}}, j}_{s-}||_{-\Gamma_{1},a}^2 \, dt \qquad\text{(by \cref{lem:densityIneq})}
\end{align*}

In the second line of \cref{eq:fp_squared}, by \cref{lem:contMartingale}, the stochastic integral $2\int_0^{t\wedge\theta^{2b}_{\vec{\zeta},\kappa}}  \la f_p, \Xi^{{\vec{\zeta}}, j}_{s-}\ra d\cC^{\vec{\zeta}, j}_s (f_p)$ is a martingale. By Jensen's inequality, Doob's martingale inequality (see \cref{lem:contMartingale}), we get
\begin{equation}\label{eq:JensenAndDoob}
\begin{aligned}
\E \sup_{t\in[0, T]}\sum_{p\geq1} \int_0^{t\wedge\theta^{2b}_{\vec{\zeta},\kappa}}  \la f_p, \Xi^{{\vec{\zeta}}, j}_{s-}\ra d\cC^{\vec{\zeta}, j}_s(f_p)
&\leq C(D_j,\kappa) \sqrt{  \mathbb{E}\int_0^{T\wedge\theta^{2b}_{\vec{\zeta},\kappa}}   || \Xi^{{\vec{\zeta}}, j}_{s-}||_{-\Gamma_{1},a}^2  \, ds} \\
&\leq C \paren{1 + \mathbb{E}\int_0^{T\wedge\theta^{2b}_{\vec{\zeta},\kappa}}   || \Xi^{{\vec{\zeta}}, j}_{s-}||_{-\Gamma_{1},a}^2  \, ds}.
\end{aligned}
\end{equation}

%
%
By \cref{eq:contMartingaleQV}, the term arising from $\la \cC^{\vec{\zeta}, j}\ra_{t\wedge\theta^{2b}_{\vec{\zeta},\kappa}} (f_p)$ on the second line of \cref{eq:fp_squared} is bounded by
\begin{equation*}
\begin{aligned}\E \sup_{t\in[0, T]}\sum_{p\geq1}\la \cC^{\vec{\zeta}, j}\ra_{t\wedge\theta^{2b}_{\vec{\zeta},\kappa}} (f_p) &=\E\sum_{p\geq1} \int_{0}^{T\wedge\theta^{2b}_{\vec{\zeta},\kappa}} \la 2D_{j} \paren{\frac{ \partial f_p}{\partial Q}(x)}^2, \mu_{s-}^{{\vec{\zeta}}, j}(dx)\ra\,ds\\
&\leq \left(2D_{j}  T \sup_{t\leq T}h_{t}^{\vec{\zeta},j,2b}\right) \sum_{p\geq1}||f_p||^2_{1+D,b}.
\end{aligned}
\end{equation*}
For the third line of \cref{eq:fp_squared}, using \cref{lem:DiffMeas} we obtain the bound
\begin{equation*}
\begin{aligned}
 &2\E\sup_{t\in [0, T]}\sum_{p\geq 1} \sum_{\ell =  1}^L \int_{0}^{t\wedge\theta^{2b}_{\vec{\zeta},\kappa}}\la f_p,\Xi^{{\vec{\zeta}}, j}_{s}\ra \int_{\tilde{\mathbb{X}}^{(\ell)}}     \frac{1}{\vec{\alpha}^{(\ell)}!}   K_\ell\left(\vec{x}\right) \left(\int_{\mathbb{Y}^{(\ell)}}    \left( \sum_{r = 1}^{\beta_{\ell j}} f_p(y_r^{(j)}) \right) m_\ell\left(\vec{y} \, |\, \vec{x} \right)\, d \vec{y} - \sum_{r = 1}^{\alpha_{\ell j}} f_p(x_r^{(j)}) \right)\,\\
&\qquad\times \sqrt{\gamma} \paren{\lambda^{(\ell)}[\mu_{s}^{\vec{\zeta}}](d\vec{x})-\lambda^{(\ell)}[\bar{\mu}_{s}](d\vec{x})}\, ds \\
&\leq  C  \E \int_{0}^{T\wedge\theta^{2b}_{\vec{\zeta},\kappa}} \sum_{k=1}^J ||\Xi^{{\vec{\zeta}},k}_{s}||^2_{-\Gamma_{1},a} \, ds \\
\end{aligned}
\end{equation*}
For the fourth line, using \cref{lem:DiffPlaceMeas} we obtain the bound
\begin{equation*}
\begin{aligned}
&\begin{multlined}[t]
  2\E\sup_{t\in [0, T]}\sum_{p\geq 1}\sum_{\ell =  1}^L \int_{0}^{t\wedge\theta^{2b}_{\vec{\zeta},\kappa}}\la f_p,\Xi^{{\vec{\zeta}}, j}_{s}\ra \int_{\tilde{\mathbb{X}}^{(\ell)}}     \frac{1}{\vec{\alpha}^{(\ell)}!}  K_\ell\left(\vec{x}\right) \left(\int_{\mathbb{Y}^{(\ell)}}    \left( \sum_{r = 1}^{\beta_{\ell j}} f_p(y_r^{(j)}) \right) \sqrt{\gamma}\paren{m^\eta_\ell\left(\vec{y} \, |\, \vec{x} \right)-m_\ell\left(\vec{y} \, |\, \vec{x} \right)}\, d \vec{y}  \right)\,\\
  \times\lambda^{(\ell)}[\mu_{s}^{\vec{\zeta}}](d\vec{x})\, ds\nonumber
\end{multlined}\\
&\leq C\left( \E \int_{0}^{T\wedge\theta^{2b}_{\vec{\zeta},\kappa}} ||\Xi^{{\vec{\zeta}}, j}_{s}||^{2}_{-\Gamma_{1},a}ds+(\sqrt{\gamma}\eta)^{2}\right)
\end{aligned}
\end{equation*}
The fifth line in \cref{eq:fp_squared} is a martingale by \cref{lem:MartingalePoisson} and thus by applying Jensen's inequality, Doob's martingale inequality, and Young's inequality, similarly to the derivation of \cref{eq:JensenAndDoob}, it is bounded by
\begin{equation*}
  \begin{aligned}
  &C_1 \sum_{p \geq 1} \sqrt{\mathbb{E}\paren{\frac{2}{\gamma}C T ||f_p||_{D,b}^4  + C ||f_p||_{D,b}^2  \int_0^{T\wedge\theta^{2b}_{\vec{\zeta},\kappa}} \paren{ \la f_p,\Xi^{{\vec{\zeta}}, j}_{s-}\ra}^2\, ds}}\\
  &\leq C \sum_{p \geq 1} \paren{\sqrt{\frac{T}{\gamma}} ||f_p||_{D,b}^2  + \sqrt{||f_p||_{D,b}^2 \E \int_0^{T\wedge\theta^{2b}_{\vec{\zeta},\kappa}} \paren{ \la f_p,\Xi^{{\vec{\zeta}}, j}_{s-}\ra}^2\, ds}}\\
  &\leq C \paren{\sum_{p\geq 1} \sqrt{\frac{T}{\gamma}} ||f_p||_{D,b}^2  + \sum_{p\geq 1} ||f_p||_{D,b}^2 + \E\int_0^{T\wedge\theta^{2b}_{\vec{\zeta},\kappa}} ||\Xi^{{\vec{\zeta}}, j}_{s-}||_{-\Gamma_{1},a}^2\, ds}.
  \end{aligned}
\end{equation*}

The last term in \cref{eq:fp_squared} by \cref{lem:poissonBd1} is bounded by
$$CT\sum_{p\geq1} ||f_p||_{D,b}^2 .$$

Therefore, for some finite constant $C$ that depends on $\kappa$, $C(K)$, $C_{\circ}$, $T$, $\|\rho\|_{L^{1}}$, $b$, $D$, $\int_{\mathbb{R}^{d}}|x|^{2a}\rho(x)dx$ and on the upper bound of Lemma~\ref{L:MomentBound} for the moments, we find that
\begin{multline}\label{eq:fp_squaredBd}
\E \sup_{t\in[0, T]}\sum_{p\geq1}\la f_p,\Xi^{{\vec{\zeta}}, j}_{t\wedge\theta^{2b}_{\vec{\zeta},\kappa}}\ra^2
\leq  \E ||\Xi^{{\vec{\zeta}}, j}_{0}||_{-\Gamma_{1},a}^2  + C + C  \E \int_0^{T\wedge\theta^{2b}_{\vec{\zeta},\kappa}} \sum_{k=1}^{J}  ||\Xi^{{\vec{\zeta}}, k}_{s-}||_{-\Gamma_{1},a}^2 \, ds \\
+C\sum_{p\geq1} ||f_p||_{1+D,b}^2+C (\sqrt{\gamma}\eta)^{2}
\end{multline}
As $\Gamma_{1}>1+d/2+D$ and $b-a>d/2$, the embedding $W^{\Gamma_{1},a}_{0}\hookrightarrow W^{1+D,b}_{0}$ is of Hilbert-Schmidt type, so $\sum_{p\geq 1}||f_p||^{2}_{1+D,b}<\infty$. Summing over $j = 1, \cdots, J$, using Parseval's identity, and using Assumption~\ref{A:AssumptionParameterRelativeRates} to make $\sqrt{\gamma} \eta$ sufficiently small, we find
\begin{equation}\label{eq:negativeNorm1}
\begin{aligned}
\sum_{j=1}^J\E{\sup_{t\in[0, T]}||\Xi^{{\vec{\zeta}}, j}_{t\wedge\theta^{2b}_{\vec{\zeta},\kappa}}||^2_{-\Gamma_{1},a}}
&\leq C_1+ \sum_{j=1}^J\E ||\Xi^{{\vec{\zeta}}, j}_{0}||_{-\Gamma_{1},a}^2 + C_2\int_0^{T\wedge\theta^{2b}_{\vec{\zeta},\kappa}} \sum_{k = 1}^J \E{\sup_{s\in[0, t]} ||\Xi^{{\vec{\zeta}}, k}_{s-}||^2_{-\Gamma_{1},a}}dt \\
&=  C_1 + \sum_{j=1}^J\E ||\Xi^{{\vec{\zeta}}, j}_{0}||_{-\Gamma_{1},a}^2+ C_2\int_0^{T} \sum_{k = 1}^J \E{\sup_{s\in[0, t]} ||\Xi^{{\vec{\zeta}}, k}_{s-}||^2_{-\Gamma_{1},a}}\chi_{\{\theta^{2b}_{\vec{\zeta},\kappa}\geq t\}}dt\\
&\leq  C_1 + \sum_{j=1}^J\E ||\Xi^{{\vec{\zeta}}, j}_{0}||_{-\Gamma_{1},a}^2+C_2\int_0^{T} \sum_{k = 1}^J \E{\sup_{s\in[0, t]} ||\Xi^{{\vec{\zeta}}, k}_{s\wedge \theta^{2b}_{\vec{\zeta},\kappa} -}||^2_{-\Gamma_{1},a}}dt
\end{aligned}
\end{equation}
for some generic constant  $C_1, C_2$ independent of $\vec{\zeta}$. Applying Gronwall's inequality to \cref{eq:negativeNorm1}, we obtain the desired result \cref{eq:negativeNormBd} since by Assumption \ref{A:AssumptionInitialCondition}  $\E ||\Xi^{{\vec{\zeta}}, j}_{0}||_{-\Gamma_{1},a}^2$ is uniformly bounded.

\end{proof}

\subsubsection{Uniform Bound on the martingales $\left\{\paren{M^{\vec{\zeta}, 1}_{t\wedge\theta^{2b}_{\vec{\zeta},\kappa}}, \cdots,  M^{\vec{\zeta}, J}_{t\wedge\theta^{2b}_{\vec{\zeta},\kappa}}}\right\}_{t\in[0, T]}$}
\begin{lemma}\label{lem:MartingaleNegBdd}
There exists a finite constant $C<\infty$ such that
\begin{equation}\label{eq:negativeNormBdMartingale}
\sup_{\vec{\zeta} \in(0,1)^2}\E [\sup_{t\in[0, T]}|| M^{\vec{\zeta}, j}_{t\wedge\theta^{2b}_{\vec{\zeta},\kappa}}||^2_{-\Gamma_{1},a}] < C,
\end{equation}
for $j = 1, \cdots, J$.
\end{lemma}

\begin{proof}
Notice that the quadratic variation of $M^{\vec{\zeta}, j}_{t\wedge\theta^{2b}_{\vec{\zeta},\kappa}}(f) $ is
\begin{align*}
&\la M^{\vec{\zeta}, j}\ra_{t\wedge\theta^{2b}_{\vec{\zeta},\kappa}}(f) = \int_{0}^{t\wedge\theta^{2b}_{\vec{\zeta},\kappa}} \la 2D_{j} \paren{\frac{ \partial f}{\partial Q}(x)}^2, \mu_{s-}^{{\vec{\zeta}}, j}(dx)\ra\,ds\\
&
\quad+\sum_{\ell =  1}^L \int_{0}^{t\wedge\theta^{2b}_{\vec{\zeta},\kappa}} \int_{\tilde{\mathbb{X}}^{(\ell)}}     \frac{1}{\vec{\alpha}^{(\ell)}!}   K_\ell\left(\vec{x}\right) \left(\int_{\mathbb{Y}^{(\ell)}}    \left( \sum_{r = 1}^{\beta_{\ell j}} f(y_r^{(j)}) - \sum_{r = 1}^{\alpha_{\ell j}} f(x_r^{(j)}) \right)^2 m^\eta_\ell\left(\vec{y} \, |\, \vec{x} \right)\, d \vec{y} \right)\,\lambda^{(\ell)}[\mu_{s-}^{{\vec{\zeta}}}](d\vec{x}) \, ds \\
&\qquad \leq C  ||f||^2_{C^{1,b}} + C ||f||^2_{C^{0,b}}\\
& \qquad \leq C ||f||^2_{1+D,b}
\end{align*}
with $D>d/2$, for some $C_1$ depending on $D_j's, C_{\circ}, C(K), T, L,\kappa$ but independent of $\vec{\zeta}$. Above the last inequality is due to the Sobolev embedding theorem since $m>d/2$. Doob's martingale inequality gives
\begin{equation*}
\E \brac{\sup_{t\in[0, T]}\paren{M^{\vec{\zeta}, j}_{t\wedge\theta^{2b}_{\vec{\zeta},\kappa}}(f)}^2} \leq 2 \E \brac{\paren{M^{\vec{\zeta}, j}_{T\wedge\theta^{2b}_{\vec{\zeta},\kappa}}(f)}^2} \leq C ||f||_{1+D,b}^2,
\end{equation*}
for $j= 1, \cdots, J$. Now let $\{ f_p \}_{p\geq 1}$ be a complete orthonormal system in $W_0^{\Gamma_{1},a}(\R^d)$ of class $C^\infty_0(\R^d)$ for $\Gamma_{1}\geq 2D$ and $b-a>d/2$.  Then, the embedding $W^{\Gamma_{1},a}_{0}\hookrightarrow W^{1+D,b}_{0}$ is of Hilbert-Schmidt type, so $\sum_{p\geq 1}||f_p||^{2}_{1+D,b}<\infty$. Thus, by Parseval's identity, we obtain
\begin{align*}
\sup_{\vec{\zeta} > 0}\E [\sup_{t\in[0, T]}|| M^{\vec{\zeta}, j}_{t\wedge\theta^{2b}_{\vec{\zeta},\kappa}}||^2_{-\Gamma_{1},a}] & \leq C\sum_{p\geq 1} ||f||_{1+D,b}^2  < \infty,
\end{align*}
concluding the proof of the lemma.
\end{proof}

\subsubsection{Control of the increments}

Since we showed \cref{lem:MartingaleNegBdd,lem:negativeNormBd}, to show tightness (see~\cite{Meleard} pg. 46), we only need to verify the following Aldous conditions.

\begin{lemma}[Aldous condition for martingale]\label{lem:AldousMartingale}
For every $\epsilon_1, \epsilon_2 > 0$, there exists $\delta > 0$ and an integer $n_0$ such that for every $F_t^n$-stopping time $\tau_n < T$, $j = 1, \cdots, J$,
\begin{equation}
\sup_{n\geq n_0} \sup_{\sigma \leq \delta} \mathbb{P} (||M^{\vec{\zeta}, j}_{(\tau_n+ \sigma)\wedge\theta^{2b}_{\vec{\zeta},\kappa}} - M^{\vec{\zeta}, j}_{\tau_n \wedge\theta^{2b}_{\vec{\zeta},\kappa} }||_{-\Gamma_{1},a} \geq \epsilon_1) \leq \epsilon_2.
\end{equation}
\end{lemma}
\begin{proof}

Let $\{ f_p \}_{p\geq 1}$ be a complete orthonormal system in $W_0^{\Gamma_{1},a}(\R^d)$ of class $C^\infty_0(\R^d)$. Then
\begin{align}
&\E \sum_{p\geq1}\la f_p, M^{\vec{\zeta}, j}_{(\tau_n+ \sigma)\wedge\theta^{2b}_{\vec{\zeta},\kappa}} - M^{\vec{\zeta}, j}_{\tau_n\wedge\theta^{2b}_{\vec{\zeta},\kappa}} \ra^2
 = \E\sum_{p\geq1} \la M\ra^{\vec{\zeta},j}_{(\tau_n+ \sigma)\wedge\theta^{2b}_{\vec{\zeta},\kappa}}(f_p) - \la M\ra^{\vec{\zeta}, j}_{\tau_n\wedge\theta^{2b}_{\vec{\zeta},\kappa} }(f_p) \nonumber\\
&
= \E\int_{\tau_n \wedge\theta^{2b}_{\vec{\zeta},\kappa}}^{(\tau_n + \sigma)\wedge\theta^{2b}_{\vec{\zeta},\kappa}} \la 2D_{j} \paren{\frac{ \partial f_{p}}{\partial Q}(x)}^2, \mu_{s-}^{{\vec{\zeta}}, j}(dx)\ra\,ds +\E\sum_{\ell =  1}^L \int_{\tau_n\wedge\theta^{2b}_{\vec{\zeta},\kappa}}^{(\tau_n + \sigma)\wedge\theta^{2b}_{\vec{\zeta},\kappa}} \int_{\tilde{\mathbb{X}}^{(\ell)}}     \frac{1}{\vec{\alpha}^{(\ell)}!}   K_\ell\left(\vec{x}\right)\nonumber\\
&
\qquad \times\left(\int_{\mathbb{Y}^{(\ell)}}    \left( \sum_{r = 1}^{\beta_{\ell j}} f_{p}(y_r^{(j)}) - \sum_{r = 1}^{\alpha_{\ell j}} f_{p}(x_r^{(j)}) \right)^2 m^\eta_\ell\left(\vec{y} \, |\, \vec{x} \right)\, d \vec{y} \right)\,\lambda^{(\ell)}[\mu_{s-}^{{\vec{\zeta}}}](d\vec{x}) \, ds\nonumber\\
&
\quad \leq C \sum_{p\geq1}||f_{p}||^2_{C^{1,b}} \times\sigma \nonumber\\
&
\quad \leq C \sum_{p\geq1}||f_{p}||^2_{1+D,b} \times\sigma \nonumber
\end{align}
with $D>d/2$, for some $C$ depending on the $D_j$'s, $C_\circ$, $C(K)$, $T$, $L$, and $\kappa$, but independent of $\vec{\zeta}$. Here the last inequality is due to Sobolev embedding theorem since $D>d/2$.

Since $\{ f_p \}_{p\geq 1}$ is a complete orthonormal system in $W_0^{\Gamma_{1},a}(\R^d)$ of class $C^\infty_0(\R^d)$ for $\Gamma_{1}\geq 2D$ and $b-a>d/2$, the embedding $W^{\Gamma_{1},a}_{0}\hookrightarrow W^{1+D,b}_{0}$ is of Hilbert-Schmidt type, so $\sum_{p\geq 1}||f_p||^{2}_{1+D,b}<\infty$. Thus, by Parseval's identity, we obtain
\begin{equation}\label{eq:negativeNormDiff2}
\E{||M^{\vec{\zeta}, j}_{(\tau_n+ \sigma)\wedge\theta^{2b}_{\vec{\zeta},\kappa}} - M^{\vec{\zeta},j}_{\tau_n \wedge\theta^{2b}_{\vec{\zeta},\kappa} }||^2_{-\Gamma_{1},a}} \leq C\sigma
 \end{equation}
where $C$ is a constant depending on $C(K)$, $C_{\circ}$, and the $D_j$'s. By the Markov inequality, we have
\begin{align}
&\sup_{n\geq n_0} \sup_{\sigma \leq \delta} \mathbb{P} (||M^{\vec{\zeta}, j}_{(\tau_n+ \sigma)\wedge\theta^{2b}_{\vec{\zeta},\kappa}} - M^{\vec{\zeta}, j}_{\tau_n\wedge\theta^{2b}_{\vec{\zeta},\kappa} }||_{-\Gamma_{1},a} \geq \epsilon_1) \nonumber\\
& \leq \sup_{n\geq n_0} \sup_{\sigma \leq \delta} \frac{1}{\epsilon_1^2} \E ||M^{\vec{\zeta}, j}_{(\tau_n+ \sigma)\wedge\theta^{2b}_{\vec{\zeta},\kappa}} - M^{\vec{\zeta}, j}_{\tau_n \wedge\theta^{2b}_{\vec{\zeta},\kappa} }||_{-\Gamma_{1},a}^2\nonumber\\
& \leq \sup_{n\geq n_0} \sup_{\sigma \leq \delta} \frac{1}{\epsilon_1^2}C\delta \leq \epsilon_2,\nonumber
\end{align}
$j = 1, \cdots, J$, for $\delta$ sufficiently small.
\end{proof}

\begin{lemma}[Aldous condition]\label{lem:Aldous}
For every $\epsilon_1, \epsilon_2 > 0$, there exists $\delta > 0$ and an integer $n_0$ such that for every $F_t^n$-stopping time $\tau_n < T$, $j = 1, \cdots, J$,
\begin{equation}
\sup_{n\geq n_0} \sup_{\sigma \leq \delta} \mathbb{P} (||\Xi^{\vec{\zeta}, j}_{(\tau_n+ \sigma)\wedge\theta^{2b}_{\vec{\zeta},\kappa}} - \Xi^{\vec{\zeta}, j}_{\tau_n \wedge\theta^{2b}_{\vec{\zeta},\kappa}}||_{-\Gamma_{1},a} \geq \epsilon_1) \leq \epsilon_2.
\end{equation}
\end{lemma}
\begin{proof}
By the Markov inequality, we have
\begin{align*}
&\sup_{n\geq n_0} \sup_{\sigma \leq \delta} \mathbb{P} (||\Xi^{\vec{\zeta}, j}_{(\tau_n+ \sigma)\wedge\theta^{2b}_{\vec{\zeta},\kappa}} - \Xi^{\vec{\zeta}, j}_{\tau_n \wedge\theta^{2b}_{\vec{\zeta},\kappa}}||_{-\Gamma_{1},a} \geq \epsilon_1)\\
& \leq \sup_{n\geq n_0} \sup_{\sigma \leq \delta} \frac{1}{\epsilon_1^2} \E ||\Xi^{\vec{\zeta}, j}_{(\tau_n+ \sigma)\wedge\theta^{2b}_{\vec{\zeta},\kappa}} - \Xi^{\vec{\zeta}, j}_{\tau_n \wedge\theta^{2b}_{\vec{\zeta},\kappa}}||_{-\Gamma_{1},a}^2\\
& \leq \sup_{n\geq n_0} \sup_{\sigma \leq \delta} \frac{1}{\epsilon_1^2}\paren{ C_1\sigma + C_2\sqrt{\sigma}} \leq \frac{C_1\delta + C_2\sqrt{\delta}}{\epsilon_1^2} \leq \epsilon_2
\end{align*}
as long as we choose $\delta$ sufficiently small, using the following argument to derive the bound on $\E ||\Xi^{\vec{\zeta}, j}_{(\tau_n+ \sigma)\wedge\theta^{2b}_{\vec{\zeta},\kappa}} - \Xi^{\vec{\zeta}, j}_{\tau_n \wedge\theta^{2b}_{\vec{\zeta},\kappa}}||_{-\Gamma_{1},a}^2$.

Let $\{f_p\}_{p \geq 1}$ denote a complete orthonormal system in $W_0^{\Gamma_1,a}(\R^d)$ of class $C_0^{\infty}(\R^d)$. By definition (let us ignore the stopping time $\theta^{2b}_{\vec{\zeta},\kappa}$ for the moment),
\begin{align*}
&\la f_p,  \Xi^{\vec{\zeta}, j}_{\tau_n+ \sigma} - \Xi^{\vec{\zeta}, j}_{\tau_n }\ra
=\int_{\tau_n}^{\tau_n+\sigma} \la (\mathcal{L}_j f_p)(x) , \Xi^{\vec{\zeta}, j}_{s-} (dx)\ra ds+ \cC^{\vec{\zeta}, j}_{\tau_n+\sigma} (f_p) - \cC^{\vec{\zeta}, j}_{\tau_n}(f_p) + \cD^{\vec{\zeta}, j}_{\tau_n+\sigma} (f_p) - \cD^{\vec{\zeta}, j}_{\tau_n}(f_p) \\
&
\quad+ \sum_{\ell = 1}^L \int_{\tau_n}^{\tau_n+\sigma} \int_{\tilde{\mathbb{X}}^{(\ell)}}     \frac{1}{\vec{\alpha}^{(\ell)}!}   K_\ell\left(\vec{x}\right) \left(\int_{\mathbb{Y}^{(\ell)}}    \left( \sum_{r = 1}^{\beta_{\ell j}} f(y_r^{(j)}) \right) m_\ell\left(\vec{y} \, |\, \vec{x} \right)\, d \vec{y} - \sum_{r = 1}^{\alpha_{\ell j}} f(x_r^{(j)}) \right)\\
&
\qquad\times \sqrt{\gamma} \paren{\lambda^{(\ell)}[\mu_{s}^{\vec{\zeta}}](d\vec{x})-\lambda^{(\ell)}[\bar{\mu}_{s}](d\vec{x})}\, ds \\
&
\quad+ \sum_{\ell = 1}^L\int_{\tau_n}^{\tau_n+\sigma} \int_{\tilde{\mathbb{X}}^{(\ell)}}     \frac{1}{\vec{\alpha}^{(\ell)}!}   K_\ell\left(\vec{x}\right) \left(\int_{\mathbb{Y}^{(\ell)}}    \left( \sum_{r = 1}^{\beta_{\ell j}} f(y_r^{(j)}) \right) \sqrt{\gamma}\paren{m^\eta_\ell\left(\vec{y} \, |\, \vec{x} \right)-m_\ell\left(\vec{y} \, |\, \vec{x} \right)}\, d \vec{y}  \right)\, \lambda^{(\ell)}[\mu_{s}^{\vec{\zeta}}](d\vec{x})\, ds.
\end{align*}
Applying Ito's formula, we obtain
\begin{align}\label{eq:fp_squaredST}
&\la f_p, \Xi^{\vec{\zeta}, j}_{\tau_n+ \sigma} - \Xi^{\vec{\zeta}, j}_{\tau_n }\ra^2
=2\int_{\tau_n}^{\tau_n+\sigma} \la f_p, \Xi^{\vec{\zeta}, j}_{s-}\ra\la (\mathcal{L}_j f_p)(x) , \Xi^{\vec{\zeta}, j}_{s-} (dx)\ra \, ds \nonumber\\
 &
\quad + 2\int_{\tau_n}^{\tau_n+\sigma} \la f_p, \Xi^{\vec{\zeta}, j}_{s-}\ra d\cC^{\vec{\zeta}, j}_s (f_p) + \la \cC^{\vec{\zeta}, j}\ra_{\tau_n+\sigma} (f_p) - \la \cC^{\vec{\zeta}, j}\ra_{\tau_n} (f_p)\nonumber\\
&
\quad+ 2\sum_{\ell =  1}^L \int_{\tau_n}^{\tau_n+\sigma}  \la f_p,\Xi^{\vec{\zeta}, j}_{s}\ra \int_{\tilde{\mathbb{X}}^{(\ell)}}     \frac{1}{\vec{\alpha}^{(\ell)}!}   K_\ell\left(\vec{x}\right) \left(\int_{\mathbb{Y}^{(\ell)}}    \left( \sum_{r = 1}^{\beta_{\ell j}} f_p(y_r^{(j)}) \right) m_\ell\left(\vec{y} \, |\, \vec{x} \right)\, d \vec{y} - \sum_{r = 1}^{\alpha_{\ell j}} f_p(x_r^{(j)}) \right)\,\nonumber\\
&
\qquad\times \sqrt{\gamma} \paren{\lambda^{(\ell)}[\mu_{s}^{n}](d\vec{x})-\lambda^{(\ell)}[\bar{\mu}_{s}](d\vec{x})}\, ds \nonumber\\
&
\quad+ 2\sum_{\ell =  1}^L\int_{\tau_n}^{\tau_n+\sigma}  \la f_p,\Xi^{\vec{\zeta}, j}_{s}\ra \int_{\tilde{\mathbb{X}}^{(\ell)}}     \frac{1}{\vec{\alpha}^{(\ell)}!}   K_\ell\left(\vec{x}\right) \left(\int_{\mathbb{Y}^{(\ell)}}    \left( \sum_{r = 1}^{\beta_{\ell j}} f_p(y_r^{(j)}) \right) \sqrt{\gamma}\paren{m^\eta_\ell\left(\vec{y} \, |\, \vec{x} \right)-m_\ell\left(\vec{y} \, |\, \vec{x} \right)}\, d \vec{y}  \right)\, \lambda^{(\ell)}[\mu_{s}^{n}](d\vec{x})\, ds\nonumber\\
&
\quad + \sum_{\ell = 1}^L \int_{\tau_n}^{\tau_n+\sigma}  \int_{\mathbb{I}^{(\ell)}} \int_{\mathbb{Y}^{(\ell)}}\int_{\mathbb{R}_{+}^2} \paren{ \paren{\la f_p,\Xi^{\vec{\zeta}, j}_{s-}\ra +  g^{\ell, j, f_p, \mu^{\vec{\zeta}}}(s,\vec{i}, \vec{y},\sigma_1, \sigma_2) }^2 - \la f_p,\Xi^{\vec{\zeta}, j}_{s-}\ra^2} \,d\tilde{N}_{\ell}(s,\vec{i}, \vec{y},\sigma_1, \sigma_2)  \nonumber\\
&
\quad + \sum_{\ell = 1}^L \int_{\tau_n}^{\tau_n+\sigma} \int_{\mathbb{I}^{(\ell)}} \int_{\mathbb{Y}^{(\ell)}}\int_{\mathbb{R}_{+}^2}\paren{ \paren{\la f_p,\Xi^{\vec{\zeta}, j}_{s-}\ra + g^{\ell, j, f_p, \mu^{\vec{\zeta}}}(s,\vec{i}, \vec{y},\sigma_1, \sigma_2) }^2 - \la f_p,\Xi^{\vec{\zeta}, j}_{s-}\ra^2 \right. \nonumber\\
&
\qquad\qquad\left.  - 2g^{\ell, j, f_p, \mu^{\vec{\zeta}}}(s,\vec{i}, \vec{y},\sigma_1, \sigma_2) \times \la f_p,\Xi^{\vec{\zeta}, j}_{s-}\ra} \,  d\bar{N}_{\ell}(s,\vec{i}, \vec{y},\sigma_1, \sigma_2).  \nonumber\\
\end{align}

Similar to the derivation of \cref{eq:fp_squaredBd} from \cref{eq:fp_squared}, we will sum over all $p\geq 1$ in \cref{eq:fp_squaredST}, and take expectations. Let us also reintroduce the stopping time, $\theta^{2b}_{\vec{\zeta},\kappa}$. Then the first line of \cref{eq:fp_squaredST}'s righthand side becomes
\begin{equation*}
  2 \E \sum_{p\geq1} \int_{\tau_n\wedge\theta^{2b}_{\vec{\zeta},\kappa}}^{(\tau_n+\sigma)\wedge\theta^{2b}_{\vec{\zeta},\kappa}} \la f_p, \Xi^{\vec{\zeta}, j}_{s-}\ra\la (\mathcal{L}_j f_p), \Xi^{\vec{\zeta}, j}_{s-}\ra \, ds = 2 \E \int_{\tau_n\wedge\theta^{2b}_{\vec{\zeta},\kappa}}^{(\tau_n+\sigma)\wedge\theta^{2b}_{\vec{\zeta},\kappa}}  \la \Xi^{\vec{\zeta}, j}_{s-},\mathcal{L}_j^* \Xi^{\vec{\zeta}, j}_{s-}\ra \, ds.
\end{equation*}
In the second line of \cref{eq:fp_squaredST}, by \cref{lem:contMartingale}, the stochastic integral $2\int_{\tau_n\wedge\theta^{2b}_{\vec{\zeta},\kappa}}^{(\tau_n+\sigma)\wedge\theta^{2b}_{\vec{\zeta},\kappa}} \la f_p, \Xi^{\vec{\zeta}, j}_{s-}\ra d\cC^{\vec{\zeta}, j}_s (f_p)$ is a martingale and therefore its expectation is zero. By \cref{eq:contMartingaleQV}, the term $\la \cC^{\vec{\zeta}, j}\ra_{(\tau_n+\sigma)\wedge\theta^{2b}_{\vec{\zeta},\kappa}} (f_p) - \la \cC^{\vec{\zeta}, j}\ra_{\tau_n\wedge\theta^{2b}_{\vec{\zeta},\kappa}} (f_p)$ from the second line of \cref{eq:fp_squaredST} is bounded by (using the Sobolev embedding theorem with $m>d/2$)
$$\E \sum_{p\geq1} \paren{ \la \cC^{\vec{\zeta}, j}\ra_{(\tau_n+\sigma)\wedge\theta^{2b}_{\vec{\zeta},\kappa}} (f_p) - \la \cC^{\vec{\zeta}, j}\ra_{\tau_n\wedge\theta^{2b}_{\vec{\zeta},\kappa}} (f_p) }\leq C \sigma  \sum_{p\geq1}||f_p||^2_{1+D,b}.$$
The third line of \cref{eq:fp_squaredST}, according to \cref{lem:DiffMeas}, is bounded by
\begin{align*}
  &2\E \sum_{p\geq 1} \sum_{\ell =  1}^L \int_{\tau_n\wedge\theta^{2b}_{\vec{\zeta},\kappa}}^{(\tau_n+\sigma)\wedge\theta^{2b}_{\vec{\zeta},\kappa}} \la f_p,\Xi^{\vec{\zeta}, j}_{s}\ra \int_{\tilde{\mathbb{X}}^{(\ell)}}     \frac{1}{\vec{\alpha}^{(\ell)}!}   K_\ell\left(\vec{x}\right) \left(\int_{\mathbb{Y}^{(\ell)}}    \left( \sum_{r = 1}^{\beta_{\ell j}} f_p(y_r^{(j)}) \right) m_\ell\left(\vec{y} \, |\, \vec{x} \right)\, d \vec{y} - \sum_{r = 1}^{\alpha_{\ell j}} f_p(x_r^{(j)}) \right)\\
  &\qquad\times \sqrt{\gamma} \paren{\lambda^{(\ell)}[\mu_{s}^{n}](d\vec{x})-\lambda^{(\ell)}[\bar{\mu}_{s}](d\vec{x})}\, ds \nonumber\\
  & \leq  C \E \int_{\tau_n\wedge\theta^{2b}_{\vec{\zeta},\kappa}}^{(\tau_n+\sigma)\wedge\theta^{2b}_{\vec{\zeta},\kappa}} \sum_{k=1}^J ||\Xi^{\vec{\zeta},k}_{s}||^{2}_{-\Gamma_{1},a} \, ds.
\end{align*}
The fourth line of \cref{eq:fp_squaredST}, by \cref{lem:DiffPlaceMeas}, is bounded by
\begin{align}
& 2\E \sum_{p\geq 1}\sum_{\ell =  1}^L \int_{\tau_n\wedge\theta^{2b}_{\vec{\zeta},\kappa}}^{(\tau_n+\sigma)\wedge\theta^{2b}_{\vec{\zeta},\kappa}} \la f_p,\Xi^{\vec{\zeta}, j}_{s}\ra \int_{\tilde{\mathbb{X}}^{(\ell)}}     \frac{1}{\vec{\alpha}^{(\ell)}!}   K_\ell\left(\vec{x}\right) \left(\int_{\mathbb{Y}^{(\ell)}}    \left( \sum_{r = 1}^{\beta_{\ell j}} f_p(y_r^{(j)}) \right) \sqrt{\gamma}\paren{m^\eta_\ell\left(\vec{y} \, |\, \vec{x} \right)-m_\ell\left(\vec{y} \, |\, \vec{x} \right)}\, d \vec{y}  \right)\, \lambda^{(\ell)}[\mu_{s}^{\vec{\zeta}}](d\vec{x})\, ds\nonumber\\
&\leq C\left( (\sqrt{\gamma}\eta)^{2}\sigma+ \E\int_{\tau_n\wedge\theta^{2b}_{\vec{\zeta},\kappa}}^{(\tau_n+\sigma)\wedge\theta^{2b}_{\vec{\zeta},\kappa}}  ||\Xi^{\vec{\zeta}, j}_{s}||^{2}_{-\Gamma_{1},a}ds\right).
\end{align}
The second to last line in \cref{eq:fp_squaredST} is a martingale by \cref{lem:MartingalePoisson} and its expectation is zero. The last term in \cref{eq:fp_squaredST} by \cref{lem:poissonBd1} is bounded by
\begin{equation*}
  C \sum_{p\geq1} ||f_p||_{D,b}^2 \sigma.
\end{equation*}
Therefore, we obtain
\begin{multline*}
  \E \sum_{p\geq1}\la f_p, \Xi^{\vec{\zeta}, j}_{(\tau_n+ \sigma)\wedge\theta^{2b}_{\vec{\zeta},\kappa}} - \Xi^{\vec{\zeta}, j}_{\tau_n \wedge\theta^{2b}_{\vec{\zeta},\kappa}} \ra^2
  \leq   2\E \int_{\tau_n\wedge\theta^{2b}_{\vec{\zeta},\kappa}}^{(\tau_n+\sigma)\wedge\theta^{2b}_{\vec{\zeta},\kappa}}  \la \Xi^{\vec{\zeta}, j}_{s-} , \mathcal{L}_j^* \Xi^{\vec{\zeta}, j}_{s-} \ra \, ds + C \sigma \sum_{p\geq1}||f_p||^2_{1+D,b}\\
  + C(\sqrt{n}\eta)^{2}\sigma+ C \E \int_{\tau_n\wedge\theta^{2b}_{\vec{\zeta},\kappa}}^{(\tau_n+\sigma)\wedge\theta^{2b}_{\vec{\zeta},\kappa}} \sum_{k=1}^J ||\Xi^{\vec{\zeta},k}_{s}||^{2}_{-\Gamma_{1},a} \, ds.
\end{multline*}

Let $\Gamma_{1}\geq 2D$ and $b-a>d/2$ so that the embedding $W^{\Gamma_{1},a}_{0}\hookrightarrow W^{1+D,b}_{0}$ is of Hilbert-Schmidt type, and hence $\sum_{p\geq 1}||f_p||^{2}_{1+D,b}<\infty$. By Parseval's identity, \cref{lem:negativeNormBd}, and \cref{lem:densityIneq}, we obtain
\begin{equation}\label{eq:negativeNormDiff}
\E{||\Xi^{\vec{\zeta}, j}_{(\tau_n+ \sigma)\wedge\theta^{2b}_{\vec{\zeta},\kappa}} - \Xi^{\vec{\zeta}, j}_{\tau_n \wedge\theta^{2b}_{\vec{\zeta},\kappa}}||^2_{-\Gamma_{1},a}} \leq C\sigma
 \end{equation}
where $C$ is a constant which only depends on $\sqrt{\gamma}\eta$,  $C(K)$, $C_{\circ}$, $\kappa$, and $\sup_{n> n_0}\E[ \sup_{t\in [0, T]}||\Xi^{\vec{\zeta}, j}_{t} ||^2_{-\Gamma_{1},a}] < \infty$.
\end{proof}

\begin{theorem}\label{thm:tightness}
The martingale process $\left\{\paren{M^{\vec{\zeta}, 1}_t, \cdots,  M^{\vec{\zeta}, J}_t},t\in[0, T]\right\}_{\vec{\zeta}\in(0,1)^{2}}$ and fluctuation process $\left\{\paren{\Xi^{\vec{\zeta}, 1}_t, \cdots, \Xi^{\vec{\zeta}, J}_t},t\in[0, T]\right\}_{\vec{\zeta}\in(0,1)^{2}}$,  are tight in the space of $D_{(W^{-\Gamma_{1}-1,a}(\R^d))^{\otimes J}}([0, T])$.
\end{theorem}
\begin{proof}
By Lemma \ref{lem:negativeNormBd} and the fact that the set $\{\phi\in W^{-\Gamma_{1}-1,a}(\R^d): ||\phi||_{-\Gamma_{1},a} \leq C_\epsilon \}$ is a compact subset of $W^{-\Gamma_{1}-1,a}$  we get that the compact containment condition for $\left\{\paren{\Xi^{\vec{\zeta}, 1}_t, \cdots, \Xi^{\vec{\zeta}, J}_t},t\in[0, T]\right\}_{\vec{\zeta}\in(0,1)^{2}}$ holds. Similarly, due to \cref{lem:MartingaleNegBdd}  we obtain that the compact containment condition for $\left\{\paren{M^{\vec{\zeta}, 1}_t, \cdots,  M^{\vec{\zeta}, J}_t}, t\in[0, T]\right\}_{\vec{\zeta}\in(0,1)^{2}}$ also holds. These facts together with  Lemma \cref{lem:AldousMartingale} and Lemma \ref{lem:Aldous} show that the tightness conditions in $D_{(W^{-\Gamma_{1}-1,a}(\R^d))^{\otimes J}}([0, T])$ are satisfied.
\end{proof}

\subsection{Identification of the Limit}\label{S:Identification}

\begin{lemma}\label{lem:cont}
Any limiting process of $\left\{\paren{\Xi^{\vec{\zeta}, 1}_t, \cdots, \Xi^{\vec{\zeta}, J}_t}\right\}_{t\in[0, T]}$ is continuous in time, i.e. takes value in the space of $C([0, T], \paren{W^{-\Gamma,a}(\R^d)}^{\otimes J})$.
\end{lemma}
\begin{proof}
It suffices to show that for $j = 1, \cdots, J$,
\begin{equation}
\lim_{\vec{\zeta}\to 0} \E[\sup_{t\in [0, T]}|| \Xi^{\vec{\zeta}, j}_t - \Xi^{\vec{\zeta}, j}_{t-}||^2_{-\Gamma,a}] = 0.
\end{equation}

Now let $\{ f_p \}_{p\geq 1}$ be a complete orthonormal system in $W_0^{\Gamma,a}(\R^d)$ of class $C^\infty_0(\R^d)$. By construction, the discontinuity of $\la f_p,  \Xi^{\vec{\zeta}, j}_t\ra$ comes from the Poisson martingale part $\cD^{\vec{\zeta}, j}_{t}(f_p)$ defined in \cref{eq:disMartingaleDef}. Notice that the jump size of \cref{eq:jump} is uniformly bounded by $\mathcal{O}(\frac{1}{\sqrt{\gamma}})$. Then we can get
\begin{equation}
\la f_p,  \Xi^{\vec{\zeta}, j}_t - \Xi^{\vec{\zeta}, j}_{t-} \ra^{2}\leq \frac{C}{\sqrt{\gamma}}||f_p||^2_{D,b},
\end{equation}
for some generic finite constant $C<\infty$ and $b>a+d/2$.  Since $\Gamma\geq 2D>D+d/2$ and $b-a>d/2$, the embedding $W^{\Gamma,a}_{0}\hookrightarrow W^{D,b}_{0}$ is of Hilbert-Schmidt type, so $\sum_{p\geq 1}||f_p||^{2}_{D,b}<\infty$. Thus, by Parseval's identity,  we obtain,
\begin{equation}
\E[\sup_{t\in [0, T]}|| \Xi^{\vec{\zeta}, j}_t - \Xi^{\vec{\zeta}, j}_{t-}||^2_{-\Gamma,a}] \leq \frac{C}{\sqrt{\gamma}},
\end{equation}
and therefore
\begin{equation}
\lim_{\vec{\zeta}\to 0} \E[\sup_{t\in [0, T]}|| \Xi^{\vec{\zeta}, j}_t - \Xi^{\vec{\zeta}, j}_{t-}||^2_{-\Gamma,a}] = 0.
\end{equation}
\end{proof}

\begin{lemma}\label{lem:convMartingale}
For every $f, g\in W_0^{\Gamma,a}$, the process $\paren{M^{\vec{\zeta}, 1}_t, \cdots,  M^{\vec{\zeta}, J}_t}$  defined by \cref{eq:MartingaleDef} is a martingale in $\paren{W^{-\Gamma,a}(\R^d)}^{\otimes J}$ that converges in distribution in $D_{(W^{-\Gamma,a}(\R^d))^{\otimes J}}([0, T])$ to a  mean-zero Gaussian martingale $\paren{\bar{M}^1_t, \cdots, \bar{M}^J_t}$
with covariance structure
\begin{equation}\label{eq:limitcov}
\begin{aligned}
\Cov[\bar{M}^j_t(f), \bar{M}^k_s(g)] &=  \sum_{\ell =  1}^L \int_{0}^{s\wedge t} \int_{\tilde{\mathbb{X}}^{(\ell)}}     \frac{1}{\vec{\alpha}^{(\ell)}!}   K_\ell\left(\vec{x}\right) \left(\int_{\mathbb{Y}^{(\ell)}}    \left( \sum_{r = 1}^{\beta_{\ell j}} f(y_r^{(j)}) - \sum_{r = 1}^{\alpha_{\ell j}} f(x_r^{(j)}) \right)\right.\\
&\qquad\qquad\qquad\qquad\left. \times \left( \sum_{r = 1}^{\beta_{\ell k}} g(y_r^{(k)}) - \sum_{r = 1}^{\alpha_{\ell k}} g(x_r^{(k)}) \right) m_\ell\left(\vec{y} \, |\, \vec{x} \right)\, d \vec{y} \right)\,\lambda^{(\ell)}[\bar{\mu}_{r}](d\vec{x}) \, dr\\
&+\int_{0}^{s\wedge t} \la 2D_{j} \frac{ \partial f}{\partial Q}(x)\frac{ \partial g}{\partial Q}(x), \bar{\mu}_{r}^j(dx)\ra 1_{\{k=j\}}\,dr
\end{aligned}
\end{equation}
for $\leq s,t\leq T$.
%

\end{lemma}

\begin{proof}
Let $f_k\in W_0^{\Gamma,a}, k = 1, \cdots, J$. Consider the martingale $\sum_{k = 1}^J \sigma_k M^{\vec{\zeta}, k}_t(f_k)$, for some $\sigma_k \in\R, k = 1, \cdots, J$.  By definition of  $M^{\vec{\zeta}, k}_t(f_k)$ in \cref{eq:MartingaleDef}, we see that
\begin{align*}
\sum_{k = 1}^J \sigma_k M^{\vec{\zeta}, k}_t(f_k)
&
= \sum_{k = 1}^J \sigma_k \frac{1}{{\sqrt{\gamma}}}\sum_{i\geq 1}\int_{0}^{t}1_{\{i\leq \gamma \la 1, \mu_{s-}^{{\vec{\zeta}}, k}\ra\}}\sqrt{2D_{k}}\frac{\partial f_k}{\partial Q}(H^i(\gamma\mu_{s-}^{{\vec{\zeta}}, j}))dW^{i, k}_{s}\\
&
\quad+  \sum_{k = 1}^J \sigma_k \sqrt{\gamma}\sum_{\ell = 1 }^L \int_{0}^{t}\int_{\mathbb{I}^{(\ell)}} \int_{\mathbb{Y}^{(\ell)}}\int_{\mathbb{R}_{+}^2}\left(\la \sigma_k f_k, \mu^{{\vec{\zeta}}, k}_{s-}-\frac{1}{\gamma} \sum_{r = 1}^{\alpha_{\ell k}} \delta_{H^{i_r^{(k)}}(\gamma\mu^{{\vec{\zeta}},k}_{s-})} + \frac{1}{\gamma}\sum_{r = 1}^{\beta_{\ell k}} \delta_{y_r^{(k)}} \ra  - \la \sigma_k f_k,\mu^{{\vec{\zeta}}, k}_{s-}\ra\right)\\
&\qquad\times 1_{\{\vec{i} \in \Omega^{(\ell)}(\gamma\mu^{\vec{\zeta}}_{s-})\}}   \times 1_{ \{ \theta_1 \leq K_\ell^{\gamma}\left(\mathcal{P}^{(\ell)}(\gamma\mu_{s-}^{\vec{\zeta}}, \vec{i}) \right) \}}  \times 1_{ \{ \theta_2 \leq  m^\eta_\ell\left(\vec{y} \,  | \, \mathcal{P}^{(\ell)}(\gamma\mu_{s-}^{\vec{\zeta}}, \vec{i}) \right) \}}  d\tilde{N}_{\ell}(s,\vec{i}, \vec{y},\theta_1, \theta_2).
\end{align*}
with quadratic variation
\begin{align*}
\la \sum_{k = 1}^J \sigma_k M^{\vec{\zeta}, k}(f_k)\ra_t
&= \sum_{k = 1}^J \sigma_k^2 \int_{0}^{t} \la 2D_{k} \paren{\frac{ \partial f_k}{\partial Q}(x)}^2, \mu_{s-}^{{\vec{\zeta}}, k}(dx)\ra\,ds\\
&\quad+ \sum_{\ell = 1}^L \int_{0}^{t} \int_{\tilde{\mathbb{X}}^{(\ell)}}   \left(\int_{\mathbb{Y}^{(\ell)}}    \left( \sum_{k = 1}^J\sum_{r = 1}^{\beta_{\ell k}} \sigma_k f_k(y_r^{(k)}) -\sum_{k = 1}^J \sum_{r = 1}^{\alpha_{\ell k}} \sigma_k f_k(x_r^{(k)}) \right)^2 m^\eta_\ell\left(\vec{y} \, |\, \vec{x} \right)\, d \vec{y} \right)\\
&\qquad\times   \frac{1}{\vec{\alpha}^{(\ell)}!}   K_\ell\left(\vec{x}\right)  \lambda^{(\ell)}[\mu_{s-}^{{\vec{\zeta}}}](d\vec{x}) \, ds.
\end{align*}
For two different species, $S_i$ and $S_j$, their Brownian motions are independent, and therefore there are no cross terms in the quadratic variation arising from the integrals with respect to Brownian motions that involve $S_i$ and $S_j$ (see Chapter 3 Proposition 2.17 in \cite{KS:1998}). Note, however, species can share the same Poisson "random generators" via reactions in which they both participate.

We know that $ \paren{\mu^{\vec{\zeta}, 1}_{s}, \cdots, \mu^{\vec{\zeta}, J}_{s}}\to \paren{\bar{\mu}^1_{s}, \cdots, \bar{\mu}^J_{s}}$ in probability in $D([0, T], M_F(\R^d)^{\otimes J})$. Applying the continuous mapping theorem, we can get that as $\vec{\zeta}\to 0$, $\la \sum_{k=1}^J\sigma_k M^{\vec{\zeta}, k}(f_k)\ra_t$ converges in probability to
\begin{align}\label{eq:QVofh}
&\sum_{k = 1}^J \sigma_k^2 \int_{0}^{t} \la 2D_{k} \paren{\frac{ \partial f_k}{\partial Q}(x)}^2, \bar{\mu}_{s-}^{k}(dx)\ra\,ds\nonumber\\
&
+ \sum_{\ell = 1}^L \int_{0}^{t} \int_{\tilde{\mathbb{X}}^{(\ell)}} \frac{1}{\vec{\alpha}^{(\ell)}!}   K_\ell\left(\vec{x}\right)    \left(\int_{\mathbb{Y}^{(\ell)}}    \left( \sum_{k = 1}^J\sum_{r = 1}^{\beta_{\ell k}} \sigma_k f_k(y_r^{(k)}) -\sum_{k = 1}^J \sum_{r = 1}^{\alpha_{\ell k}} \sigma_k f_k(x_r^{(k)}) \right)^2 m^\eta_\ell\left(\vec{y} \, |\, \vec{x} \right)\, d \vec{y} \right)  \lambda^{(\ell)}[\bar{\mu}_{s-}](d\vec{x}) \, ds.\nonumber\\
&
=\sum_{k = 1}^J \sigma_k^2 \int_{0}^{t} \la 2D_{k} \paren{\frac{ \partial f_k}{\partial Q}(x)}^2, \bar{\mu}_{s-}^k(dx)\ra\,ds\nonumber\\
&
+ \sum_{k=1}^J \sigma_k^2 \sum_{\ell = 1}^L \int_{0}^{t} \int_{\tilde{\mathbb{X}}^{(\ell)}} \frac{1}{\vec{\alpha}^{(\ell)}!}   K_\ell\left(\vec{x}\right)    \left(\int_{\mathbb{Y}^{(\ell)}}    \left( \sum_{r = 1}^{\beta_{\ell k}} f_k(y_r^{(k)}) - \sum_{r = 1}^{\alpha_{\ell k}} f_k(x_r^{(k)}) \right)^2 m^\eta_\ell\left(\vec{y} \, |\, \vec{x} \right)\, d \vec{y} \right)  \lambda^{(\ell)}[\bar{\mu}_{s-}](d\vec{x}) \, ds\nonumber\\
&
+ \sum_{k=1}^J  \sum_{j \neq k} \sigma_k\sigma_j \sum_{\ell = 1}^L \int_{0}^{t} \int_{\tilde{\mathbb{X}}^{(\ell)}} \frac{1}{\vec{\alpha}^{(\ell)}!}   K_\ell\left(\vec{x}\right) \left(\int_{\mathbb{Y}^{(\ell)}}    \left( \sum_{r = 1}^{\beta_{\ell j}} f_j(y_r^{(j)}) - \sum_{r = 1}^{\alpha_{\ell j}} f_j(x_r^{(j)}) \right)\right.\nonumber\\
&
\qquad \left. \times \left( \sum_{r = 1}^{\beta_{\ell k}} f_k(y_r^{(k)}) - \sum_{r = 1}^{\alpha_{\ell k}} f_k(x_r^{(k)}) \right) m_\ell\left(\vec{y} \, |\, \vec{x} \right)\, d \vec{y} \right)\,\lambda^{(\ell)}[\bar{\mu}_{s}](d\vec{x}) \, ds.\nonumber\\
\end{align}

 As the jump size in the Poisson martingale part is uniformly bounded by $\mathcal{O}(\frac{1}{\sqrt{\gamma}})$, by a similar argument as in \cref{lem:cont}, we can show that
\begin{equation}
\lim_{\gamma\to\infty} \E \brac{\sup_{t\in [0, T]} \abs{\sum_{k=1}^J \sigma_k M^{\vec{\zeta}, k}_t(f_k) -\sum_{k = 1}^J \sigma_k M^{\vec{\zeta}, k}_{t-} (f_k)}} = 0.
\end{equation}
Therefore  $\sum_{k=1}^J\sigma_k M^{\vec{\zeta}, k}_t(f_k)$ converges in distribution to a mean-zero Gaussian martingale $\sum_{k=1}^J\sigma_k \bar{M}^k_t(f_k) $ with the quadratic variation \cref{eq:QVofh}.

Since this is true for arbitrary $\sigma_k's$, by Cramer-Wold theorem, see Theorem 29.4 in \cite{Billingsley}, the vector $\paren{  M ^{\vec{\zeta}, 1}_t, \cdots,  M ^{\vec{\zeta}, J}_t}$ converges to  $\paren{  \bar{M}^1_t, \cdots, \bar{M}^J_t}$ in $W^{-\Gamma,a}(\R^d)^{\otimes J}$
with covariance structure \cref{eq:limitcov}.

\end{proof}



\begin{theorem}
For any $f\in W_0^{\Gamma,a}$, $\left\{\paren{\Xi^{\vec{\zeta}, 1}_t, \cdots, \Xi^{\vec{\zeta}, J}_t}, t\in[0, T]\right\}_{\vec{\zeta}\in(0,1)^{2}}$ converges in law in the space $C([0, T], \paren{W^{-\Gamma,a}(\R^d)}^{\otimes J})$ to the process  $\{\paren{\bar{\Xi}^1_t, \cdots, \bar{\Xi}^J_t}\}_{t\in[0, T]}$ that must satisfy in $W^{-(2+\Gamma),a}$ the stochastic evolution equation
\begin{multline}\label{eq:limitEqThm}
\la f,\bar{\Xi}^j_t\ra =\la f,\bar{\Xi}^j_0\ra + \int_{0}^{t} \la (\mathcal{L}_jf)(x) , \bar{\Xi}^j_s (dx)\ra ds  + \bar{M}^{j}_t(f) \\
+ \sum_{\ell =  1}^L \int_{0}^{t} \int_{\tilde{\mathbb{X}}^{(\ell)}}     \frac{1}{\vec{\alpha}^{(\ell)}!}   K_\ell\left(\vec{x}\right) \left(\int_{\mathbb{Y}^{(\ell)}}    \left( \sum_{r = 1}^{\beta_{\ell j}} f(y_r^{(j)}) \right) m_\ell\left(\vec{y} \, |\, \vec{x} \right)\, d \vec{y} - \sum_{r = 1}^{\alpha_{\ell j}} f(x_r^{(j)}) \right)\,\Delta^{(\ell)}[\bar{\mu}_s, \bar{\Xi}_s](d\vec{x}) \, ds
\end{multline}
for $j = 1, \cdots, J$, where $\paren{\bar{M}^1_t, \cdots, \bar{M}^J_t}$ is a mean-zero Gaussian process defined in \cref{lem:convMartingale}.
\end{theorem}

\begin{proof}
By \cref{thm:tightness} and \cref{lem:convMartingale}, the sequence $(\Xi^{\vec{\zeta}, 1}_t, \cdots, \Xi^{\vec{\zeta}, J}_t, \mu^{\vec{\zeta}, 1}_t,\cdots, \mu^{\vec{\zeta}, J}_t, M ^{\vec{\zeta}, 1}_t, \cdots, M ^{\vec{\zeta}, J}_t)$ is relatively compact in $D([0, T], W^{-\Gamma,a}(\R^d)^{\otimes J}\times M_F(\R^d)^{\otimes J}\times W^{-\Gamma,a}(\R^d)^{\otimes J})$. Since $W^{-\Gamma,a}\hookrightarrow W^{-(2+\Gamma),a}$, it follows that it is also relatively compact in $D([0, T], W^{-(2+\Gamma),a}(\R^d)^{\otimes J}\times M_F(\R^d)^{\otimes J}\times W^{-(2+\Gamma),a}(\R^d)^{\otimes J})$. Note also, as for any $i\in\{1,\cdots, J\}$
\begin{equation*}
\|  \mathcal{L}_i f\|_{\Gamma,a}\leq C\|f\|_{2+\Gamma,a}
\end{equation*}
the integral $\int_{0}^{t} (\mathcal{L}_i)^{*} \Xi^{\vec{\zeta},i}_{s} ds$ makes sense as a Bochner integral in $W^{-(2+\Gamma),a}$.

Let us recall now the equation that $(\Xi^{\vec{\zeta}, 1}_t, \cdots, \Xi^{\vec{\zeta}, J}_t, \mu^{\vec{\zeta}, 1}_t,\cdots, \mu^{\vec{\zeta}, J}_t, M ^{\vec{\zeta}, 1}_t, \cdots, M ^{\vec{\zeta}, J}_t)$ satisfies for $f\in W^{2+\Gamma,a}_{0}$
\begin{equation}\label{eq:fluctuationProcessCopy}
\begin{aligned}
\la f,\Xi^{{\vec{\zeta}}, j}_{t}\ra
&=\la f,\Xi^{{\vec{\zeta}}, j}_{0}\ra + \int_{0}^{t} \la (\mathcal{L}_jf)(x) , \Xi^{\vec{\zeta}, j}_{s-} (dx)\ra ds  + M^{\vec{\zeta}, j}_t(f) \\
&\quad+ \sum_{\ell =  1}^L \int_{0}^{t} \int_{\tilde{\mathbb{X}}^{(\ell)}}     \frac{1}{\vec{\alpha}^{(\ell)}!}   K_\ell\left(\vec{x}\right) \left(\int_{\mathbb{Y}^{(\ell)}}    \left( \sum_{r = 1}^{\beta_{\ell j}} f(y_r^{(j)}) \right) m_\ell\left(\vec{y} \, |\, \vec{x} \right)\, d \vec{y} - \sum_{r = 1}^{\alpha_{\ell j}} f(x_r^{(j)}) \right)\,\\
&\qquad\times \sqrt{\gamma} \paren{\lambda^{(\ell)}[\mu_{s}^{\vec{\zeta}}](d\vec{x})-\lambda^{(\ell)}[\bar{\mu}_{s}](d\vec{x})}\, ds \\
&\quad+ \sum_{\ell =  1}^L \int_{0}^{t} \int_{\tilde{\mathbb{X}}^{(\ell)}}     \frac{1}{\vec{\alpha}^{(\ell)}!}   K_\ell\left(\vec{x}\right) \left(\int_{\mathbb{Y}^{(\ell)}}    \left( \sum_{r = 1}^{\beta_{\ell j}} f(y_r^{(j)}) \right) \sqrt{\gamma}\paren{m^\eta_\ell\left(\vec{y} \, |\, \vec{x} \right)-m_\ell\left(\vec{y} \, |\, \vec{x} \right)}\, d \vec{y}  \right)\, \lambda^{(\ell)}[\mu_{s}^{\vec{\zeta}}](d\vec{x})\, ds.
\end{aligned}
\end{equation}
In the second to last line, we rewrite
\begin{equation*}
\sqrt{\gamma} \paren{\lambda^{(\ell)}[\mu_{s}^{\vec{\zeta}}](d\vec{x})-\lambda^{(\ell)}[\bar{\mu}_{s}](d\vec{x})} = \begin{cases}
   \Xi^{\vec{\zeta},k}_t(x), & \text{reaction $\ell$ is: } S_k \to \dots\\
   \Xi^{\vec{\zeta},k}_t(x) \mu^{\vec{\zeta},r}(y) + \bar{\mu}^{\vec{\zeta},k}(x) \Xi^{\vec{\zeta},r}_t(y) , & \text{reaction $\ell$ is: } S_k + S_r \to \dots
\end{cases}
\end{equation*}
to look like an approximation to $\Delta^{(\ell)}\brac{\bar{\mu},\bar{\Xi}}$, see~\eqref{eq:reactantmappingdef}. Notice, these expressions depend linearly on  $\{ \mu^{\vec{\zeta}, j}_t \}$ and  $\{ \Xi^{\vec{\zeta}, j}_t \}$.

The last line of \cref{eq:fluctuationProcessCopy} may depend quadratically on $\{ \mu^{\vec{\zeta}, j}_t\}$ but can be rewritten as a sum of two terms, $I + II$. Here
\begin{equation*}
  I = \sum_{\ell =  1}^L \int_{0}^{t} \int_{\tilde{\mathbb{X}}^{(\ell)}}     \frac{1}{\vec{\alpha}^{(\ell)}!}   K_\ell\left(\vec{x}\right) \left(\int_{\mathbb{Y}^{(\ell)}}    \left( \sum_{r = 1}^{\beta_{\ell j}} f(y_r^{(j)}) \right) \sqrt{\gamma}\paren{m^\eta_\ell\left(\vec{y} \, |\, \vec{x} \right)-m_\ell\left(\vec{y} \, |\, \vec{x} \right)}\, d \vec{y}  \right)\, \lambda^{(\ell)}[\bar{\mu}_{s}](d\vec{x})\, ds
\end{equation*}
does not depend on $\{ \mu^{\vec{\zeta}, j}_t \}$ and  $\{ \Xi^{\vec{\zeta}, j}_t \}$ and goes to zero as $\gamma\to\infty$ and $\eta\to 0$ due to the constraint $\sqrt{\gamma}\eta\to 0$. The second term,
\begin{multline*}
  II = \sum_{\ell =  1}^L \int_{0}^{t} \int_{\tilde{\mathbb{X}}^{(\ell)}}     \frac{1}{\vec{\alpha}^{(\ell)}!}   K_\ell\left(\vec{x}\right) \left(\int_{\mathbb{Y}^{(\ell)}}    \left( \sum_{r = 1}^{\beta_{\ell j}} f(y_r^{(j)}) \right) \sqrt{\gamma}\paren{m^\eta_\ell\left(\vec{y} \, |\, \vec{x} \right)-m_\ell\left(\vec{y} \, |\, \vec{x} \right)}\, d \vec{y}  \right)\\
  \times \paren{\lambda^{(\ell)}[\mu_{s}^{\vec{\zeta}}](d\vec{x}) - \lambda^{(\ell)}[\bar{\mu}_{s}](d\vec{x})}\, ds,
\end{multline*}
again depends linearly on  $\{ \mu^{\vec{\zeta}, j}_t \}$ and  $\{ \Xi^{\vec{\zeta}, j}_t \}$.

Due to the linearity,  we can directly obtain that $(\bar{\Xi}^{1}_t,\cdots,  \bar{\Xi}^{J}_t, \bar{\mu}^{ 1}_t, \cdots, \bar{\mu}^{J}_t, \bar{M} ^{ 1}_t, \cdots, \bar{M} ^{J}_t)$  solves \cref{eq:limitEqThm} by \cref{lem:convMartingale} and Theorem 5.5 in \cite{KP:1996}.
\end{proof}


\subsection{Uniqueness of Limiting Solution}\label{S:Uniqueness}
\begin{theorem}\label{T:Uniqueness}
The process $\{\paren{\bar{\Xi}^1_t, \cdots, \bar{\Xi}^J_t}\}_{t\in[0, T]}$ satisfying
\begin{multline}\label{eq:limitEqUnique}
\la f,\bar{\Xi}^j_t\ra
=\la f,\bar{\Xi}^j_0\ra + \int_{0}^{t} \la (\mathcal{L}_jf)(x) , \bar{\Xi}^j_s (dx)\ra ds  + \bar{M}^{j}_t(f)\\
\quad+ \sum_{\ell =  1}^L \int_{0}^{t} \int_{\tilde{\mathbb{X}}^{(\ell)}}     \frac{1}{\vec{\alpha}^{(\ell)}!}   K_\ell\left(\vec{x}\right) \left(\int_{\mathbb{Y}^{(\ell)}}    \left( \sum_{r = 1}^{\beta_{\ell j}} f(y_r^{(j)}) \right) m_\ell\left(\vec{y} \, |\, \vec{x} \right)\, d \vec{y} - \sum_{r = 1}^{\alpha_{\ell j}} f(x_r^{(j)}) \right)\,\Delta^{(\ell)}[\bar{\mu}_s, \bar{\Xi}_s](d\vec{x}) \, ds
\end{multline}
for $j = 1, \cdots, J$ and any $f\in W_0^{2+\Gamma,a}$, is unique in $\paren{W^{-(2+\Gamma),a}(\R^d)}^{\otimes J}$.
\end{theorem}

\begin{proof}
We start by making a couple of general observations. First, since for any $i\in\{1,\cdots, J\}$
\[
\|  \mathcal{L}_i f\|_{\Gamma,a}\leq C\|f\|_{2+\Gamma,a}
\]
the integral $\int_{0}^{t} (\mathcal{L}_i)^{*} \bar{\Xi}^{j}_{s} ds$ makes sense as a Bochner integral in $W^{-(2+\Gamma),a}$. Second, it follows by Lemma \ref{lem:negativeNormBd} that there exists a constant $C$ such that
\begin{equation*}
\E [\sup_{t\in[0, T]}||\bar{\Xi}^{ j}_{t}||^2_{-\Gamma,a}] \leq C,
\end{equation*}
for $j = 1, 2, \cdots, J$. In addition, since $W^{-\Gamma,a}\hookrightarrow W^{-(2+\Gamma),a}$ we also have that
\begin{equation*}
\E [\sup_{t\in[0, T]}||\bar{\Xi}^{ j}_{t}||^2_{-(2+\Gamma),a}] \leq \E [\sup_{t\in[0, T]}||\bar{\Xi}^{ j}_{t}||^2_{-\Gamma,a}]\leq C.
\end{equation*}

Suppose now that there are two solutions $\{\paren{\bar{\Xi}^1_t, \cdots, \bar{\Xi}^J_t}\}$ and $\{\paren{\tilde{\Xi}^1_t, \cdots, \tilde{\Xi}^J_t}\}$ satisfying \cref{eq:limitEqUnique} with the same initial condition. We consider the evolution of $\Phi^j_t := \bar{\Xi}^j_t-\tilde{\Xi}^j_t$ that solves
\begin{align*}
\la f,\Phi^j_t\ra &= \int_{0}^{t} \la (\mathcal{L}_jf)(x) , \Phi^j_s (dx)\ra ds\\
&\quad+ \sum_{\ell =  1}^L \int_{0}^{t} \int_{\tilde{\mathbb{X}}^{(\ell)}}     \frac{1}{\vec{\alpha}^{(\ell)}!}   K_\ell\left(\vec{x}\right) \left(\int_{\mathbb{Y}^{(\ell)}}    \left( \sum_{r = 1}^{\beta_{\ell j}} f(y_r^{(j)}) \right) m_\ell\left(\vec{y} \, |\, \vec{x} \right)\, d \vec{y} - \sum_{r = 1}^{\alpha_{\ell j}} f(x_r^{(j)}) \right)\,\Delta^{(\ell)}[\bar{\mu}_s, \Phi_s](d\vec{x}) \, ds
\end{align*}
with $\la f,  \Phi^J_0 \ra  = 0$. By the chain rule we have
\begin{multline*}
\la f,  \Phi^j_t \ra^2
= 2\int_{0}^{t} \la f,  \Phi^j_s \ra  \la (\mathcal{L}_jf)(x) , \Phi^j_s (dx)\ra ds\\
+ 2\sum_{\ell =  1}^L \int_{0}^{t} \la f,  \Phi^j_s\ra \int_{\tilde{\mathbb{X}}^{(\ell)}}     \frac{1}{\vec{\alpha}^{(\ell)}!}   K_\ell\left(\vec{x}\right) \left(\int_{\mathbb{Y}^{(\ell)}}    \left( \sum_{r = 1}^{\beta_{\ell j}} f(y_r^{(j)}) \right) m_\ell\left(\vec{y} \, |\, \vec{x} \right)\, d \vec{y} - \sum_{r = 1}^{\alpha_{\ell j}} f(x_r^{(j)}) \right)\\
\times \Delta^{(\ell)}[\bar{\mu}_s, \Phi_s](d\vec{x}) \, ds.
\end{multline*}
Let $\{ f_p \}_{p\geq 1}$ be a complete orthonormal system in $W_0^{2+\Gamma,a}(\R^d)$ of class $C^\infty_0(\R^d)$. Then, similar to the proof of \cref{eq:negativeNormBd}, taking $f = f_p$, summing over $p\geq 0$ and $j = 1, \cdots, J$, and using the bounds of Lemmas \ref{lem:densityIneq} and \ref{lem:DiffMeas} we find that
\begin{align*}
\sum_{j = 1}^J|| \Phi^j_t ||^2_{-(2+\Gamma),a} \leq C \int_0^t\sum_{k=1}^J || \Phi^k_s ||^2_{-(2+\Gamma),a}  \, ds ,
\end{align*}
for some generic constant $C$ depending on $J$, $D_j$'s, $C_{\circ}$, $C(K)$, $\norm{K}_{C^{\Gamma}}$, $\|\rho\|_{L^{1}(\mathbb{R})^{d}}$, and $\int_{\mathbb{R}^{d}}|w|^{2a}\rho(w)dw<\infty$. By Gronwall's inequality with the initial condition $ \norm{\Phi^j_0}^2_{-(2+\Gamma),a} = 0$, we obtain that for all $t\in[0,T]$
\begin{equation*}
\sum_{j = 1}^J|| \Phi^j_t ||^2_{-(2+\Gamma),a} = 0,
\end{equation*}
and therefore, uniqueness holds.
\end{proof}


\appendix

\section{Forward Equation Fock Space Representation} \label{S:FockSpaceExpect}
Our goal in this appendix is to derive~\eqref{Eq:Expectation_Mean_Formula} given Assumption~\ref{A:AssumptionInitialCondition}. Our basic approach is to begin with the forward equation Fock space representation for the dynamics, derive the analogous equation in the Fock space representation, and then show it is equivalent under a change of notation to~\eqref{Eq:Expectation_Mean_Formula}. This approach allows us to bypass estimates that would be needed to rigorously derive~\eqref{Eq:Expectation_Mean_Formula} directly from~\eqref{Eq:EM_j_formula} in the case that the test function can be unbounded (as needed for moment estimates).

To describe the Fock space representation, we first define some alternative
notation for representing the numbers of particles of each species and their
positions. Let
\begin{equation*}
  \begin{aligned}
  \vN(t) &:= \paren{N_1(t),\dots,N_J(t)} \\
  &\phantom{:}= \paren{\la 1, \nu^1_t \ra,\dots,\la 1, \nu^J_t \ra} \\
  &\phantom{:}= \gamma \paren{\la 1, \mu^1_t \ra,\dots,\la 1, \mu^J_t \ra}
\end{aligned}
\end{equation*}
denote the vector stochastic process for the number of each species at time $t$, with $\vn = (n_1,\dots,n_J)$ a possible value for this process. We treat $\vN(t)$ as a multi-index, so that $N(t) = \sum_{j=1}^J N_j(t) = \abs{\vN(t)}$ gives the total number of particles at time $t$. $\sum_{j} n_j = \abs{\vn}$ is defined  analogously. Note, we previously used $N_{\ell}$ to denote a Poisson random measure, but within this appendix $N_j(t)$ will always denote the stochastic process for the number of type $j$ molecules at time $t$.

We denote by $Q^{N_j(t)}_i(t) = H^i(\nu_t^j) \subset \R^d$ the stochastic process for the position of the $i$th particle of species $j$ at time $t$, with
\begin{equation*}
  \vQ^{N_j(t)}(t) = \paren{Q^{N_j(t)}_1(t),\dots, Q^{N_j(t)}_{N_j(t)}(t)} \subset \R^{d N_j(t)}
\end{equation*}
the stochastic process for the state vector of species $J$ and
\begin{equation*}
  \vQ^{\vN(t)}(t) = \paren{\vQ^{N_1(t)}(t),\dots,\vQ^{N_J(t)}(t)} \subset \R^{d N(t)}
\end{equation*}
the stochastic process for the state vector of all particles. Possible values for each of these are given by $q^{n_j}_i$, $\vqnj$ and $\vqn$ respectively.

With these definitions we denote by $\Pn(\vqn,t)$ the probability density that $\vN(t) = \vn$ with $\vQ^{\vN(t)}(t) = \vqn$, and define the Fock space probability vector $\vP(t) = \{\Pn(\vqn,t)\}_{\vn}$. Note that because particles of the same species are assumed to be identical, $\Pn(\vqn,t)$ is symmetric with respect to reorderings of the particle positions of type $j$ within $\vqnj$. As in \cite{IMS:2022}, we consider $\vP(t)$ for $t$ fixed to be an element of an
$L^2$ Fock space, $F = L^2(X)$, where
\begin{equation*}
  X = \bigoplus_{\vn} \R^{d \abs{\vn}}.
\end{equation*}
Note that Assumption~\ref{Assume:kernalBdd} implies $N(t)$ is bounded from above, and hence $X$ is finite dimensional. Letting $\vG = \{G^{\vn}\}_{\vn} \in F$ and $\tilde{\vec{G}} = \{\tilde{G}^{\vn}\}_{\vn} \in F$, we construct an inner product on $F$ by
\begin{equation*}
  \la \vG, \tilde{\vec{G}} \ra_F = \sum_{\vn} \frac{1}{\vn !}
  \int_{\R^{d \abs{\vn}}} G^{\vn}(\vqn) \tilde{G}^{\vn}(\vqn) \, d\vqn,
\end{equation*}
where $\vn!$ is defined in the usual multi-index sense.

The forward equation for $\vP(t)$ is then given in strong form by
\begin{equation} \label{eq:FockSpacePDE}
  \frac{d}{dt}\vP(t) = (L + R) \vP(t),
\end{equation}
where $L$ denotes a diffusion operator and $R$ a reaction operator. The former is defined by
\begin{equation*}
  (L \vP)^{\vn}(\vqn,t) = \sum_{j=1}^J D_j \lap_{\vqnj} \Pn(\vqn,t),
\end{equation*}
where $\lap_{\vqnj}$ denotes the Laplacian in $\vqnj$.

To define the reaction operator we need to introduce notation for adding and removing particles from a state $\vqn$, which requires notation for configurations of particles that could have been substrates or products of one reaction. Abusing notation, we let $\Ol(\vn) := \Ol(\nu)$, see Definition~\ref{def:effSamplingSpace}, label the collection of allowable reactant particle indices when there are $\vn$ particles in the system, and note that the number of elements in this set is given by
\begin{equation*}
  \abs{\Ol(\vn)} = \begin{pmatrix}
    \vn \\
    \vec{\alpha}^{(\ell)}
  \end{pmatrix}.
\end{equation*}
We define the allowable product index sampling space $\Otl(\vn) \subset \mathbb{J}^{(\ell)}$ by
\begin{equation*}
  \Otl(\vn) = \begin{cases}
  \varnothing, & \abs{\vec{\beta}^{(\ell)}}=0,\\
  \{ \vec{i} = i_{1}^{(j)} \in \mathbb{J}^{(\ell)} \, | \, i_1^{(j)} \leq n_j \}, & \abs{\vec{\beta}^{(\ell)}}=\beta_{\ell j} = 1,\\
  \{ \vec{i} = (i_{1}^{(j)},i_{2}^{(j)}) \in \mathbb{J}^{(\ell)} \, | \,   i_1^{(j)} \neq i_{2}^{(j)}, i_k^{(j)} \leq n_j, k=1,2\}, & \abs{\vec{\beta}^{(\ell)}}=\beta_{\ell j} = 2,\\
  \{ \vec{i} = (i_{1}^{(j)},i_{1}^{(k)})\in \mathbb{J}^{(\ell)} \, | \, i_1^{(j)} \leq n_j, i_1^{(k)} \leq n_k\}, & \abs{\vec{\beta}^{(\ell)}}=2, \quad \beta_{\ell j} = \beta_{\ell k} = 1, \quad j < k.
  \end{cases}
\end{equation*}
Note, this is a slightly different definition than that of $\Ol(\vn)$, needed due to our choice of normalization for $m_{\ell}\paren{\vy \given \vx}$. The number of elements in $\Otl(\vn)$ is
\begin{equation*}
  \abs{\Otl(\vn)} = \frac{\vn!}{(\vn - \vec{\beta}^{(\ell)})!}.
\end{equation*}

Given the current state of a system, $\vqn$, we will write $\vqn_{\bi}$ for $\bi \in \Ol(\vn)$ to label the position of the substrates for the $\ell$th reaction determined by $\bi$. Similarly, for $\bj \in \Otl(\vn)$, $\vqn_{\bj}$ will denote one product collection of particles in $\vqn$ that could have been produced through one occurrence of reaction $\ell$. To represent the state vector with particles $\vqn_{\bi}$ removed from $\vqn$ we use the notation $\vqn \setminus \vqn_{\bi}$. We analogously use the notation $\vqn \cup \vx$ to denote adding the particles in $\vx$ into the state vector $\vqn$. Finally, we define the overall reaction interaction function, encoding both the rate of a reaction and the placement of products, by
\begin{equation*}
  \Klg \paren{\vy \given \vx} = \Klg(\vx) m_{\ell}^{\eta}\paren{\vy \given \vx},
\end{equation*}
and note that
\begin{equation*}
  \int_{\mathbb{Y}^{(\ell)}} \Klg \paren{\vy \given \vx} d\vy = \Klg(\vx)
\end{equation*}
by assumption. With these notations, the reaction operator is given by
\begin{equation*}
  (R\vP)^{\vn}(\vqn,t) = \sum_{\ell} \bigg[-\smashoperator{\sum_{\bi \in \Ol(\vn)}}\Klg(\vqn_{\bi}) P^{\vn}(\vqn\!\!,t) + \frac{1}{\vec{\alpha}^{(\ell)}!}\!\!\sum_{\bi \in \Otl(\vn)}\!\!\int_{\tilde{\mathbb{X}}^{(\ell)}} \Klg\paren{\vqn_{\bi} \given \vx} P^{\vn - \vupl}(\vqn \setminus \vqn_{\bi} \cup \vx,t) \, d\vx\bigg],
\end{equation*}
where $\vupl = \vec{\beta}^{(\ell)} - \vec{\alpha}^{(\ell)}$ is the net stoichiometry vector for the $\ell$th reaction.

In the special case of the $A + B \leftrightarrows C$ reaction, we proved in~\cite{IMS:2022} that $\vP(t)$ is non-negative, normalized in time, i.e.
\begin{equation*}
  \sum_{\vn} \frac{1}{\vn!} \int_{\R^{d \abs{\vn}}} \Pn(\vqn,t) \, d\vqn = 1
\end{equation*}
and that the mild form of~\eqref{eq:FockSpacePDE},
\begin{equation} \label{eq:ForwardEq}
  \vP(t) = e^{L t} \vP(0) + \int_0^t e^{L(t-s)} (R\vP)(s) \, ds,
\end{equation}
has a unique solution in $C(\brac{0,T},H^2(X))$ if $\vP(0) \in H^1(X)$. Let $D(L)$ denote the domain of $L$. Assuming that $\vP(0) \in D(L) \cap H^1(x)$, we have that this mild solution is actually a strong solution to~\eqref{eq:FockSpacePDE} in $C(\brac{0,T},D(L)\cap H^2(X)) \cap C^1(\brac{0,T},L^2(X))$, see Proposition 4.3.9 of~\cite{CH:1998}. We do not show here, but subsequently assume, that these properties still hold for the general class of reaction systems considered in this work under the assumptions of our main result, Theorem~\ref{T:FluctuationsThm}. Recall this includes an assumption that the total number of molecules is uniformly bounded from above, which rules out finite-time blowup.


In the calculations that follow, we assume that we are working with a function $f \in C^{2}(\R^d)$ for which the involved quantities are finite and well defined for the corresponding chemical reaction network. More specifically, the case needed for establishing~\eqref{Eq:Expectation_Mean_Formula} in the context of Lemma~\ref{L:MomentBound} is $f(x) = \abs{x}^{8D}$, for which Assumption~\ref{A:AssumptionInitialCondition} assumes finiteness. To derive~\eqref{Eq:Expectation_Mean_Formula}, we begin by noting that for a given scalar test function $f \in C^{2}(\R^d)$, we have that 
\begin{equation*}
  \E \la f, \gamma \mu_t^j \ra = \E\brac{\sum_{i=1}^{N_j(t)} f\paren{Q_i^{N_j(t)}(t)}} = \la \vG, \vP(t) \ra_F,
\end{equation*}
where
\begin{equation*}
  G^{\vn}(\vqn) = \sum_{i=1}^{n_j} f\paren{q_i^{n_j}}.
\end{equation*}

From~\eqref{eq:FockSpacePDE} we then have that
\begin{equation*}
  \la \vG, \vP(t) \ra_F = \la \vG, \vP(0) \ra_F + \int_0^t \la \vG, (L + R)\vP(t) \ra_F.
\end{equation*}
We note that $L$ is self-adjoint, see~\cite{IMS:2022}, so that
\begin{equation} \label{eq:GLPinnerprod}
\begin{aligned}
  \la \vG, L\vP(t) \ra_F &= \la L \vG, \vP(t) \ra_F \\
  &=\sum_{\vn} \frac{1}{\vn !} \int_{\R^{d \abs{\vn}}} \sum_{i=1}^{n_j} D_j \lap_{q_i^{n_j}} f(q_i^{n_j}) \Pn(\vqn,t) \, d\vqn \\
  &= \E \brac{\sum_{i=1}^{N_j(t)} D_j (\lap f)\paren{Q_i^{N_j(t)}}(t)} \\
  &= \E \la (\mathcal{L}_j f)(x), \gamma \mu_t^j \ra.
\end{aligned}
\end{equation}

In considering the reaction operator, let $\bi_1 \in \Otl(\vn)$ denote one fixed configuration of product particle indices, so that the symmetry of $G^{\vn}$, $\Klg$, and $\Pn$ with respect to permutations in orderings of particles of the same species implies
\begin{align*}
  \sum_{\vn} &\frac{1}{\vn !} \int_{\R^{d \abs{\vn}}} G^{\vn}(\vqn) \brac{\frac{1}{\vec{\alpha}^{(\ell)}!}\sum_{\bi \in \Otl(\vn)} \int_{\tilde{\mathbb{X}}^{(\ell)}} \Klg\paren{\vqn_{\bi} \given \vx} P^{\vn - \vupl}(\vqn \setminus \vqn_{\bi} \cup \vx,t) \, d\vx} \, d\vqn \\
  &=  \sum_{\vn} \frac{\abs{\Otl(\vn)}}{\vn! \, \vec{\alpha}^{(\ell)}!} \int_{\R^{d \abs{\vn}}} \int_{\tilde{\mathbb{X}}^{(\ell)}} G^{\vn}(\vqn) \Klg\paren{\vqn_{\bi_1} \given \vx} P^{\vn - \vupl}(\vqn \setminus \vqn_{\bi_1} \cup \vx,t) \, d\vx \, d\vqn \\
  &=  \sum_{\vn} \sum_{\bi \in \Ol(\vn-\vupl)}  \tfrac{\abs{\Otl(\vn)}}{\vn!\,\vec{\alpha}^{(\ell)}!\, \abs{\Ol(\vn - \vupl)}} \\
  &\phantom{=} \quad \times \int_{\R^{d \abs{\vn - \vupl}}}
  \int_{\mathbb{Y}^{(\ell)}} G^{\vn}\!\paren{\vq^{\vn - \vupl}\!\setminus\!\vq_{\bi}^{\vn - \vupl}\!\cup\!\vy} \Klg\!\paren{\vy \given \vq_{\bi}^{\vn-\vupl}} P^{\vn - \vupl}\!\!\paren{\vq^{\vn - \vupl},t} \, d\vy \, d\vq^{\vn-\vupl}, \\
  &=: I.
\end{align*}
Noting that
\begin{align*}
  \abs{\Ol(\vn-\vupl)} = \begin{pmatrix}
    \vn-\vupl \\
    \vec{\alpha}^{\ell}
  \end{pmatrix},
\end{align*}
we find
\begin{align*}
  I &=
  \begin{multlined}[t]
    \sum_{\vn}\!\! \sum_{\bi \in \Ol(\vn-\vupl)}  \frac{1}{\paren{\vn - \vupl}!} \int_{\R^{d \abs{\vn-\vupl}}}\! \int_{\mathbb{Y}^{(\ell)}} \!G^{\vn}\!\paren{\vq^{\vn - \vupl}\!\setminus\!\vq_{\bi}^{\vn - \vupl}\!\cup\!\vy}\! \Klg\!\paren{\vy \given \vq_{\bi}^{\vn-\vupl}}\\ \times P^{\vn - \vupl}\!\!\paren{\vq^{\vn - \vupl},t} \, d\vy \, d\vq^{\vn-\vupl},
  \end{multlined}\\
  &= \sum_{\vn} \sum_{\bi \in \Ol(\vn)} \frac{1}{\vn!}
  \int_{\R^{d \abs{\vn}}} \int_{\mathbb{Y}^{(\ell)}} G^{\vn}\!\paren{\vq^{\vn}\setminus\vq_{\bi}^{\vn}\cup\vy} \Klg\!\paren{\vy \given \vq_{\bi}^{\vn}} P^{\vn}\paren{\vq^{\vn},t} \, d\vy \, d\vq^{\vn}.
\end{align*}
Using this identity and the definition of $\Klg\!\paren{\vy \given \vx}$, we find
\begin{equation*}
  \begin{aligned}
    \la \vG, R\vP(t)\ra &=  \sum_{\vn} \frac{1}{\vn!}
    \int_{\R^{d \abs{\vn}}} \brac{ \sum_{\bi \in \Ol(\vn)} \int_{\mathbb{Y}^{(\ell)}} \brac{G^{\vn}\!\paren{\vq^{\vn}\setminus\vq_{\bi}^{\vn}\cup\vy} - G^{\vn}\!\paren{\vqn}} \Klg\!\paren{\vy \given \vq_{\bi}^{\vn}} \, d\vy} P^{\vn}\paren{\vq^{\vn},t} \, d\vq^{\vn}.
  \end{aligned}
\end{equation*}
For our specific choice of $\vG$,
\begin{equation*}
  G^{\vn}\!\paren{\vq^{\vn}\setminus\vq_{\bi}^{\vn}\cup\vy} - G^{\vn}\!\paren{\vqn}
  = \sum_{r=1}^{\beta_{\ell j}} f(y_r^{(j)}) - \sum_{r=1}^{\alpha_{\ell j}} f\paren{q_{i_r^{(j)}}^{n_j}}
\end{equation*}
so that
\begin{equation} \label{eq:GRPinnerprod}
  \begin{aligned}
    \la \vG, R\vP(t)\ra &=  \E \brac{ \sum_{\bi \in \Ol(\vn)} \int_{\mathbb{Y}^{(\ell)}} \brac{\sum_{r=1}^{\beta_{\ell j}} f(y_r^{(j)}) - \sum_{r=1}^{\alpha_{\ell j}} f\paren{Q_{i_r^{(j)}}^{N_j(t)}(t)}} \Klg\!\paren{\vy \given \vQ_{\bi}^{\vN(t)}(t)} \, d\vy} \\
    &= \E \brac{\frac{1}{\vec{\alpha}^{(\ell)}!} \int_{\tilde{\mathbb{X}}^{(\ell)}}\int_{\mathbb{Y}^{(\ell)}} \brac{\sum_{r=1}^{\beta_{\ell j}} f(y_r^{(j)}) - \sum_{r=1}^{\alpha_{\ell j}} f\paren{x_{r}^{(j)}}} \Klg\!\paren{\vy \given \vx} \lambda^{(\ell)}\brac{\gamma \mu_{t^{-}}}\, d\vy \, d\vx} \\
    &= \E \brac{\frac{\gamma}{\vec{\alpha}^{(\ell)}!} \int_{\tilde{\mathbb{X}}^{(\ell)}}\int_{\mathbb{Y}^{(\ell)}} \brac{\sum_{r=1}^{\beta_{\ell j}} f(y_r^{(j)}) - \sum_{r=1}^{\alpha_{\ell j}} f\paren{x_{r}^{(j)}}} K_{\ell}\!\paren{\vy \given \vx} \lambda^{(\ell)}\brac{\mu_{t^-}}\, d\vy \, d\vx}.
  \end{aligned}
\end{equation}
Combining~\eqref{eq:GRPinnerprod} with~\eqref{eq:GLPinnerprod}, and integrating the Fock space forward equation, we obtain~\eqref{Eq:Expectation_Mean_Formula}.
\section{Auxiliary lemmas}


\begin{lemma}\label{L:AuxiliaryBound1}
Let $a\geq 1$ and $J\geq 0$. Then, for any $\Psi\in W^{J+2,a}_{0}$, there is constant $C<\infty$, that may depend on $D_{j}$ and $a$ but not on $\Psi$, such that
$$ \la \mathcal{L}_j \Psi, \Psi\ra_{J,a}\leq C||\Psi||^2_{J,a}$$
\end{lemma}
\begin{proof}[Proof of Lemma \ref{L:AuxiliaryBound1}]
For notational convenience we set, without any loss of generality, $D_{j}=1$.

Let us first consider the case that $J=0$. By definition of the norm in question and integration by parts we have
\begin{align*}
\la \mathcal{L}_j \Psi, \Psi\ra_{0,a}&=\int_{\mathbb{R}^{d}}\frac{1}{1+|x|^{2a}}
\Delta\Psi(x)\Psi(x)dx\nonumber\\
&=-\int_{\mathbb{R}^{d}}\frac{1}{1+|x|^{2a}}|\nabla \Psi(x)|^{2}dx-
\int_{\mathbb{R}^{d}}
\nabla\left(\frac{1}{1+|x|^{2a}}
\right)
\cdot\nabla \Psi(x)\Psi(x)dx\nonumber
\end{align*}
Integration by parts of the last expression gives
\begin{align*}
-\int_{\mathbb{R}^{d}}
\nabla\left(\frac{1}{1+|x|^{2a}}
\right)
\cdot\nabla \Psi(x)\Psi(x)dx&=\int_{\mathbb{R}^{d}}
\nabla\left(\frac{1}{1+|x|^{2a}}
\right)
\cdot\nabla \Psi(x)\Psi(x)dx+\int_{\mathbb{R}^{d}}\Delta\left(\frac{1}{1+|x|^{2a}}
\right)|\Psi(x)|^2 dx
\end{align*}
so that
\begin{align*}
\la \mathcal{L}_j \Psi, \Psi\ra_{0,a}&=\int_{\mathbb{R}^{d}}\frac{1}{1+|x|^{2a}}
\Delta\Psi(x)\Psi(x)dx\nonumber\\
&=-\int_{\mathbb{R}^{d}}\frac{1}{1+|x|^{2a}}|\nabla \Psi(x)|^{2}dx+
\frac{1}{2}\int_{\mathbb{R}^{d}}\Delta\left(\frac{1}{1+|x|^{2a}}
\right)|\Psi(x)|^2 dx
\nonumber
\end{align*}
Direct algebra then shows that
\begin{align}
\Delta\frac{1}{1+|x|^{2a}}&=\frac{2a(2-d-2a) |x|^{2(a-1)}(1+|x|^{2a})+8a^2|x|^{4a-2}}
{(1+|x|^{2a})^{3}}
\end{align}
so that
\begin{align*}
\frac{1}{2}\int_{\mathbb{R}^{d}}\Delta\left(\frac{1}{1+|x|^{2a}}
\right)|\Psi(x)|^2 dx&=\int_{\mathbb{R}^{d}}
\frac{a(2-d-2a) |x|^{2(a-1)}(1+|x|^{2a})+4a^2|x|^{4a-2}}
{(1+|x|^{2a})^{3}}|\Psi(x)|^2 dx\nonumber\\
&\leq \int_{\mathbb{R}^{d}}
\frac{4a^2|x|^{4a-2}}
{(1+|x|^{2a})^{2}}\frac{1}{1+|x|^{2a}}|\Psi(x)|^2 dx.
\end{align*}
Analyzing now the latter integral separately for the regions $|x|<1$ and $|x|\geq 1$, and using the assumption $a\geq1$, we obtain that there is a constant $C<\infty$ that depends on $a$ such that
\begin{align*}
\int_{\mathbb{R}^{d}}
\frac{4a^2|x|^{4a-2}}
{(1+|x|^{2a})^{2}}\frac{1}{1+|x|^{2a}}|\Psi(x)|^2 dx&\leq C||\Psi||^2_{0,a}.
\end{align*}
The latter then readily gives that
\begin{align*}
\la \mathcal{L}_j \Psi, \Psi\ra_{0,a}&=\int_{\mathbb{R}^{d}}\frac{1}{1+|x|^{2a}}\Delta\Psi(x)\Psi(x)dx\nonumber\\
&\leq-\int_{\mathbb{R}^{d}}\frac{1}{1+|x|^{2a}}|\nabla\Psi(x)|^{2}dx+C||\Psi||^2_{0,a}\nonumber\\
&\leq C||\Psi||^2_{0,a},
\end{align*}
completing the proof for $J=0$. The computation for $J>0$ follows along exactly the same lines.

\end{proof}

\end{document}